\documentclass[13pt]{amsart}
\usepackage{amssymb}
\usepackage{amsfonts} 
\usepackage{latexsym}   
\usepackage{amssymb}    
\usepackage{amsmath}    
\usepackage{amsbsy}
\usepackage{amsgen}
\usepackage{amsfonts}
\usepackage{array}
\usepackage{epsfig}
\usepackage{psfrag}
\usepackage[all]{xy}
\usepackage{hyperref}

\usepackage{xcolor}
\usepackage{mathabx}
\usepackage{graphicx}

\providecommand{\U}[1]{\protect\rule{.1in}{.1in}}

\newcommand{\R}{\mathbb{R}}

\theoremstyle{plain}

\newtheorem{algorithm}{Algorithm}[section]
\newtheorem{thm}{Thm}

\newtheorem{corollary}[algorithm]{Corollary}

\newtheorem{definition}[algorithm]{Definition}
\newtheorem{example}[algorithm]{Example}

\newtheorem{lemma}[algorithm]{Lemma}

\newtheorem{theorem} [algorithm] {Theorem}
\newtheorem{theoremlet}[thm]{Theorem}
\newtheorem{theoremlet'}[thm]{Theorem$'$}

\usepackage{amssymb}
\usepackage{amsfonts} 
\usepackage{latexsym}   
\usepackage{amssymb}    
\usepackage{amsmath}    
\usepackage{amsbsy}
\usepackage{amsgen}
\usepackage{amsfonts}
\usepackage{array}
\usepackage{scalerel}

\usepackage{epsfig}
\usepackage{psfrag}
\usepackage[all]{xy}
\usepackage{amssymb}
\usepackage{mathabx}

\newtheorem{proposition}[algorithm]{Proposition}

\def\bdm{\begin{displaymath}}
\def\edm{\end{displaymath}}
\def\beq{\begin{equation}}
\def\eeq{\end{equation}}
\def\bes{\begin{equation*}}
\def\ees{\end{equation*}}
\def\epcm{\end{picture}\end{center}\end{minipage}}
\def\bpcm{\begin{minipage}{80pt}\begin{center}\begin{picture}}

\def\t2{T^2}

\def\f4{F_4}
\def\g2{G_2}

\def\dim{\textrm{dim}}

 \numberwithin{equation}{section}
  \numberwithin{figure}{section}

\makeatletter
\let\above\@@above
\let\abovewithdelims\@@abovewithdelims
\let\over\@@over
\let\overwithdelims\@@overwithdelims
\let\atop\@@atop
\let\atopwithdelims\@@atopwithdelims

\usepackage{cite}
\usepackage{multibib}
\usepackage{lscape}

\usepackage{eucal}

\let\mathcal=\relax
\usepackage{mathrsfs}

\usepackage{hyperref}

\usepackage{fullpage}
\usepackage[all]{xy}
\usepackage{color}
\usepackage{graphicx}
\usepackage{tikz}
\usepackage{setspace}
\usetikzlibrary{matrix,arrows}

\makeatletter
\let\temp@index\index
\def\index{\ifmmode\mathop{\rm index}\nolimits\else\temp@index\fi}

\def\cal{\fam2}

\begingroup\ifx\SetFigFont\undefined%
\gdef\SetFigFont#1#2#3#4#5{%
  \reset@font\fontsize{#1}{#2pt}%
  \fontfamily{#3}\fontseries{#4}\fontshape{#5}%
  \selectfont}%
\fi\endgroup%
\usepackage[paperwidth = 8.5in,paperheight=11in,
vmarginratio = 1:1,hmarginratio = 1:1,margin = 1in,footskip=0.5in]{geometry}

\def\Grass{\mathop{\rm Gr}\nolimits}

\def\RiemScal{\mathop{\cal R}\nolimits^+}

\def\Riem{{\cal R}}
\def\tor{\mathop{\rm torp}\nolimits}
\def\toe{\mathrm{toe}}
\def\boot{\mathrm{boot}}
\def\std{\mathrm{std}}
\def\bend{\mathrm{bend}}
\newcommand{\stret}{\mathrm{stretch}}
\newcommand{\p}{\partial}
\def\R{\mathbb{R}}

\newcommand{\msH}{\mathscr{H}}
\newcommand{\msV}{\mathscr{V}}
\newcommand{\Bb}{{\textcolor{blue}{\bf B}}}
\newcommand{\Ff}{{\textcolor{red}{\bf F}}}

\newcommand{\Diff}{\mathrm{Diff}}

\let\textbar\=
\def\=#1{\ifmmode\bar{#1}\else\textbar#1\fi}

\begingroup\makeatletter\ifx\SetFigFont\undefined%
\gdef\SetFigFont#1#2#3#4#5{%
  \reset@font\fontsize{#1}{#2pt}%
  \fontfamily{#3}\fontseries{#4}\fontshape{#5}%
  \selectfont}%
\fi\endgroup%

\begin{document}
\newcommand{\comment}[1]{\vspace{5 mm}\par \noindent
\marginpar{\textsc{Note}}
\framebox{\begin{minipage}[c]{0.95 \textwidth}
#1 \end{minipage}}\vspace{5 mm}\par}

\title[Positive ${\bf (p, n)}$-intermediate scalar curvature and cobordism]{Positive  $(p, n)$-intermediate scalar curvature and cobordism}

\author[Burkemper]{Matthew Burkemper}
\address[Burkemper]{Department of Mathematics, Statistics, and Physics, Wichita State University, Wichita, Kansas}
\email{burkemper@math.wichita.edu}

\author[Searle]{Catherine Searle}
\address[Searle]{Department of Mathematics, Statistics, and Physics, Wichita State University, Wichita, Kansas}
\email{searle@math.wichita.edu}

\author[Walsh]{Mark Walsh}
\address[Walsh]{Department of Mathematics, National University of Maynooth, Ireland}
\email{Mark.Walsh@mu.ie}

\subjclass[2000]{Primary: 53C20; Secondary: 57S25} 

\date{\today}

%

\begin{abstract} 

In this paper we consider a well-known construction due to Gromov and Lawson, Schoen and Yau, Gajer, and Walsh which allows for the extension of a metric of positive scalar curvature over the trace of a surgery in codimension at least $3$ to a metric of positive scalar curvature which is a product near the boundary. We generalize this construction to work for $(p,n)$-intermediate scalar curvature for $0\leq p\leq n-2$ for surgeries in codimension at least $p+3$. We then use it to generalize a well known theorem of Carr. Letting $\Riem^{s_{p,n}>0}(M)$ denote the space of positive $(p, n)$-intermediate scalar curvature metrics on an $n$-manifold $M$, we show for $0\leq p\leq 2n-3$ and $n\geq 2$, that for a closed, spin, $(4n-1)$-manifold $M$ admitting a metric of positive $(p,4n-1)$-intermediate scalar curvature,
$\Riem^{s_{p,4n-1}>0}(M)$ has infinitely many path components.
\end{abstract}
\maketitle

\section{Introduction}

A central question in modern geometry is the following: given a smooth $n$-dimensional manifold $M$ and a preferred curvature condition $C$, can we specify a Riemannian metric $g$ on $M$ so that $g$ satisfies $C$? The scalar curvature is the weakest invariant of the curvature tensor and so, unsurprisingly, the greatest success has been achieved in classifying which manifolds admit metrics of positive scalar curvature (psc-metrics). By contrast, the problem of classifying which manifolds admit metrics whose sectional curvature is strictly positive, or finding new examples of manifolds admitting such metrics, is formidably hard. 

Such classification problems require obstructive tools for ruling out certain manifolds from consideration and constructive tools for building metrics in the case no obstructions exist. Despite its relative weakness as a curvature constraint, there exist many smooth manifolds which do not admit metrics of positive scalar curvature. For example, work of Schr{\"o}dinger and Lichnerowicz \cite{lichnerowicz} and Hitchin \cite{Hit}, shows that if $M$ is a closed, spin manifold of dimension $n$ admitting a psc-metric, a certain invariant $\alpha(M)\in KO^{-n}$, representing the index of the Dirac operator and generalizing the $\hat{A}$-genus, must vanish. This obstructive tool is complemented by a powerful constructive result due to Gromov and Lawson \cite{gromov-lawson} and Schoen and Yau \cite{schoen-yau}: suppose $M$ and $M'$ are smooth manifolds and $M'$ is obtained from $M$ via surgery in codimension at least three, then any psc-metric $g$ on $M$ can be used to construct a psc-metric $g'$ on $M'$. Combining these respective obstructive and constructive tools led to considerable progress in classifying which manifolds admitted psc-metrics. In particular, Stolz \cite{stolz} showed that a closed, smooth, simply connected manifold $M$ admits a psc-metric if and only if $M$ is either non-spin, or $M$ is spin and $\alpha(M)=0$.

The surgery technique breaks down when one attempts to apply it to stronger notions like positive Ricci or sectional curvatures. Indeed, not only does the construction not work for positive Ricci or positive sectional curvature, there are topological obstructions to the existence of metrics of positive Ricci curvature on manifolds arising from surgeries which are admissible in the positive scalar curvature setting.

It is thus natural to consider intermediate curvatures interpolating between  scalar and Ricci, and scalar and sectional curvatures. These intermediate curvatures are the $k$-positive Ricci curvature, where $k\in\{1,2,\cdots, n\}$, defined by Wolfson in \cite{Wolfson}, and the $(p,n)$-intermediate scalar curvatures (originally, the $p$-curvature), where $p\in\{0,1,\cdots, n-2\}$ defined by Labbi in \cite{labbi-surgery}; see Definition \ref{spn}. Analogs of the Surgery Theorem for these respective curvatures have been proven in \cite{Wolfson} and \cite{labbi-surgery}. In particular, it is shown in \cite{labbi-surgery} that if $M$ and $M'$ are smooth $n$-dimensional manifolds and $M'$ is obtained from $M$ via surgery in codimension at least $p+3$, then any metric, $g$, on $M$ admitting positive $(p,n)$-intermediate scalar curvature can be used to construct a metric, $g'$, on $M'$ which also admits positive $(p,n)$-intermediate scalar curvature. This is a generalization of the Surgery Theorem of \cite{gromov-lawson} and \cite{schoen-yau} stated above, where the scalar curvature corresponds to the case when $p=0$.

In this paper, we focus our attention on the $(p,n)$-intermediate scalar curvature.  Our main result, Theorem \ref{A}, concerns an analogous strengthening of a more general positive scalar curvature construction. The surgery technique of \cite{gromov-lawson} and \cite{schoen-yau} can be enhanced to give rise to a psc-metric on the trace of a surgery in codimension $\geq 3$ by work of Gajer \cite{gajer} and Walsh \cite{walsh1}. In particular, the resulting metric satisfies a Riemannian product structure on a collar neighborhood of the boundary. In Theorem \ref{A}, we extend the result of \cite{gajer} and \cite{walsh1} from the case when $p=0$ to more general $p$ as follows. 
\\

\begin{theoremlet}\label{A} Let $M$ be a smooth $n$-manifold, $\phi:S^{k}\times D^{n-k}\to M$, a smooth embedding, and $\{\bar{M}_{\phi}; M, M_{\phi}\}$, the trace of the surgery on $\phi$. Suppose furthermore that $n-k\geq 3$ and $p\in\{0, 1, \cdots, n-k-3\}$. Then for any metric $g$ on $M$ with positive $(p,n)$-intermediate scalar curvature, there are metrics $g_{\phi}$ on $M_{\phi}$ and $\bar{g}_{\phi}$ on $\bar{M}_{\phi}$ satisfying:
\begin{enumerate}
\item The metrics $g_{\phi}$ and $\bar{g}_{\phi}$ have respectively positive $(p,n)$ and $(p,n+1)$-intermediate scalar curvatures on $M_{\phi}$ and $\bar{M}_{\phi}$; and
\item Near the boundary components $M$ and $M_{\phi}$, $\bar{g}_{\phi}$ takes the form of respective products $\bar{g}_{\phi}=g+dt^{2}$ and $\bar{g}_{\phi}=g_{\phi}+dt^{2}$.
\end{enumerate}\end{theoremlet}

An important application of this theorem is in exhibiting non-triviality in the topology of the space of Riemannian metrics of positive $(p, n)$-intermediate scalar curvature on a smooth manifold. We denote by $\Riem(M)$, the space of all Riemannian metrics on the smooth manifold $M$. This space has a standard $C^{\infty}$ topology; see Chapter 1 of \cite{TW} for a definition. For each $p\in\{0,1,\cdots n-2\}$, we consider the subspace $\Riem^{s_{p,n}>0}(M)$ of Riemannian metrics of positive $(p,n)$-intermediate curvature on $M$.  In the case when $p=0$, the space $\Riem^{s_{0,n}>0}(M)$ denotes the space of psc-metrics on $M$. 

More generally, one may consider, for any curvature condition $C$, the subspace $\Riem^{C}(M)\subset \Riem(M)$ of Riemannian metrics which satisfy $C$. In recent years, there has been substantial interest in understanding the topology of the space, $\Riem^{C}(M)\subset \Riem(M)$, for a variety of manifolds $M$ and curvature conditions $C$. Much of this has also involved the corresponding moduli spaces obtained as a quotient of $\Riem^{C}(M)$ by the action of appropriate subgroups of the group of self-diffeomorphisms of $M$, $\Diff (M)$. Recall that $\Diff (M)$ acts on $\Riem^{C}(M)$ by means of pulling back metrics. The most progress has occurred in the case when $C$ denotes positive scalar curvature; see for example results due to Botvinnik, Ebert, and Randall-Williams \cite{BERW}, Botvinnik, Hanke, Schick, and Walsh \cite{BHSW}, Coda-Marquez \cite{Coda}, Crowley and Schick \cite{CS}, Ebert and Randall-Williams \cite{ERW1} and \cite{ERW2}, Ebert and Wiemeler \cite{EW}, Frenck \cite{Frenck}, Hanke,  Schick, and Stiemle  \cite{HSS}, and  Walsh \cite{walshloops}. There are numerous results for other curvature conditions such as negative sectional curvature or positive Ricci curvature, see, for example, Tuschmann and Wraith \cite{TW}.

Theorem \ref{A} can be used to exhibit non-triviality in the topology of this space for many manifolds and many $p\geq 0$ by strengthening existing results for positive scalar curvature $(p=0)$. We will not provide a comprehensive account of this here but rather, in Theorem \ref{B}, an example which illustrates this point.
\\
 
\begin{theoremlet}\label{B}
Let $M$ be a smooth closed spin manifold of dimension $4n-1, n\geq 2$, which admits an $s_{p,4n-1}>0$ curvature metric for some $p\in\{0,1,\cdots, 2n-3\}$. Then the space $\Riem^{s_{p,4n-1}>0}(M)$ has infinitely many path components.
\end{theoremlet}

Theorem \ref{B} generalizes Theorem 4 of Carr \cite{carr}, for positive scalar curvature. In particular, Theorem 4 of \cite{carr} is  the $p=0$ case of Theorem \ref{B} when $M=S^{4n-1}$ for $n\geq 2$. Note that extending the theorem from the case when $M=S^{4n-1}$ to an arbitrary closed, spin manifold admitting psc-metrics is not difficult. Indeed, it follows as an immediate corollary of the main theorem of  \cite{EW}. The main work is in dealing with the spherical case.

Note that the theorem of \cite{carr} was generalized to positive Ricci curvature by Wraith in \cite{Wraith} for $M$, a homotopy $(4n-1)$-sphere bounding a parallelizable manifold. This result was achieved by showing that the moduli space of $\Riem^{\mathrm{Ric}>0}(M)$ contains infinitely many path-components.
However, there is no clear implication relating positive $(p, n)$-intermediate scalar curvature to positive Ricci curvature. In fact, the methods used to prove Theorem \ref{B} are quite different from those used in \cite{Wraith}, where the Kreck-Stolz $s$-invariant is used to distinguish between the path components of the moduli space of $\Riem^{\mathrm{Ric}>0}(M)$.

During the writing of this paper we discovered that Theorem \ref{B} also follows as a case of a recent theorem of Frenck and Kordass \cite{FK}. Their work extends some powerful techniques of \cite{BERW}, for positive scalar curvature to the cases of positive $(p, n)$-intermediate scalar curvature and also positive $k$-Ricci curvature; see Theorems \ref{A} and \ref{B} of \cite{FK}. Our work has independent value, as  Theorem \ref{A} above provides a geometrically explicit construction of an $s_{p,n}>0$ metric over the trace of an appropriate surgery, something which is not done in \cite{FK}. In doing so, we also provide detailed curvature calculations for sectional and intermediate scalar curvature of various warped product metrics that  are of value in their own right.  Likewise, our proof of Theorem \ref{B} differs from theirs in that ours gives an explicit construction for representative elements of distinct path components of $\Riem^{s_{p,4n-1}>0}(M)$.

\subsection{Organization} The paper is organized as follows. In Section \ref{s2} we establish preliminaries. In Section \ref{s3} we establish an isotopy-concordance result for positive $(p, n)$-intermediate scalar curvature. In Section \ref{s4}, we determine the intermediate scalar curvature of a warped product metric. In Section \ref{s5}, we apply these calculations to the standard metrics on the sphere and the disk. In Section \ref{s6}, we prove Theorem \ref{A} and in Section \ref{s7}, we prove Theorem \ref{B}.

\subsection*{Acknowledgements}

The authors wish to thank Boris Botvinnik and David Wraith for helpful conversations.
M. Burkemper and C. Searle gratefully acknowledge  partial support by C. Searle's NSF Grant \#DMS-1906404.  C. Searle and M. Walsh   thank the Department of Mathematics at the National University of Ireland at Maynooth and  the Department of Mathematics at Wichita State University, respectively,  for their hospitality during visits where a portion of this work was completed.

\section{Preliminaries}\label{s2}
\subsection{$(p, n)$-Curvature}

We consider a generalization due to Labbi \cite{labbi-surgery} of the sectional and scalar curvatures which we call the $(p,n)$-intermediate scalar curvature and denote $s_{p,n}$. This was originally referred to as the $p$-curvature, $s_p$. However, we will often deal with cylinders and more general manifolds with boundary, where dimensions $n$ and $n-1$ arise in tandem. Thus, we adopt this term to aid the reader in distinguishing between the $p$-curvature on the ambient $n$-manifold and the $p$-curvature on an embedded submanifold or boundary component of dimension $n-1$.

Let $M$ be a smooth $n$-dimensional manifold, possibly with non-empty boundary. For any $x\in M$, $k\in\{0,1,\cdots, n\}$, we denote by $\Grass_k(T_x M)$, the Grassmann manifold of $k$-dimensional subspaces of the tangent space $T_x M$ and by $\Grass_k(M)$, the corresponding Grassmann bundle of $k$-dimensional subspaces obtained as the union of $\Grass_k(T_x M)$ over $x\in M$. We now define the {\em $(p,n)$-intermediate scalar curvature} of a Riemannian metric on $M$ as follows.

\begin{definition}\label{spn}
{\em Let $(M,g)$ be an $n$-dimensional Riemannian manifold (with possibly non-empty boundary) and let $p\in\{0,1,\cdots, n-2\}$. The {\em $(p,n)$-intermediate scalar curvature} of $M$ is the function $s_{p,n}:\Grass_p(M)\to\R$ defined for $x\in M$, $P$ a $p$-plane in $T_xM$ and $\{e_1, \cdots, e_{n-p}\}$, an orthonormal basis of the orthogonal complement $P^\perp$ of $P$ in $T_xM$, by 
$$s_{p,n}(x,P) := \sum_{i,j} K_x(e_i,e_j),$$
where $K_x(e_i, e_j)$ is the sectional curvature at $x$ of the subspace of $T_{x}M$ spanned by the vectors $e_i$ and $e_j$.}
\end{definition}

It follows that $s_{p,n}(x,P)$ is the scalar curvature at $x$ of the locally specified $n-p$-dimensional submanifold of $M$ given by restricting the exponential map of $g$ at $x$ to the subspace $P^{\perp}\subset T_{x}M$. In particular, it is well defined for any choice of orthonormal basis $\{e_1, \cdots, e_{n-p}\}$ for $P^{\perp}$. When $p=0$, $P^{\perp}= T_{x}M$ and so $s_{0,n}(x):=s_{0, n}(x, \bar{0})$ is precisely the scalar curvature of the Riemannian manifold $(M,g)$ at the point $x$. When $\dim P=p=n-2$, $s_{n-2, n}(x,P)$ is twice the sectional curvature at $x$ of the plane $P^{\perp}\subset T_{x}M$ with respect to $(M,g)$. The $(p,n)$-intermediate scalar curvatures for $0\leq p\leq n-2$ are therefore a collection of curvatures interpolating between the scalar curvature, when $p=0$, and (twice the) sectional curvature when $p=n-2$. For any given value of $p<n-2$, the $(p,n)$-intermediate scalar curvature is a trace of the $(p+1,n)$-intermediate scalar curvature. 

Of particular interest to us is the case when the $(p,n)$-intermediate scalar curvature is positive on a manifold. We say that a Riemannian metric $g$ on $M$ is a {\em metric of positive $(p,n)$-intermediate scalar curvature} (an {\em $s_{p,n}>0$ metric}) if for any $x\in M$ and any $P\subset T_x M$, $s_{p,n}(x, P)>0$. It is obvious that if the $(p,n)$-intermediate scalar curvature of $(M,g)$ is positive for $p>0$, then the $(p-1,n)$-intermediate scalar curvature is positive as well. By continuing to take traces, it is clear that this holds  for any integer $0\leq q\leq p$.  We summarize this hierarchy in the following proposition.

\begin{proposition}\label{hierarchyofcurvatures}
If a Riemannian manifold $(M,g)$ has positive $(p,n)$-intermediate scalar curvature, then it has positive $(q,n)$-intermediate scalar curvature for $0\leq q\leq p$.
\end{proposition}

Note that the converse is not true. For any dimension $n$ and any $0\leq p< n-2$, we can construct a Riemannian manifold that has positive $(p,n)$-intermediate scalar curvature, but not positive $(p+1,n)$-intermediate scalar curvature. In fact we only need to look at products of spheres.

\begin{example}\label{curvature of a product of spheres}
{\em Let $M$ be the $n$-dimensional Riemannian product manifold of $m$ standard round spheres of radius one and of dimension at least one. Then $M$ has positive $(p,n)$-intermediate scalar curvature if and only if $p<n-m$.}
\end{example}
\noindent This example can be extended to include other factors with positive sectional curvature.

\subsection{Isotopy and Concordance}

Various notions of isotopy and concordance arise throughout Mathematics. Here, we are chiefly concerned with metrics of positive $(p,n)$-intermediate scalar curvature and we define these notions in this case.

\begin{definition}
{\em Two metrics $g_0, g_1$ on an $n$-dimensional manifold $M$ with positive $(p,n)$-intermediate scalar curvature are said to be {\em $(s_{p,n}>0)$-isotopic} if they are connected by a path $t\mapsto g_{t}$ in the space of positive $(p,n)$-intermediate scalar curvature metrics on $M$, $t\in [0,1]$. The connecting path is called an {\em $(s_{p,n}>0)$-isotopy.}}
\end{definition}
\begin{definition}
{\em The metrics $g_0$ and $g_1$ on $M$ are said to be {\em $(s_{p,n}>0)$-concordant} if, for some $L>0$, there is a metric $\bar{g}$ on the cylinder $M\times [0,L+2]$, of positive $(p,n+1)$-curvature, and satisfying
$$\bar{g}|_{M\times [0,1]}=g_{0}+dt^{2}\hspace{0.3cm}\text{ and }\hspace{0.3cm}  \bar{g}|_{M\times [L+1,L+2]}=g_{1}+dt^{2}.$$ The metric $\bar{g}$ is known as an {\em $(s_{p,n}>0)$-concordance.} }
\end{definition}

We will frequently shorten $(s_{p,n}>0)$-isotopy and $(s_{p,n}>0)$-concordance to just isotopy and concordance. It is straightforward to show that both isotopy and concordance determine equivalence relations on the space of positive $(p,n)$-intermediate scalar curvature metrics on the manifold.

The problem of whether or not a given pair of concordant metrics are in turn isotopic is notoriously difficult and we do not consider it here. The converse problem however is much more tractable. It has long been known in the case of metrics of positive scalar curvature ($s_{0,n}>0$), that isotopic metrics are concordant. This indeed holds more generally, as we demonstrate in Proposition \ref{isotoconc}. 

\section{Isotopy Implies Concordance}\label{s3}

We start with an isotopy $g_r$ on $M$. To create a concordance from this isotopy, it seems natural to turn this into the metric $g_r+dr^2$ on $M\times[0,1]$. However, this metric does not necessarily have positive $(p,n+1)$-intermediate scalar curvature since, even though the metric $g_r$ on the slice $M\times\{r\}$ has positive curvature, there may be negative curvature coming from the $r$ direction. Therefore we will introduce a function $f:\R\to[0,1]$ and consider a new metric $g_{f(t)}+dt^2$ on $M\times\R$. This function will allow us to control the changes in the $t$ direction so we can minimize the contributions of negative curvature while keeping the positive contributions from the metric $g_{f(t)}$ on the slice $M\times\{f(t)\}$.

We begin with a lemma about the Riemann curvature of such a metric.

\begin{lemma}\label{curv-calc}
Let $M$ be an $n$-dimensional manifold, $g_r$ with $r\in [0, 1]$  a smooth path of metrics on $M$, and $f:\R\to[0,1]$ a smooth function. Define the metric $\bar{g} = g_{f(t)} + dt^2$ on $M\times\R$. If $(x_0,t_0)\in M\times R$ has local coordinates $(x_1,\ldots,x_n,t)$ where $(x_1,\ldots,x_n)$ are normal coordinates for $x_0$ with respect to the metric $g_{f(t_0)}$ on $M$, then the Riemannian curvatures $\bar R$ of $\bar g$ at $(x_0,t_0)$ satisfy
$$\bar R_{ijk}^\ell=R_{ijk}^\ell+O(|f'|^2),
\bar R_{ijk}^t=O(|f'|), \textrm{ and } 
\bar R_{itk}^t=O(|f'|^2)+O(|f''|).$$
In particular, the sectional curvatures $\bar K$ of $\bar g$ are then given by
 $$\bar K_{ij}=K_{ij}+O(|f'|^2) \textrm{ and} \,\,\bar K_{it} =  O(|f''|)+O(|f'|^2).$$

\end{lemma}
\begin{proof}

Denote the coordinate vector fields by $\partial_1,\ldots,\partial_n,\partial_t$. Let $\nabla$ and $\bar\nabla$ be the Levi-Civita connections of $g_{f(t_0)}$ and $\bar g$ respectively and let $\bar\Gamma_{ij}^k$ and $\Gamma_{ij}^k$ be the Christoffel symbols of $\bar\nabla$  and $\nabla$, respectively, at $(x_0,t_0)$.  We denote the Riemannian and sectional curvatures, respectively, by $\bar R_{ijk}^\ell$ and $\bar K_{ij}$ for $\bar g$ and $R_{ijk}^\ell$ and $K_{ij}$ for $g_{f(t_0)}$
Since $\bar K_{ij}=\sum\limits_\ell\bar g_{i\ell}\bar R_{ijj}^\ell$, and 
$\bar K_{it} =   -\bar g_{tt}\bar R_{iti}^t - \sum_{\ell\neq t} \bar g_{\ell t}\bar R_{iti}^\ell-\bar R_{iti}^t$, the results for the sectional curvatures follow immediately.

Since $\bar g_{ij}=(g_{f(t_0)})_{ij}$, we have $\bar\Gamma_{ij}^k=\Gamma_{ij}^k$
for $i,j,k=1,\ldots,n$.
Moreover, we have the constant values $\bar g_{it}=0$ for all $i=1,\ldots,n$ and $\bar g_{tt}=\bar g^{tt}=1$. Therefore
 
 $$\bar\Gamma_{it}^t=\bar\Gamma_{tt}^k=\bar\Gamma_{tt}^t=0,$$
$$\bar\Gamma_{ij}^t=-\frac12\partial_t\bar g_{ij}, \,\,\textrm{ and } \bar\Gamma_{it}^k=\frac12\sum_\ell\bar g^{k\ell}\partial_t\bar g_{i\ell}.$$

Setting $r=f(t)$,  we obtain
$$\partial_t\bar g_{ij}=\frac{\partial(g_r)_{ij}}{\partial r}f'(t),$$

and using normal coordinates at the point $(x_0, t_0)$ this gives us that 
$$\bar\Gamma_{ij}^t=-\frac12\left[\frac{\partial(g_r)_{ij}}{\partial r}\right]_{(x_0,f(t_0))}f'(t_0)=O(|f'|),$$ 
$$\bar\Gamma_{it}^k=\frac12\left[\frac{\partial(g_r)_{ik}}{\partial r}\right]_{(x_0,f(t_0))}f'(t_0)=O(|f'|).$$

We now calculate  the Riemannian curvatures. 
For $i,j,k,\ell=1,\ldots,n$, since $\bar\Gamma_{ij}^k=\Gamma_{ij}^k$ and $m$ varies over $1,\ldots,n$ and $t$, we have
\begin{align*}
\bar R_{ijk}^\ell&=\partial_i\Gamma_{jk}^\ell-\partial_j\Gamma_{ik}^\ell+\sum_{m\neq t}(\Gamma_{im}^\ell\Gamma_{jk}^m-\Gamma_{jm}^\ell\Gamma_{ik}^m)+\bar\Gamma_{it}^\ell\bar\Gamma_{jk}^t-\bar\Gamma_{jt}^\ell\bar\Gamma_{ik}^t\\
&=R_{ijk}^\ell -\frac14\left[\frac{\partial(g_r)_{i\ell}}{\partial r}\frac{\partial(g_r)_{jk}}{\partial r}-\frac{\partial(g_r)_{j\ell}}{\partial r}\frac{\partial(g_r)_{ik}}{\partial r}\right]_{(x_0,f(t_0))}(f'(t_0))^2\\
&=R_{ijk}^\ell+O(|f'|^2).
\end{align*}

For the Riemannian curvatures involving $t$, we only need to compute $\bar R_{ijk}^t$ and $\bar R_{itk}^t$ for $i,j,k=1,\ldots,n$. First,
\begin{align*}
\bar R_{ijk}^t&=-\frac12\left[\frac{\partial^2(g_r)_{jk}}{\partial x_i\partial r}-\frac{\partial^2(g_r)_{ik}}{\partial x_j\partial r}\right]_{(x_0,f(t_0))}f'(t_0)\\
&=O(|f'|).
\end{align*}
To calculate $\bar R_{itk}^t$, we note that a straightforward calculation gives us that 
at the point $(x_0,t_0)$, we have
$$\partial_t\bar\Gamma_{ik}^t=-\frac12\left[\frac{\partial^2(g_r)_{ik}}{\partial r^2}\right]_{(x_0,f(t_0))}(f'(t_0))^2-\frac12\left[\frac{\partial(g_r)_{ik}}{\partial r}\right]_{(x_0,f(t_0))}f''(t_0).$$
Therefore, we have
\begin{align*}
\bar R_{itk}^t&=\frac12\left[\frac{\partial(g_r)_{ik}}{\partial r}\right]_{(x_0,f(t_0))}f''(t_0)\\
&\quad+\left[\frac12\frac{\partial^2(g_r)_{ik}}{\partial r^2}-\frac14\sum_{m\neq t}\frac{\partial(g_r)_{im}}{\partial r}\frac{\partial(g_r)_{km}}{\partial r}\right]_{(x_0,f(t_0))}(f'(t_0))^2\\
&=O(|f''|)+O(|f'|^2).
\end{align*}
\end{proof}

We are now ready to estimate the $(p,n+1)$-intermediate scalar curvature on $(M\times\R, \bar g)$.

\begin{lemma}\label{isotopyconcordance}
Let $M$ be a compact $n$-dimensional manifold and $g_r$ with $r\in [0, 1]$  a smooth path of $(s_{p,n}>0)$-metrics on $M$. Then there exists a positive constant $C\leq 1$ 
so that for every smooth function $f: \R\rightarrow [0, 1]$ with $|f'|, |f''|\leq C$ the metric
$\bar{g} = g_{f(t)} + dt^2$ on $M\times\R$ has positive $(p, n+1)$-intermediate scalar curvature, $0\leq p\leq n-2$.
\end{lemma}
\begin{proof}
Choose a point $z_0=(x_0,t_0)\in M\times\R$. We use $\bar s$, $\bar R$, $\bar K$ to denote intermediate scalar, Riemann, and sectional curvatures for the metric $\bar g$ and $s$, $R$, $K$ for those curvatures in the metric $g_{f(t_0)}$. Fix a value $0\leq p\leq n-2$.

Since $M$ is compact, the Grassmannian $\Grass_p(M)$ is also compact. By the continuity of the curvature, this means for each $r$ there is a positive lower bound for the $(p,n)$-intermediate scalar curvature of the metric $g_r$. As the $r$ vary over the compact interval $[0,1]$, we can choose a positive lower bound $B_p$ for the $(p,n)$-intermediate scalar curvatures of all the metrics and thus  a positive lower bound $B_{p-1}$ for the $(p-1,n)$-intermediate scalar curvatures of the metrics $g_r$ as well.

Let $P$ be a $p$-plane in $T_{z_0}(M\times\R)$ so that $P^\perp$ is $(n-p+1)$-dimensional. We denote the hyperplane $T_{z_0}(M\times\{t_0\})$ by $T$. Consider the intersection $Q=P^\perp\cap T$. If $P^\perp$ is fully contained in $T$ then $Q=P^\perp$. If $P^\perp$ is not inside $T$, then the span of both must be the entire $(n+1)$-dimensional $T_{z_0}(M\times\R)$, and so by Grassmann's identity we have
\begin{align*}
\dim(Q)&=\dim(P^\perp\cap T)=\dim(P^\perp)+\dim(T)-\dim(P^\perp+T)\\
&=(n-p+1)+n-(n+1)=n-p.
\end{align*}
This leads to three cases: Case 1, where $P^\perp$ is contained in $T$, Case 2, where $P$ is contained in $T$, and Case 3, where either $P$ nor $P^\perp$ is contained in $T$.

\noindent{\bf Case 1:} $P^\perp$ is contained in $T$.

In this case, we can take an orthonormal basis $\{e_1,\ldots,e_{n-p+1}\}$ for $P^\perp$ and extend it to orthonormal bases $\{e_1,\ldots,e_n\}$ for $T$ and $\{e_1,\ldots,e_n,e_t\}$ for $T_{z_0}(M\times\R)$. Using the exponential map for $g_{f(t_0)}$, we get local coordinates $(x_1,\ldots,x_n,x_t)$ at $z_0$ where $(x_1,\ldots,x_n)$ are normal coordinates at $x_0$. Then we compute using Lemma \ref{curv-calc},
\begin{align*}
\bar s_{p,n+1}(P)&=\sum_{i,j=1}^{n-p+1}\bar K_{ij}\\
&=\sum_{i,j=1}^{n-p+1}\bigl[K_{ij}+O(|f'|^2)\bigr]\\
&=s_{p-1,n}(P\cap T)+O(|f'|^2).
\end{align*}
As $s_{p-1,n}(P\cap T)\geq B_{p-1}>0$, then for a small enough value of $C$, we can  force the contributions of $f$ to not be too negative and allow $\bar s_{p,n+1}(P)$ to remain positive.

\noindent{\bf Case 2:} $P$ is contained in $T$.

If $P\subseteq T$ then $T^\perp\subseteq P^\perp$, and since $T^\perp$ is $1$-dimensional, let $e_t$ be a unit vector spanning $T^\perp$. Then taking an orthonormal basis $\{e_1,\ldots,e_{n-p}\}$ for $Q=P^\perp\cap T$, we have an orthonormal basis $\{e_1,\ldots,e_{n-p},e_t\}$ for $P^\perp$. We can extend this to orthonormal bases $\{e_1,\ldots,e_n\}$ for $T$ and $\{e_1,\ldots,e_n,e_t\}$ for $T_{z_0}(M\times\R)$ and again we get local coordinates $(x_1,\ldots,x_n,x_t)$ at $z_0$ where $(x_1,\ldots,x_n)$ are normal coordinates at $x_0$. Then we compute using Lemma \ref{curv-calc}
\begin{align*}
\bar s_{p,n+1}(P)&=\sum_{i,j=1}^{n-p}\bar K_{ij}+2\sum_{i=1}^{n-p}\bar K_{it}\\
&=\sum_{i,j=1}^{n-p}\bigl[K_{ij}+O(|f'|^2)\bigr]+2\sum_{i=1}^{n-p}\bigl[O(|f'|^2)+O(|f''|)\bigr]\\
&=s_{p,n}(P)+O(|f'|^2)+O(|f''|).
\end{align*}
As $s_{p,n}(P)\geq B_p>0$, then for a small enough value of $C$, we can again force the contributions of $f$ to allow $\bar s_{p,n+1}(P)$ to remain positive.

\noindent{\bf Case 3:} Neither $P$ nor $P^\perp$ is contained in $T$.

In this case $Q=P^\perp\cap T$ is $(n-p)$-dimensional. We depict this situation in Figure \ref{Pplane}. Take $e_1,\ldots,e_{n-p}$ to be an orthonormal basis for $Q$, and choose $v\in P^\perp\setminus Q$ so that $e_1,\ldots,e_{n-p},v$ is an orthonormal basis for $P^\perp$. Then obviously $v$ is not in $T$, but it also is not perpendicular to $T$, or we would have $P\subseteq T$. Therefore $v$ has a nonzero projection into $T$ which we rescale to a unit vector $e_{n-p+1}$. Putting this together with the previously defined $e_i$, we extend to orthonormal bases $\{e_1,\ldots,e_n\}$ for $T$ and $\{e_1,\ldots,e_n,e_t\}$ for $T_{z_0}(M\times\R)$ and again we get local coordinates $(x_1,\ldots,x_n,x_t)$ at $z_0$ where $(x_1,\ldots,x_n)$ are normal coordinates at $x_0$. Note that by construction, we have $v=\alpha e_{n-p+1}+\beta e_t$, as the projection of $v$ onto $T$, is a multiple of $e_{n-p+1}$.

\begin{figure}[!htbp]
\vspace{5cm}
\hspace{3.5cm}
\begin{picture}(0,0)
\includegraphics{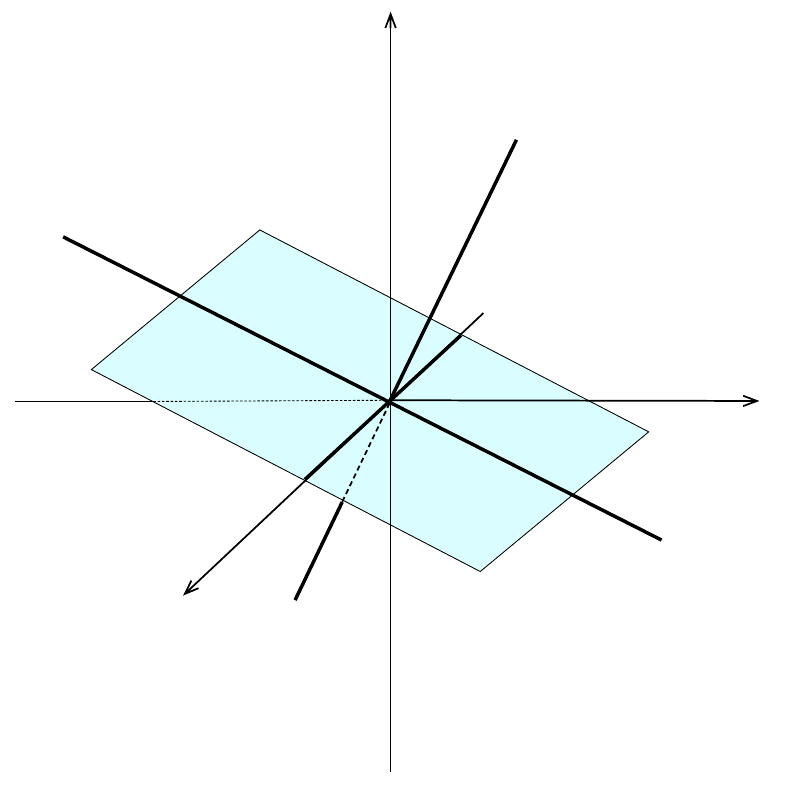}%
\end{picture}
\setlength{\unitlength}{3947sp}
\begin{picture}(5079,1559)(1902,-7227)
\put(2700,-6450){\makebox(0,0)[lb]{\smash{{\SetFigFont{10}{8}{\rmdefault}{\mddefault}{\updefault}{\color[rgb]{0,0,0}$Q$}%
}}}}
\put(4300,-4250){\makebox(0,0)[lb]{\smash{{\SetFigFont{10}{8}{\rmdefault}{\mddefault}{\updefault}{\color[rgb]{0,0,0}$P$}%
}}}}
\put(5000,-6000){\makebox(0,0)[lb]{\smash{{\SetFigFont{10}{8}{\rmdefault}{\mddefault}{\updefault}{\color[rgb]{0,0,0}$P^{\perp}\cap Q^{\perp}$}%
}}}}
\put(4100,-6350){\makebox(0,0)[lb]{\smash{{\SetFigFont{10}{8}{\rmdefault}{\mddefault}{\updefault}{\color[rgb]{0,0,0}$P^{\perp}$}%
}}}}
\put(5550,-5400){\makebox(0,0)[lb]{\smash{{\SetFigFont{10}{8}{\rmdefault}{\mddefault}{\updefault}{\color[rgb]{0,0,0}$ T_{z_0}(\{x_0\}\times\mathbb{R})$}%
}}}}
\put(2200,-3650){\makebox(0,0)[lb]{\smash{{\SetFigFont{10}{8}{\rmdefault}{\mddefault}{\updefault}{\color[rgb]{0,0,0}$Q^{\perp}\cap T_{z_0}(M\times\{t_0\})$}%
}}}}

\end{picture}%
\caption{The arrangement of the subspace $P\subset T_{z_0}(M\times\mathbb{R})$ in Case 3}
\label{Pplane}
\end{figure}
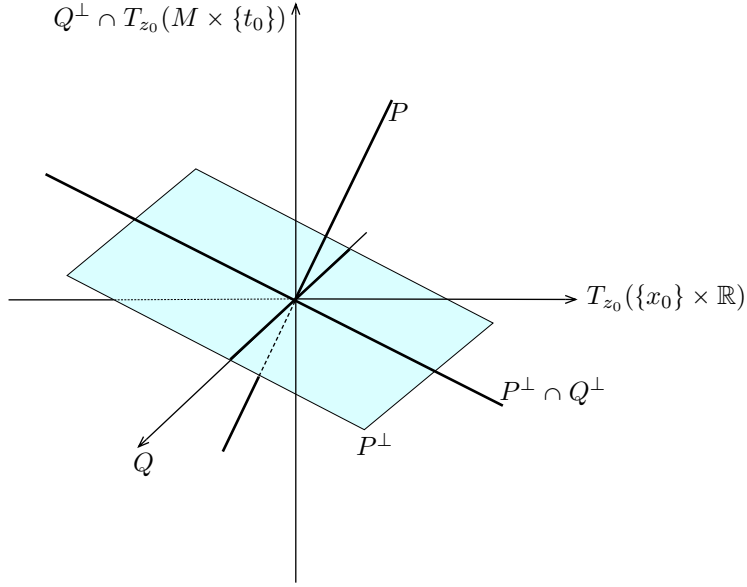   

We compute the curvatures for $1\leq i\leq n-p$ using Lemma \ref{curv-calc}
\begin{align*}
\bar K(e_i,v)&=\bar g(\bar R(v,e_i)e_i,v)\\
&=\bar g(\bar R(\alpha e_{n-p+1}+\beta e_t,e_i)e_i,\alpha e_{n-p+1}+\beta e_t)\\
&=\alpha^2\bar R_{(n-p+1)ii}^{n-p+1}+2\alpha\beta\bar R_{(n-p+1)ii}^t+\beta^2\bar R_{tii}^t\\
&=\alpha^2 \bigl[R_{(n-p+1)ii}^{n-p+1}+O(|f'|^2)\bigr]+2\alpha\beta O(|f'|)+\beta^2\bigl[O(|f'|^2)+O(|f''|)\bigr]\\
&=\alpha^2 K_{i(n-p+1)}+O(|f'|)+O(|f'|^2)+O(|f''|).
\end{align*}
Therefore we have
\begin{align*}
\bar s_{p,n+1}(P)&=\sum_{i,j=1}^{n-p}\bar K_{ij}+2\sum_{i=1}^{n-p}\bar K(e_i,v)\\
&=\sum_{i,j=1}^{n-p}\bigl[K_{ij}+O(|f'|^2)\bigr]+2\sum_{i=1}^{n-p}\bigl[\alpha^2 K_{i(n-p+1)}+O(|f'|)+O(|f'|^2)+O(|f''|)\bigr]\\
&=\sum_{i,j=1}^{n-p}K_{ij}+2\alpha^2\sum_{i=1}^{n-p}K_{i(n-p+1)}+O(|f'|)+O(|f'|^2)+O(|f''|).
\end{align*}

Set $A=\sum\limits_{i,j=1}^{n-p}K_{ij}$ and $B=2\sum\limits_{i=1}^{n-p}K_{i(n-p+1)}$. Then $A=s_{p,n}(Q^\perp)\geq B_p>0$ and $$A+B=\sum_{i,j=1}^{n-p+1}K_{ij}=s_{p-1,n}(P\cap T)\geq B_{p-1}>0.$$

Since $\alpha^2>0$,  for $B\geq 0$, we have $A+\alpha^2 B\geq A\geq B_p$. For $B<0$ since $\alpha^2<1$, we have $A+\alpha^2 B>A+B\geq B_{p-1}$. In either case, $A+\alpha^2 B\geq \min\{B_p, B_{p-1}\}$, and hence for a small enough choice of $C$, we can force the contributions of $f$ to allow $\bar s_{p,n+1}(P)$ to remain positive as well.

\end{proof}

Rephrasing this result in the language of isotopy and concordance, we have the following:

\begin{proposition}\label{isotoconc} Let $M$ be a smooth compact manifold of dimension $n$. Then, for any $p\in\{0,1,\cdots, n-2\}$,  metrics which are $(s_{p,n}>0)$-isotopic on $M$ are also $(s_{p,n}>0)$-concordant.
\end{proposition}
\begin{proof}
Let $g_0$ and $g_1$ be two $(s_{p,n}>0)$-isotopic metrics on $M$ with the isotopy $g_r$ for $r\in[0,1]$. By Lemma \ref{isotopyconcordance}, there is a $C\leq 1$ such that for every smooth function $f:\R\to[0,1]$ with $|f'|,|f''|\leq C$, the metric $\bar{g}=g_{f(t)}+dt^2$ on $M\times R$ has positive $(p,n+1)$-intermediate scalar curvature.

\begingroup\singlespacing
Let $\mu:(-\infty,\infty)\to[0,1]$ be the function
$$\mu(t)=\begin{cases}0&t\leq 0\cr {e^{-1/t}\over e^{-1/t}+e^{-1/(1-t)}}&t\in(0,1)\cr 1&t\geq 1\end{cases}$$
\endgroup
that smoothly transitions from $0$ to $1$ over the interval $[0,1]$. For any $L>0$, a translation and rescaling gives us the function $\mu_L(t)=\mu\left({t-1\over L}\right)$ that smoothly transitions from $0$ to $1$ over the interval $[1,L+1]$. The derivatives $\mu'$ and $\mu''$ are bounded and we have $\mu_L'(t)={1\over L}\mu'\left({t-1\over L}\right)$ and $\mu_L''(t)={1\over L^2}\mu''\left({t-1\over L}\right)$. Therefore we can choose $L$ sufficiently large to force $|\mu_L'|,|\mu_L''|\leq C$.

Taking $f$ to be the restriction of $\mu_L$ to the interval $[0,L+2]$, the manifold $M\times[0,L+2]$ with metric $\bar{g}=g_{f(t)}\times dt^2$ has positive $(p,n+1)$-intermediate scalar curvature. Since $\bar{g}=g_0\times dt^2$ for $0\leq t\leq 1$ and $\bar{g}=g_1\times dt^2$ for $L+1\leq t\leq L+2$, then by definition, this is a $(s_{p,.n}>0)$-concordance between $g_0$ and $g_1$.
\end{proof}

\section{Curvature of Warped Product Metrics}\label{s4}
\setcounter{figure}{0}
We first fix notation.
Let  $(B^b, g_B)$ and $(F^n, g_F)$ be, respectively, $b$ and $n$-dimensional Riemannian manifolds and  consider their product, $M=B\times F$ with the warped product metric $g=g_{B}+\beta^{2}g_{F}$, where $\beta:B\rightarrow (0,\infty)$ is a smooth function. We assume that $b,n\geq 1$. We denote by $\pi_{B}$ and $\pi_{F}$, the corresponding projections from $M$ to $B$ and $F$,
and for a point $x\in M$, set
$$\check{x}:=\pi_{B}(x)\quad \mathrm{ and } \quad \hat{x}=\pi_{F}(x).$$
At each point $x\in M$, the maps, $\pi_{B}$ and $\pi_{F}$, induce derivative maps
$$(\pi_{B})_*:T_xM\rightarrow T_{\check{x}}B\quad \mathrm{ and } \quad (\pi_{F})_*:T_xM\rightarrow T_{\hat{x}}F.$$
The warped product structure of the metric gives us the following horizontal and vertical spaces at $x\in M$
$$\msH_{x}:=T_{x}(B\times\{\pi_{F}(x)\})\quad \mathrm{ and } \quad \msV_{x}:=T_{x}(\{\pi_{B}(x)\}\times F).$$
In particular, the restriction of the derivative map $(\pi_{B})_{*}$ to the horizontal space $\msH_{x}$ is the isometry
$$(\pi_{B})_{*}|_{\msH_x}: (\msH_{x}, g_{x}|_{\msH_{x}})\rightarrow (T_{\check{x}}B, (g_{B})_{\check{x}}).$$

We denote the {\em vertical} and {\em horizontal distributions} of the submersion  by $\msV$ or $\msH$. The notation, $\msV$, $\msH$, serves a dual purpose as we also use it to mean the projection onto the vertical or horizontal subspace. Let $u\in T_{x}M$ be some tangent vector. Then, $$u_{F}:=\msV(u)\in \msV_{x}\quad \mathrm{and} \quad u_{B}:=\msH(u)\in\msH_{x},$$ denote the corresponding orthogonal projections. 
Note that the vectors $u_B$ and $u_V$ in $\msV_{x}$ and $\msH_{x}$ are distinct from their corresponding images under $(\pi_{F})_{*}$ and $(\pi_{B})_{*}$.  In the case of the derivative maps, we write:
$$\hat{u}:=(\pi_{F})_{*}(u)\quad \mathrm{ and } \quad \check{u}=(\pi_{B})_{*}(u).$$ 

\subsection{The Riemann Curvature Tensor.} In computing the curvature tensor, we will make use of well-known formulas of Gray \cite{gray} and O'Neill \cite{oneill}, (see also  Theorem 9.28 in \cite{besse}). These formulas involve the tensors, ${\bf A,T}:\Gamma TM\otimes \Gamma TM\rightarrow \Gamma TM$, defined for vectors fields $E_1, E_2$ on $M$ as follows:
\begin{equation*}
\begin{split}
{\bf A}_{E_{1}} E_{2}&=\msH(\nabla_{\msH(E_{1})} \msV(E_{2}))+\msV(\nabla_{\msH(E_{1})} \msH(E_{2})).\\
{\bf T}_{E_{1}} E_{2}&=\msH(\nabla_{\msV(E_{1})} \msV(E_{2}))+\msV(\nabla_{\msV(E_{1})} \msH(E_{2})).
\end{split}
\end{equation*}

Note that here the horizontal distribution for each $x\in M$ is naturally identified with $T_{\check{x}}B$ and in particular, this horizontal distribution is {\em integrable} (see Chapter 19 in \cite{lee}) so the tensor ${\bf A}$ above vanishes in our case (see 9.24 in \cite{besse}).

We recall the formulas from  \cite{besse} involving the ${\bf T}$ tensor.
\begin{theorem}[\cite{besse}]\label{ONT} 
Let manifold $M=B\times F$ be equipped with the warped product metric $g=g_{B}+\beta^{2}g_{F}$. Let $X,Y$ be a pair of horizontal vector fields  and $U,V$, a pair of vertical vector fields tangent to $M$. Then
\begin{align*}
{\bf T}_{X}U&={\bf T}_{X}{Y}=0, &{\bf T}_{U}V&={\bf T}_{V}U,\\
{\bf T}_{U}V&=\msH(\nabla_{U}V), & {\bf T}_{U}X&=\msV(\nabla_{U}X),
\end{align*}
and
$$g({\bf T}_{U}V, X)=-g({\bf T}_{U}X, V).$$
\end{theorem}

Using the fact that ${\bf A}$ is everywhere zero, we recall Theorem 9.28 in \cite{besse}, 
\begin{theorem}[\cite{besse}]\label{GONRm}
Let manifold $M=B\times F$ be equipped with the warped product metric $g=g_{B}+\beta^{2}g_{F}$. Let $\pi_B:(M,g)\to(B,g_B)$ denote a Riemannian warped product submersion for some warping function $\beta:B\to(0,\infty)$. Let $X,Y,Z, Z'$ be horizontal and $U,V,W, W'$ vertical vector fields tangent to $M$. Finally, let $R_{B}$ and $R_{F}$ denote the respective Riemann curvature tensors for $(B, g_{B})$ and $(F, g_{F})$. Then the Riemann curvature tensor, $R$, of $g$ satisfies the following properties.
\begin{enumerate}
\item[(i.)] $R(U,V,W,W') = g(R_{F}(\hat{U},\hat{V})\hat{W}, \hat{W'})+g({\bf T}_{U}W, {\bf T}_{V}W')-g({\bf T}_{V}W, {\bf T}_{U}W')$,
\item[(ii.)] $R(U,V,W,X) = g((\nabla_{U}{\bf T})_{V}W,X)-g((\nabla_{V}{\bf T})_{U}W,X)$,
\item[(iii.)] $R(X,U,Y,V) =g({\bf T}_{U}X,{\bf T}_{V}Y)- g((\nabla_{X}{\bf T})_{U}V,Y)$,
\item[(iv.)] $R(U,V,X,Y) = g({\bf T}_{U}X,  {\bf T}_{V}Y)-g({\bf T}_{V}X, {\bf T}_{U}Y)$,
\item[(v.)] $R(X,Y,Z,U) = 0$,
\item[(vi.)] $R(X,Y,Z,Z') = g(R_{B}(\check{X},\check{Y})\check{Z} ,\check{Z}')$.
\end{enumerate}
\end{theorem}

We now  introduce some conventions we will use to make our computations easier to follow. On the manifold $B\times F$, we assume coordinate vector fields $\p_1, \cdots, \p_{b+n}$, such that $\p_1,\ldots,\p_b$ are horizontal and $\p_{b+1},\ldots,\p_{b+n}$ are vertical and $g_x(\p_i,\p_j)=0$ for $i\neq j$, and  $g_x(\p_i,\p_j)=\beta(\check x)^2\delta_{ij}$ for $i,j\in\{b+1,\cdots,b+n\}$.
We adopt the convention that the indices $\lambda,\mu,\nu$ will be used for the base directions $\{1,\ldots,b\}$ and the indices $i,j,k,\ell$ will be used for the fiber directions $\{b+1,\ldots,b+n\}$. For an index that varies over all of $\{1,\ldots,b+n\}$, we use $s$.

The following lemma gives us the Christoffel symbols for a warped product metric. We note that there are six cases to consider depending on whether the coordinates are in the base or the fiber. The proof is a straightforward computation that we leave to the reader (see \cite{B} for more details).
\begin{lemma} Let $(B\times F, g_{B}+\beta^{2}g_{F})$ be as described above.
Then  at a point $x\in B\times F$, the Christoffel symbols are given by the following equations
\begin{align}\label{Christoffelatpoint}
\Gamma_{\lambda\mu}^\nu &= \check\Gamma_{\lambda\mu}^\nu,&
\Gamma_{ij}^k &= 0,\nonumber\\
\Gamma_{\lambda\mu}^k &= 0,&
\Gamma_{\lambda j}^\nu &=0,\\
\Gamma_{\lambda j}^k &= {\beta_\lambda\over\beta}\delta_{jk},&
\Gamma_{ij}^\nu &= -\beta\beta_\nu g_B^{\nu\nu}\delta_{ij}.\nonumber
\end{align}
\end{lemma}

From now on, we let $v_{\Bb}$ and $v_{\Ff}$ denote the horizontal and vertical parts of $v$.
Using Parts (i.), (v.), and (vi.) of Theorem \ref{GONRm} and the symmetries of the curvature tensor, we obtain the following lemma, giving us an expression for the Riemann curvature tensor of a $2$-plane, $P\subset T_{x}M$ generated by $v$ and $w$ in terms of the horizontal and vertical parts of each vector.  
\begin{lemma}  Let  $(B\times F, g_{B}+\beta^{2}g_{F})$ be as described above. Given $x\in B\times F$, consider an arbitrary $2$-dimensional subspace $P\subset T_{x}M$. We let $v,w\in P\subset T_{x}M$ be an arbitrary pair of linearly independent vectors. Then 
\begin{equation}\label{RPSumSimpler2}
\begin{split}
R(v,w,w,v)\hspace{0.2cm}   &=  R_B(\check{v}_{\Bb}, \check{w}_{\Bb},\check{w}_{\Bb},\check{v}_{\Bb})\\
&+ \beta^2 R_{F}(\hat{v}_{\Ff}, \hat{w}_{\Ff},\hat{w}_{\Ff},\hat{v}_{\Ff})+g({\bf T}_{v_{\Ff}}{w_{\Ff}},{\bf T}_{v_{\Ff}}{w_{\Ff}})-g({\bf T}_{v_{\Ff}}{v_{\Ff}},{\bf T}_{w_{\Ff}}{w_{\Ff}})\\\\
&+g((\nabla_{v_{\Bb}}{\bf T})_{w_{\Ff}}{w_{\Ff}},{v_{\Bb}})-g({\bf T}_{w_{\Ff}}{v_{\Bb}},{\bf T}_{w_{\Ff}}{v_{\Bb}})\\
&+g((\nabla_{w_{\Bb}}{\bf T})_{v_{\Ff}}{v_{\Ff}},{w_{\Bb}})-g({\bf T}_{v_{\Ff}}{w_{\Bb}},{\bf T}_{v_{\Ff}}{w_{\Bb}})\\\\
&-2g((\nabla_{w_{\Ff}}{\bf T})_{v_{\Ff}}{w_{\Ff}},{v_{\Bb}})+2g((\nabla_{v_{\Ff}}{\bf T})_{w_{\Ff}}{w_{\Ff}},{v_{\Bb}})\\
&+2g((\nabla_{w_{\Ff}}{\bf T})_{v_{\Ff}}{v_{\Ff}},{w_{\Bb}})-2g((\nabla_{v_{\Ff}}{\bf T})_{w_{\Ff}}{v_{\Ff}},{w_{\Bb}})\\\\
&+2g({\bf T}_{w_{\Ff}}{v_{\Bb}},  {\bf T}_{v_{\Ff}}{w_{\Bb}})-2g({\bf T}_{v_{\Ff}}{v_{\Bb}}, {\bf T}_{w_{\Ff}}{w_{\Bb}})\\
&+2g({\bf T}_{w_{\Ff}}{v_{\Bb}},{\bf T}_{v_{\Ff}}{w_{\Bb}})- 2g((\nabla_{v_{\Bb}}{\bf T})_{w_{\Ff}}{v_{\Ff}},{w_{\Bb}}).
\end{split}
\end{equation}

\end{lemma}

Using Theorem \ref{ONT} and Display \ref{Christoffelatpoint}, we obtain the following expressions for the components of the ${\bf T}$ tensor at the point $x$:

$${\bf T}_{\p_i}{\p_{j}}=-\sum_{\lambda}\beta\beta_\lambda g_{B}^{\lambda\lambda}\delta_{ij}\p_\lambda, \,\,
\mathrm{ and }\,\,
{\bf T}_{\p_i}{\p_\lambda} =  \frac{\beta_\lambda}{\beta}\p_{i}.$$
With these expressions, we can compute the inner products involving  the $\bf T$ tensor in Display \ref{RPSumSimpler2}. 
We obtain
\begin{equation}\label{T}
\begin{split}
g({\bf T}_{\p_i}{\p_j}, {\bf T}_{\p_k}{\p_\ell})\,\,\,\,&=\,\,\beta^{2}\delta_{ij}\delta_{k\ell}\sum_\lambda(\beta_\lambda)^{2}g_{\lambda\lambda}^B,\\
g({\bf T}_{\p_i}{\p_\lambda}, {\bf T}_{\p_j}{\p_\mu})\,\,\,\,&=\,\,\beta_\lambda\beta_\mu\delta_{ij},\\
g({\bf T}_{\p_i}{\p_\lambda}, {\bf T}_{\p_j}{\p_\mu})\,\,\,\,&=\,\,\beta_\lambda\beta_\mu\delta_{ij},\\
g((\nabla_{\p_\lambda}{\bf T})_{\p_{i}} {\p_{j}},\p_\mu)& =\,\,\, -\delta_{ij}\left[\bigl(\beta\beta_{\lambda\mu}-\beta_\lambda\beta_\mu\bigr)+\sum_\nu \beta\beta_\nu\bigl(\p_\lambda(g_{B}^{\mu\nu})+g_{B}^{\nu\nu}\check\Gamma_{\lambda\nu}^\mu\bigr)g_{\mu\mu}^B\right], \,\, \mathrm{and}\\
g((\nabla_{\p_i}{\bf T})_{\p_{j}} {\p_{k}},\p_\lambda) &=\,\,\,0.\\
\end{split}
\end{equation}

Setting $v=\sum\limits_{s=1}^{b+n} v_s\p_s\quad\text{and}\quad w=\sum\limits_{s=1}^{b+n} w_s\p_s$, then 
$v_{\Bb}=\sum\limits_\lambda v_\lambda\p_\lambda, \,\,w_{\Bb}=\sum\limits_\lambda w_\lambda\p_\lambda,\,\,
v_{\Ff}=\sum\limits_i v_i\p_i,\,\, \mathrm{and }\,\,w_{\Ff}=\sum\limits_i w_i\p_i$.
With these conventions, from Display \ref{T}, we obtain 
\begin{equation}\label{metric}
\begin{split}
g({\bf T}_{v_{\Ff}}{w_{\Ff}},{\bf T}_{v_{\Ff}}{w_{\Ff}})\,\,\,\,\,&=\,\,\,\, \sum_{i,j,\lambda} v_iw_iv_jw_j\, \beta^{2}\beta_\lambda^{2}g_{\lambda\lambda}^B,\\
g({\bf T}_{v_{\Ff}}{v_{\Ff}},{\bf T}_{w_{\Ff}}{w_{\Ff}})\,\,\,\,\,&=\,\,\,\, \sum_{i,j,\lambda} v_i^2w_j^2\, \beta^{2}\beta_\lambda^{2}g_{\lambda\lambda}^B,\\
g({\bf T}_{w_{\Ff}}{v_{\Bb}},{\bf T}_{w_{\Ff}}{v_{\Bb}})\,\,\,\,\,&=\,\,\,\, \sum_{i,\lambda,\mu} w_i^2 v_\lambda v_\mu\,\beta_\lambda\beta_\mu,\\
g({\bf T}_{v_{\Ff}}{w_{\Bb}},{\bf T}_{v_{\Ff}}{w_{\Bb}})\,\,\,\,\,&=\,\,\,\, \sum_{i,\lambda,\mu} v_i^2 w_\lambda w_\mu\,\beta_\lambda\beta_\mu,\\
g({\bf T}_{v_{\Ff}}{v_{\Bb}}, {\bf T}_{w_{\Ff}}{w_{\Bb}})\,\,\,\,\,&=\,\,\,\, \sum_{i,\lambda,\mu} v_iw_i v_\lambda w_\mu\,\beta_\lambda\beta_\mu,\\
g({\bf T}_{w_{\Ff}}{v_{\Bb}},  {\bf T}_{v_{\Ff}}{w_{\Bb}})\,\,\,\,\,&=\,\,\,\, \sum_{i,\lambda,\mu} v_iw_i v_\lambda w_\mu\,\beta_\lambda\beta_\mu,\\
g((\nabla_{w_{\Ff}}{\bf T})_{v_{\Ff}}{w_{\Ff}},{v_{\Bb}})&=0,\\
g((\nabla_{v_{\Ff}}{\bf T})_{w_{\Ff}}{w_{\Ff}},{v_{\Bb}})&=0,\\
g((\nabla_{w_{\Ff}}{\bf T})_{v_{\Ff}}{v_{\Ff}},{w_{\Bb}})&=0,\\
g((\nabla_{v_{\Ff}}{\bf T})_{w_{\Ff}}{v_{\Ff}},{w_{\Bb}})&=0,\\
g((\nabla_{v_{\Bb}}{\bf T})_{w_{\Ff}}{w_{\Ff}},{v_{\Bb}}) &= -\sum_{i,\lambda,\mu} w_i^2v_\lambda v_\mu\left[\beta\beta_{\lambda\mu}-\beta_\lambda\beta_\mu+\sum_\nu \beta\beta_\nu\bigl(\p_\lambda(g_{B}^{\mu\nu})+g_{B}^{\nu\nu}\check\Gamma_{\lambda\nu}^\mu\bigr)g_{\mu\mu}^B\right],\\
g((\nabla_{w_{\Bb}}{\bf T})_{v_{\Ff}}{v_{\Ff}},{w_{\Bb}}) &= -\sum_{i,\lambda,\mu} v_i^2w_\lambda w_\mu\left[\beta\beta_{\lambda\mu}-\beta_\lambda\beta_\mu +\sum_\nu \beta\beta_\nu\bigl(\p_\lambda(g_{B}^{\mu\nu})+g_{B}^{\nu\nu}\check\Gamma_{\lambda\nu}^\mu\bigr)g_{\mu\mu}^B\right],\,\,\mathrm{and}\\
g((\nabla_{v_{\Bb}}{\bf T})_{w_{\Ff}}{v_{\Ff}},{w_{\Bb}})&= -\sum_{i,\lambda,\mu} v_iw_iv_\lambda w_\mu\left[\beta\beta_{\lambda\mu}-\beta_\lambda\beta_\mu +\sum_\nu \beta\beta_\nu\bigl(\p_\lambda(g_{B}^{\mu\nu})+g_{B}^{\nu\nu}\check\Gamma_{\lambda\nu}^\mu\bigr)g_{\mu\mu}^B\right].
\end{split}
\end{equation}
Substituting these equalities from Display \ref{metric} into Display \ref{RPSumSimpler2}, we obtain the following result.

\begin{proposition} Let  $(B\times F, g_{B}+\beta^{2}g_{F})$ be as described above. Given $x\in B\times F$, consider an arbitrary $2$-dimensional subspace $P\subset T_{x}M$. We let $v,w\in P\subset T_{x}M$ be an arbitrary pair of linearly independent vectors, written in components as $v=\sum\limits_{s=1}^{b+n} v_s\p_s\quad\text{and}\quad w=\sum\limits_{s=1}^{b+n} w_s\p_s$. Then  
\begin{equation}\label{Riemanncurvatureofwarped}
\begin{split}
R(v,w,w,v)  &=  R_B(\check{v}_{\Bb}, \check{w}_{\Bb},\check{w}_{\Bb},\check{v}_{\Bb})+ \beta^2 R_{F}(\hat{v}_{\Ff}, \hat{w}_{\Ff},\hat{w}_{\Ff},\hat{v}_{\Ff})\\
&-\sum_{i<j,\lambda} (v_iw_j-v_jw_i)^2\beta^{2}\beta_\lambda^{2}g_{\lambda\lambda}^B\\
&-\sum_{i,\lambda,\mu} (w_i^2v_\lambda v_\mu+v_i^2w_\lambda w_\mu-2v_iw_iv_\lambda w_\mu)\left[\beta\beta_{\lambda\mu}+\sum_\nu \beta\beta_\nu\bigl(\p_\lambda(g_{B}^{\mu\nu})+g_{B}^{\nu\nu}\check\Gamma_{\lambda\nu}^\mu\bigr)g_{\mu\mu}^B\right].
\end{split}
\end{equation}
\end{proposition}

Note that in the case of a  product metric, $\beta=1$. Since this makes each of the derivatives of $\beta$ zero, Equation \ref{Riemanncurvatureofwarped} reduces to the well-known formula
\begin{equation}\label{standard riemann product formula}
R(v,w,w,v)  =  R_B(\check{v}_{\Bb}, \check{w}_{\Bb},\check{w}_{\Bb},\check{v}_{\Bb})+ R_{F}(\hat{v}_{\Ff}, \hat{w}_{\Ff},\hat{w}_{\Ff},\hat{v}_{\Ff}).
\end{equation}

\subsection{The Sectional Curvature of the Specific Warped Product}

We now restrict our attention to a metric of the form 
$$g=dr^2+\omega(r,t)^2dt^2+\beta(r)^2ds_n^2$$
defined on the product manifold $M=B\times F$ where $B=(0,b_1)\times(0,b_2)$ and $F=S^n$ for $n\geq 2$ with the standard round metric, where $\beta:(0,b_1)\to(0,\infty)$ and $\omega:(0,b_1)\times(0,b_2)\to(0,\infty)$ are smooth warping functions. We are interested in computing the $(p,n+2)$-intermediate scalar curvatures of $(M,g)$. We begin by computing the sectional curvatures using what we have already computed. 

\begin{lemma}\label{curvatures-base}
The sectional curvatures of  $M=B\times F$ as above, are given by
\begin{equation}
 K_{rt}=-{\omega_{rr}\over\omega},\quad \displaystyle K_{ri}={-\beta_{rr}\over\beta},\quad K_{ti} = -{\omega_r\beta_r\over \omega\beta},\quad \displaystyle K_{ij}={1-\beta_r^2\over\beta^2}.
\end{equation}
\end{lemma}
\begin{proof}
The metric on the base is  given by $g_B=dr^2 + \omega(r,t)^2 dt^2$. We  compute the Christoffel symbols and obtain
\begin{equation}\label{christoffelofwarped}
\begin{split}
\Gamma_{rr}^r&=0,\\
\Gamma_{rr}^t&=0,\\
\Gamma_{rt}^r&=0,\\
\Gamma_{tt}^r&=-\omega\omega_r,\\
\Gamma_{tt}^t&=\frac{\omega_t}{\omega},\,\, \mathrm{and}\\
\Gamma_{rt}^t&=\frac{\omega_r}{\omega}.
\end{split}
\end{equation}

 Note that since we are dealing with coordinate vector fields, we have $\nabla_{\partial_r}\partial_t=\nabla_{\partial_t}\partial_r$. Thus,
\begin{align*}
\nabla_{\partial_r}\partial_r&=0,\\
\nabla_{\partial_t}\partial_t&=-\omega\omega_r\partial_r +\frac{\omega_t}{\omega}\partial_t,\\
\nabla_{\partial_r}\partial_t&=\frac{\omega_r}{\omega}\partial_t.\end{align*}
 Since $\nabla_{\partial_r}\partial_r=0$, we have $\nabla_{\partial_t}\nabla_{\partial_r}\partial_r=0$ and a calculation gives
$\nabla_{\partial_r}\nabla_{\partial_t}\partial_r= \frac{\omega_{rr}}{\omega}\partial_t$.
Therefore
\begin{equation}\label{e1}
\langle R(\partial_t, \partial_r)\partial_r, \partial_t\rangle= -\frac{\omega_{rr}}{\omega}\omega^2.
\end{equation}


We now expand out the sums over the base directions $r$ and $t$ in our Riemann curvature equation \ref{Riemanncurvatureofwarped}. Since we know the form of the base metric and its inverse matrix, and that the function $\beta$ only depends on $r$, so that $\beta_t=0$. Equation \ref{Riemanncurvatureofwarped} then simplifies to
\begin{equation*}
\begin{split}
R(v,w,w,v)  &=  R_B(\check{v}_{\Bb}, \check{w}_{\Bb},\check{w}_{\Bb},\check{v}_{\Bb})+ \beta^2 R_{F}(\hat{v}_{\Ff}, \hat{w}_{\Ff},\hat{w}_{\Ff},\hat{v}_{\Ff})\\
&-\sum_{i<j,\lambda} (v_iw_j-v_jw_i)^2\beta^{2}\beta_r^{2}\\
&-\sum_{i} (w_i^2v_r v_r+v_i^2w_r w_r-2v_iw_iv_r w_r)\left[\beta\beta_{rr}+\beta\beta_r\check\Gamma_{rr}^r\right]\\
&-\sum_{i} (w_i^2v_r v_t+v_i^2w_r w_t-2v_iw_iv_r w_t)\beta\beta_r\check\Gamma_{rr}^t\omega^2\\
&-\sum_{i} (w_i^2v_t v_r+v_i^2w_t w_r-2v_iw_iv_t w_r)\beta\beta_r\check\Gamma_{tr}^r\\
&-\sum_{i} (w_i^2v_t v_t+v_i^2w_t w_t-2v_iw_iv_t w_t)\beta\beta_r\check\Gamma_{tr}^t\omega^2.\\
\end{split}
\end{equation*}
Using the Christoffel symbols we computed in Display \ref{christoffelofwarped}, we then obtain
\begin{equation*}
\begin{split}
R(v,w,w,v) &=R_B(\check{v}_{\Bb}, \check{w}_{\Bb},\check{w}_{\Bb},\check{v}_{\Bb})+ \beta^2 R_{F}(\hat{v}_{\Ff}, \hat{w}_{\Ff},\hat{w}_{\Ff},\hat{v}_{\Ff})
-\sum_{i<j}(v_iw_j-v_jw_i)^2\beta^{2}\beta_r^{2}\\
&\,\,\,\,\,\,\,\,-\sum_{i} (v_iw_r-w_iv_r)^2\beta\beta_{rr}
-\sum_{i}(v_iw_t-w_iv_t)^2\beta\beta_r\omega\omega_r.
\end{split}
\end{equation*}

The fibers $F$ are unit spheres of dimension at least $2$ with $K_F=1$, so we have
\begin{equation*}
\begin{split}
R_{F}(\hat{v}_{\Ff}, \hat{w}_{\Ff},\hat{w}_{\Ff},\hat{v}_{\Ff}) &= |\hat{v}_{\Ff}\wedge\hat{w}_{\Ff}|_F^2 K_F(\hat{v}_{\Ff},\hat{w}_{\Ff})\\
&=\sum_{i<j} (v_iw_j-v_jw_i)^2.
\end{split}
\end{equation*}

Similarly, the base $B$ is two-dimensional and from Lemma \ref{curvatures-base} its section curvature is given by $K_{rt}=-{\omega_{rr}\over\omega}$. So using the metric $g_B=dr^2+\omega^2dt^2$, we have
\begin{equation*}
\begin{split}
R_{F}(\check{v}_{\Bb}, \check{w}_{\Bb},\check{w}_{\Bb},\check{v}_{\Bb}) &= |\check{v}_{\Bb}\wedge\check{w}_{\Bb}|_B^2 K_F(\check{v}_{\Bb},\check{w}_{\Bb})\\
&= -\omega\omega_{rr}(v_tw_r-v_rw_t)^2\\
\end{split}
\end{equation*}
Therefore, we get
\begin{equation}\label{fullriemformainmetric}
\begin{split}
R(v,w,w,v) &= -\omega\omega_{rr}(v_tw_r-v_rw_t)^2+\sum_{i<j}(v_iw_j-v_jw_i)^2\beta^2(1-\beta_r^{2})\\
&\,\,\,\,\,\,\,\,-\sum_{i} (v_iw_r-w_iv_r)^2\beta\beta_{rr}-\sum_{i}(v_iw_t-w_iv_t)^2\beta\beta_r\omega\omega_r.\\
\end{split}
\end{equation}

We now compute the specific sectional curvatures when the vectors $v$ and $w$ are in the coordinate directions.
First we start with $v=\partial_t$ and $w=\partial_i$, so that $v_t=1$, $v_r=v_j=0$ for all $j$ directions in the fiber, $w_r=w_t=0$, $w_i=1$, and $w_j=0$ for all $j\neq i$. Substituting this information into Equation \ref{fullriemformainmetric}, we get
\begin{equation}\label{e2}
R_{tiit} =-(v_tw_i)^2\beta\beta_r\omega\omega_r=-\beta\beta_r\omega\omega_r.
\end{equation}

Next, let $v=\partial_r$ and $w=\partial_i$, so that $v_r=1$, $v_t=v_j=0$ for all $j$ directions in the fiber, $w_r=w_t=0$, $w_i=1$, and $w_j=0$ for all $j\neq i$. Substituting this information into Equation \ref{fullriemformainmetric}, we get
\begin{equation}\label{e3}
R_{riir} =-(v_rw_i)^2\beta\beta_{rr}=-2\beta_r^2-\beta\beta_{rr}.
\end{equation}

Finally, let $v=\partial_i$ and $w=\partial_j$ for $i\neq j$ so that $v_r=v_t=w_r=w_t=0$ and $v_k=w_k=0$, except for $v_i=1$ and $w_j=1$. Since $K_F=1$, we have
\begin{equation}\label{e4}
R_{ijji} =(v_iw_j)^2\beta^2(1-\beta_r^2) = \beta^2(1-\beta_r^2).
\end{equation}
The result now follows from Equations \ref{e1}, \ref{e2}, \ref{e3}, and \ref{e4}.
\end{proof}

\subsection{Intermediate Scalar Curvature}
We will now use the sectional curvatures we have computed in the previous section to derive a formula for the $(p,n+2)$-intermediate scalar curvatures of our metric. Let $0\leq p\leq n$ and let $P$ be a $p$-plane in $T_xM$. Then $P^\perp$ is a $(n-p+2)$-plane. Since the dimension of $P^\perp+T_xS^n$ can be at most the dimension of $T_xM$,
\begin{equation*}
\begin{split}
\dim(P^\perp\cap T_xS^n)&=\dim(P^\perp)+\dim(T_xS^n)-\dim(P^\perp+T_xS^n)\\
&\geq (n-p+2)+n-(n+2)\\
&=n-p.
\end{split}
\end{equation*}
Therefore, we know that there are at least $n-p$ linearly independent vectors in $P^\perp$ tangent to the sphere.

Without loss of generality, we will assume that the coordinate vector fields on the sphere are $\p_1,\ldots,\p_n$ where $\p_1,\ldots,\p_{n-p}$ are in $P^\perp$. Completing an orthogonal basis for $P^\perp$, we have at most two basis vectors $v,w$ that are not in $T_xS^n$. We consider three cases.

\noindent{\bf Case 1.} The projections of $v$ and $w$ into $T_xS^n$ span a $0$-dimensional subspace. This means that $v$ and $w$ do not have any fiber component and must be spans of $\p_r$ and $\p_t$. However, this means that $\p_r$ and $\p_t$ must be in $P^\perp$ and we can just assume that these are our $v$ and $w$. In other words, $P^\perp$ has the orthogonal basis $\{\p_r,\p_t,\p_1,\ldots,\p_{n-p}\}$. We compute the $(p,n+2)$-intermediate scalar curvature,
\begin{equation*}
\begin{split}
s_{p,n+2}(P)&=2\sum_{i<j} K_{ij}+2\sum_{i=1}^{n-p}K_{ri}+2\sum_{i=1}^{n-p}K_{ti}+2K_{rt}\\
&=(n-p)(n-p-1)\frac{1-\beta_r^2}{\beta^2}-2(n-p)\frac{\beta_{rr}}{\beta}-2(n-p)\frac{\omega_r\beta_r}{\omega \beta}-\frac{2\omega_{rr}}{\omega}.
\end{split}
\end{equation*}

\noindent{\bf Case 2.} The projections of $v$ and $w$ into $T_xS^n$ span a $1$-dimensional subspace. This subspace is orthogonal to the directions $\p_1,\ldots,\p_{n-p}$, and so without loss of generality, we can assume that it is spanned by $\p_{n-p+1}$. Setting $k=n-p+1$ for brevity, this means that the vectors $v$ and $w$ have the form
\begin{align*}
v&=v_r\p_r+v_t\p_t+v_k\p_k&w&=w_r\p_r+w_t\p_t+w_k\p_k
\end{align*}
Using Equation \ref{fullriemformainmetric}, for $1\leq i\leq n-p$, we obtain
\begin{equation*}
\begin{split}
R(v,\partial_i,\partial_i,v)\,\,\,& =\, v_k^2\beta^2(1-\beta_r^{2})-v_r^2\beta\beta_{rr}-v_t^2\beta\beta_r\omega\omega_r \,\,\,\,\mathrm{ and } \\
R(w,\partial_i,\partial_i,w) &=\, w_k^2\beta^2(1-\beta_r^{2})-w_r^2\beta\beta_{rr}-w_t^2\beta\beta_r\omega\omega_r.\\
\end{split}
\end{equation*}
Since $v$ and $w$ are orthogonal to $\partial_i$ and $\|\partial_i\| = \beta\|\partial_i\|_F = \beta$, the sectional curvatures are then
\begin{equation*}
\begin{split}
K(v,\partial_i) &= {1\over\|v\|^2}\left(v_k^2(1-\beta_r^{2})-v_r^2{\beta_{rr}\over\beta}-v_t^2{\beta_r\over\beta}\omega\omega_r\right)\,\,\,\,\mathrm{ and }\\
K(w,\partial_i) &= {1\over\|w\|^2}\left(w_k^2(1-\beta_r^{2})-w_r^2{\beta_{rr}\over\beta}-w_t^2{\beta_r\over\beta}\omega\omega_r\right).\\
\end{split}
\end{equation*}
On the other hand,
\begin{equation*}
R(v,w,w,v) =-\omega\omega_{rr}(v_tw_r-v_rw_t)^2-(v_kw_r-w_kv_r)^2\beta\beta_{rr}-(v_kw_t-w_kv_t)^2\beta\beta_r\omega\omega_r,
\end{equation*}
and since $v$, $w$ are orthogonal, we get
$$K(v,w)=-{1\over\|v\|^2\|w\|^2}\left(\omega\omega_{rr}(v_tw_r-v_rw_t)^2+(v_kw_r-w_kv_r)^2\beta\beta_{rr}+(v_kw_t-w_kv_t)^2\beta\beta_r\omega\omega_r\right)$$

The $(p,n+2)$-intermediate scalar curvature is then given by
\begin{equation*}
\begin{split}
s_{p,n+2}(P) &=2\sum_{i<j} K_{ij}+2\sum_{i=1}^{n-p}K(v,\p_i)+2\sum_{i=1}^{n-p}K(w,\p_i)+2K(v,w)\\
&=(n-p)(n-p-1)\frac{1-\beta_r^2}{\beta^2}\\
&+{2(n-p)\over\|v\|^2}\left(v_k^2(1-\beta_r^{2})-v_r^2{\beta_{rr}\over\beta}-v_t^2{\beta_r\over\beta}\omega\omega_r\right)\\
&+{2(n-p)\over\|w\|^2}\left(w_k^2(1-\beta_r^{2})-w_r^2{\beta_{rr}\over\beta}-w_t^2{\beta_r\over\beta}\omega\omega_r\right)\\
&-{2\over\|v\|^2\|w\|^2}\left(\omega\omega_{rr}(v_tw_r-v_rw_t)^2+(v_kw_r-w_kv_r)^2\beta\beta_{rr}+(v_kw_t-w_kv_t)^2\beta\beta_r\omega\omega_r\right)
\end{split}
\end{equation*}

\noindent{\bf Case 3.} The projections of $v$ and $w$ onto $T_xS^n$ span a $2$-dimensional subspace. This subspace is orthogonal to the directions $\p_1,\ldots,\p_{n-p}$, and so without loss of generality, we can assume that is is spanned by $\p_{n-p+1}$ and $\p_{n-p+2}$. Setting $k=n-p+1$ so that $k+1=n-p+2$, this means that the vectors $v$ and $w$ have the form
\begin{align*}
v&=v_r\p_r+v_t\p_t+v_k\p_k+v_{k+1}\p_{k+1},&w&=w_r\p_r+w_t\p_t+w_k\p_k+w_{k+1}\p_{k+1}.
\end{align*}
Using Equation \ref{fullriemformainmetric}, for $1\leq i\leq n-p$, we get
\begin{equation*}
\begin{split}
R(v,\partial_i,\partial_i,v) \,\,\,&= \,(v_k^2+v_{k+1}^2)\beta^2(1-\beta_r^{2})-v_r^2\beta\beta_{rr}-v_t^2\beta\beta_r\omega\omega_r, \,\,\,\,\mathrm{ and }\\
R(w,\partial_i,\partial_i,w) &=\, (w_k^2+w_{k+1}^2)\beta^2(1-\beta_r^{2})-w_r^2\beta\beta_{rr}-w_t^2\beta\beta_r\omega\omega_r.\\
\end{split}
\end{equation*}
Since $v$ and $w$ are orthogonal to $\partial_i$ and $\|\partial_i\| = \beta\|\partial_i\|_F = \beta$, then we obtain
\begin{equation*}
\begin{split}
K(v,\partial_i) &= {1\over\|v\|^2}\left((v_k^2+v_{k+1}^2)(1-\beta_r^{2})-v_r^2{\beta_{rr}\over\beta}-v_t^2{\beta_r\over\beta}\omega\omega_r\right),\,\,\,\,\mathrm{ and }\\
K(w,\partial_i) &= {1\over\|w\|^2}\left((w_k^2+w_{k+1}^2)(1-\beta_r^{2})-w_r^2{\beta_{rr}\over\beta}-w_t^2{\beta_r\over\beta}\omega\omega_r\right).\\
\end{split}
\end{equation*}
On the other hand, we have
\begin{equation*}
\begin{split}
R(v,w,w,v) &=-\omega\omega_{rr}(v_tw_r-v_rw_t)^2+(v_kw_{k+1}-v_{k+1}w_k)^2\beta^2(1-\beta_r^2)\\
&\,\,\,\,\,\,\,-[(v_kw_r-w_kv_r)^2+(v_{k+1}w_r-w_{k+1}v_r)^2]\beta\beta_{rr}\\
&\,\,\,\,\,\,\,-[(v_kw_t-w_kv_t)^2+(v_{k+1}w_t-w_{k+1}v_t)^2]\beta\beta_r\omega\omega_r,
\end{split}
\end{equation*}
and since $v$, $w$ are orthogonal, we get
\begin{equation*}
\begin{split}
K(v,w)&={1\over\|v\|^2\|w\|^2}\Bigl[-\omega\omega_{rr}(v_tw_r-v_rw_t)^2+(v_kw_{k+1}-v_{k+1}w_k)^2\beta^2(1-\beta_r^2)\\
&\hspace{1in}\null-[(v_kw_r-w_kv_r)^2+(v_{k+1}w_r-w_{k+1}v_r)^2]\beta\beta_{rr}\\
&\hspace{1in}\null-[(v_kw_t-w_kv_t)^2+(v_{k+1}w_t-w_{k+1}v_t)^2]\beta\beta_r\omega\omega_r\Bigr].
\end{split}
\end{equation*}

The $(p,n+2)$-intermediate scalar curvature is then given by
\begin{equation*}
\begin{split}
s_{p,n+2}(P) &=2\sum_{i<j} K_{ij}+2\sum_{i=1}^{n-p}K(v,\p_i)+2\sum_{i=1}^{n-p}K(w,\p_i)+2K(v,w)\\
&=(n-p)(n-p-1)\frac{1-\beta_r^2}{\beta^2}\\
&+{2(n-p)\over\|v\|^2}\left((v_k^2+v_{k+1}^2)(1-\beta_r^{2})-v_r^2{\beta_{rr}\over\beta}-v_t^2{\beta_r\over\beta}\omega\omega_r\right)\\
&+{2(n-p)\over\|w\|^2}\left((w_k^2+w_{k+1}^2)(1-\beta_r^{2})-w_r^2{\beta_{rr}\over\beta}-w_t^2{\beta_r\over\beta}\omega\omega_r\right)\\
&+{2\over\|v\|^2\|w\|^2}\biggl[-\omega\omega_{rr}(v_tw_r-v_rw_t)^2+(v_kw_{k+1}-v_{k+1}w_k)^2\beta^2(1-\beta_r^2)\\
&\hspace{1in}\null-[(v_kw_r-w_kv_r)^2+(v_{k+1}w_r-w_{k+1}v_r)^2]\beta\beta_{rr}\\
&\hspace{1in}\null-[(v_kw_t-w_kv_t)^2+(v_{k+1}w_t-w_{k+1}v_t)^2]\beta\beta_r\omega\omega_r\biggr].
\end{split}
\end{equation*}

Since we know the warping functions $\beta$ and $\omega$ are strictly positive by definition, and most coefficients in our final formulas for the intermediate scalar curvatures involve nonnegative squared terms, we summarize these three cases in the following proposition.

\begin{proposition}\label{pcurvatureofwarpedproduct}
Let $M=B\times F$ with $B=(0,b_1)\times(0,b_2)$ and $F=S^n$, having a metric of the form
$$g=dr^2+\omega(r,t)^2dt^2+\beta(r)^2ds_n^2$$
where $\beta:(0,\beta_1)\to(0,\infty)$ and $\omega:(0,b_1)\times(0,b_2)\to(0,\infty)$ are smooth warping functions. If $x\in M$ and $P$ is a $p$-plane in $T_xM$ for $p\in\{0,\cdots,n\}$, then the $(p,n+2)$-intermediate scalar curvature of $P$ has the form
\begin{equation*}
s_{p, n+2}(P)=(n-p)(n-p-1)\frac{1-\beta_r^2}{\beta^2}+A^2(1-\beta_r)^2-B^2\omega\omega_{rr}-C^2\beta\beta_{rr}-D^2\beta\beta_r\omega\omega_r,
\end{equation*}
for some real-valued functions $A=A(P,\beta)$, $B=B(P,\omega)$, $C=C(P,\beta)$, and $D=D(P,\beta,\omega)$.
\end{proposition}

\section{Standard Metrics on the Sphere and the Disk}\label{s5}
\setcounter{figure}{0}

In this section we recall some well known metrics on the disk and sphere. In particular, we recall the so-called {\em torpedo} and {\em boot metrics}. Such metrics are described in detail in section 3 of \cite{walshNY} although in the context of positive scalar curvature. Here we will establish conditions whereby these metrics have positive $(p,n)$-intermediate scalar curvature for appropriate $p\in\{0,1,\cdots n-2\}$.

\subsection{Introducing the Metrics}
Before we get into a formal construction of the metrics, we give a very brief description with the aid of Figures \ref{torpcorners1} and \ref{torpcorners3} below. An important point to note is that each space will be topologically the disk $D^{n+2}$, but will be distinguished by their metrics.

A {\em $(\delta,\lambda)$-torpedo metric} on a disk, $D^{n+2}$, (where $n\geq 0$) is a metric which takes the form of a length $\lambda$-cylinder of a round $(n+1)$-sphere of radius $\delta$ near the boundary of the disk, before closing up as a round $(n+2)$-dimensional hemisphere at the center. We denote such a metric, $g_{\tor}^{n+2}(\delta)_{\lambda}$ and it is depicted by the first picture in Fig. \ref{torpcorners1}.

Restricting such a metric to an upper half-disk  $D_{+}^{n+2}$ results in a {\em half-$(\delta,\lambda)$-torpedo metric}, denoted $g_{\tor+}^{n+2}(\delta)_{\lambda}$. This is the second picture in Fig. \ref{torpcorners1}.

By carefully gluing a half-torpedo metric, $g_{\tor+}^{n+2}(\delta)_{\lambda}$ on $D_{+}^{n+2}$, to a cylinder of torpedo metrics, $g_{\tor}^{n+1}(\delta)_{\lambda}+dt^{2}$ on $D^{n+1}\times [0,1]$ along a $D^{n+1}$ contained in the boundary of $D_+^{n+2}$, we obtain a metric denoted $g_{\toe}^{n+2}(\delta)_{\lambda_1, \lambda_2}$, on the manifold with corners, $D_{\mathrm{stretch}}^{n+2}$, obtained by attaching and smoothing the underlying manifolds $D_{+}^{n+2}$ and $D^{n+1}\times [0,1]$. We will refer to the metric $g_{\toe}^{n+2}(\delta)_{\lambda_1, \lambda_2}$ as the {\em toe metric} for its shape and its role in the next construction. This is the third picture in Figure \ref{torpcorners1}. 

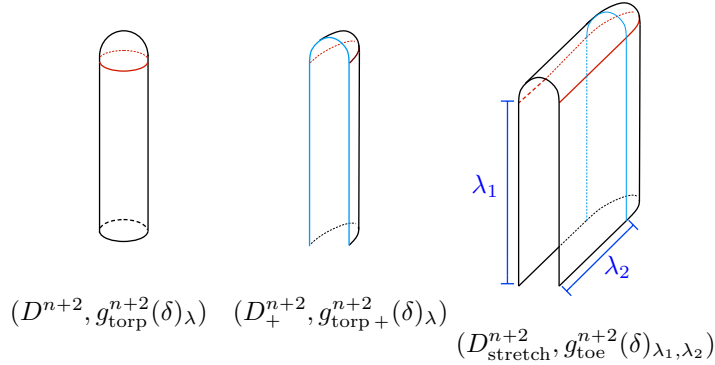
\begin{figure}[!htbp]
\vspace{1cm}
\hspace{4cm}
\begin{picture}(0,0)
\includegraphics{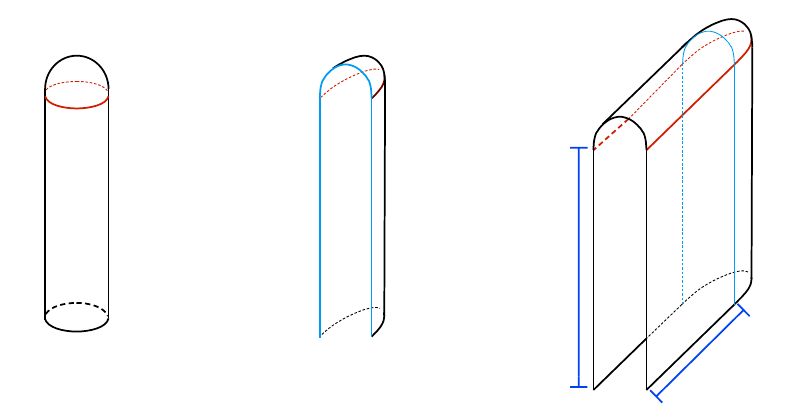}%
\end{picture}
\setlength{\unitlength}{3947sp}
\begin{picture}(5079,1559)(1902,-7227)
\put(4400,-6500){\makebox(0,0)[lb]{\smash{{\SetFigFont{10}{8}{\rmdefault}{\mddefault}{\updefault}{\color[rgb]{0,0,1}$\lambda_1$}%
}}}}
\put(5230,-7010){\makebox(0,0)[lb]{\smash{{\SetFigFont{10}{8}{\rmdefault}{\mddefault}{\updefault}{\color[rgb]{0,0,1}$\lambda_2$}%
}}}}
\put(1500,-7300){\makebox(0,0)[lb]{\smash{{\SetFigFont{10}{8}{\rmdefault}{\mddefault}{\updefault}{\color[rgb]{0,0,0}$(D^{n+2},g_{\tor}^{n+2}(\delta)_{\lambda})$}%
}}}}
\put(2900,-7300){\makebox(0,0)[lb]{\smash{{\SetFigFont{10}{8}{\rmdefault}{\mddefault}{\updefault}{\color[rgb]{0,0,0}$(D^{n+2}_+,g_{\tor+}^{n+2}(\delta)_{\lambda})$}%
}}}}
\put(4300,-7500){\makebox(0,0)[lb]{\smash{{\SetFigFont{10}{8}{\rmdefault}{\mddefault}{\updefault}{\color[rgb]{0,0,0}$(D^{n+2}_{\mathrm{stretch}},g_{\toe}^{n+2}(\delta)_{\lambda_1,\lambda_2})$}%
}}}}
\end{picture}%
\vspace{1cm}
\caption{Various metrics on the disk $D^{n+2}$.}
\label{torpcorners1}
\end{figure}

Finally, we introduce an {\em $(n+2)$-dimensional $\delta$-boot metric}. Briefly, this metric is constructed in 4 steps as follows.

\noindent{\bf Step 1.} Beginning with some torpedo metric, $g_{\tor}^{n+1}(\delta)_{\lambda}$, trace out a cylinder of torpedo metrics before bending the cylinder around an angle of $\frac{\pi}{2}$ to finish as a Riemannian cylinder perpendicular to the first part of the cylinder in the direction suggested by the rightmost image of Figure \ref{torpcorners3}. The resulting object has two cylindrical ends with different metrics. One of the form $dr^{2}+g_{\tor}^{n+1}(\delta)_{\lambda}$ and the other $dt^{2}+g_{\tor}^{n+1}(\delta)_{\lambda}$, where $r$ and $t$ are orthogonal coordinates depicted in Figure \ref{torpcorners3}. 

\noindent{\bf Step 2.} In order to control any negative sectional curvatures arising from the bend, we control the bending with a parameter $\Lambda>0$. Essentially, the bending takes place along a quarter-circle of radius $\Lambda>0$. A large choice of $\Lambda$ ensures that negative curvatures arising from the bend are small.

\noindent{\bf Step 3.} Away from the ``caps" of the torpedos, this metric takes the form $dr^{2}+dt^{2}+\delta^{2}ds_{n}^{2}$. This part can easily be extended to incorporate the corner depicted in the rightmost image of Fig. \ref{torpcorners1} and so that the necks of the torpedo ``ends" have any desired lengths, $l_1$ and $l_4$. These distance along with $\Lambda$ determine the distances $l_2$ and $l_3$, which are pictured in Figure \ref{torpcorners3}.

\noindent{\bf Step 4.} Finally, we smoothly ``cap-off" the cylindrical end which takes the form $dt_{1}^{2}+g_{\tor}^{n-1}(\delta)_{l_1}$, by attaching a half-torpedo metric, $g_{\tor+}^{n+2}(\delta)_{l_1}$. This is the so-called ``toe" of the boot metric.

The resulting metric is denoted $g_{\boot}^{n+2}(\delta)_{\Lambda, \bar{l}}$ where $\Lambda>0$ is the bending constant discussed above and $\bar{l}=(l_1, l_2, l_3,
l_4)\in\R_+^{4}$ determines the various neck-lengths. While the choices of $l_1$ and $l_4$ are arbitrary, the constants $l_2$ and $l_3$, as mentioned above, are determined by $\Lambda, l_1$ and $l_4$.

\begin{figure}[!htbp]
\vspace{3cm}
\hspace{5cm}
\begin{picture}(0,0)
\includegraphics{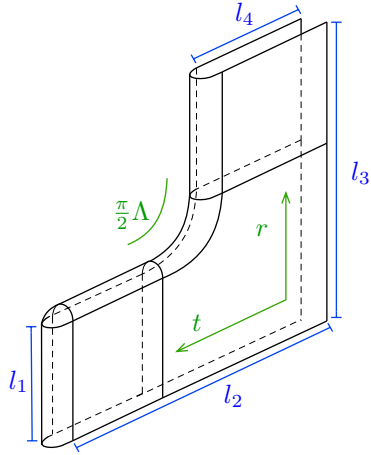}%
\end{picture}
\setlength{\unitlength}{3947sp}
\begin{picture}(5079,1559)(5922,-7137)
\put(7500,-5700){\makebox(0,0)[lb]{\smash{{\SetFigFont{10}{8}{\rmdefault}{\mddefault}{\updefault}{\color[rgb]{0,0.6,0}$r$}%
}}}}
\put(7100,-6300){\makebox(0,0)[lb]{\smash{{\SetFigFont{10}{8}{\rmdefault}{\mddefault}{\updefault}{\color[rgb]{0,0.6,0}$t$}%
}}}}
\put(6600,-5600){\makebox(0,0)[lb]{\smash{{\SetFigFont{10}{8}{\rmdefault}{\mddefault}{\updefault}{\color[rgb]{0,0.6,0}$\frac{\pi}{2}\Lambda$}
}}}}
\put(5950,-6660){\makebox(0,0)[lb]{\smash{{\SetFigFont{10}{8}{\rmdefault}{\mddefault}{\updefault}{\color[rgb]{0,0,1}$l_1$}%
}}}}
\put(7300,-6750){\makebox(0,0)[lb]{\smash{{\SetFigFont{10}{8}{\rmdefault}{\mddefault}{\updefault}{\color[rgb]{0,0,1}$l_2$}%
}}}}
\put(8100,-5350){\makebox(0,0)[lb]{\smash{{\SetFigFont{10}{8}{\rmdefault}{\mddefault}{\updefault}{\color[rgb]{0,0,1}$l_3$}%
}}}}
\put(7370,-4350){\makebox(0,0)[lb]{\smash{{\SetFigFont{10}{8}{\rmdefault}{\mddefault}{\updefault}{\color[rgb]{0,0,1}$l_4$}%
}}}}
\end{picture}%
\caption{The boot metric $g_{\boot}^{n+2}(\delta)_{\Lambda, \bar{l}}$ }
\label{torpcorners3}
\end{figure}

In the remainder of this section we will establish some results about the $(p,n+2)$-intermediate scalar curvature of these metrics. This requires a somewhat more detailed description. 

\subsection{The Torpedo Metric}
We consider a pair of smooth functions $\alpha, \beta:[0,b]\rightarrow [0,\infty)$, where $b>0$, which satisfy the following conditions.
\begin{enumerate}

\item[]\begin{equation}
\begin{array}{clll}\label{beta0}
\mathrm{(i)} \quad &\beta(r)>0, \text{ for all } r\in(0,b), &&\\
\mathrm{(ii)} \quad &\beta(0)=0, \quad\beta'(0)=1, \quad\beta^{(even)}(0)=0,&&\\
\mathrm{(iii)} \quad &\beta(b)=0, \quad\beta'(b)=-1, \quad\beta^{(even)}(b)=0,\text{ and}&&
\end{array}
\end{equation}
\item[]\begin{equation}\label{alpha0}
\alpha(r)=\int_r^{b\over 2}\sqrt{1-\beta'(u)^{2}}du
\end{equation}

\end{enumerate}

The important point to note here is that $\alpha$ and $\beta$ satisfy $(\alpha')^2+(\beta')^2=1$. In particular, if $\beta(r)=\sin r$ on $(0,b)=(0,\pi)$, then $\alpha(r)=\cos r$. We now consider the map, $F_{\beta}$, defined by:
\begin{equation*}
\begin{split}
F_\beta:(0,b)\times{S^{n+1}}&\longrightarrow\mathbb{R}^{n+2}\times\mathbb{R},\\
(r,\theta)\,\,\,\,\,\,\,\,\,\,&\longmapsto(\beta(r)\theta, \alpha(r)).
\end{split}
\end{equation*}
\begin{proposition}\label{warpingfunctionleadstoembedding}\cite[Prop. 3.1]{walshNY}
For any smooth functions $\alpha, \beta:[0,b]\rightarrow [0,\infty)$ satisfying the conditions laid out in \ref{alpha0} and \ref{beta0}, the map $F_\beta$ above is an embedding.
\end{proposition}
\noindent Pulling back the Euclidean metric on $\mathbb{R}^{n+2}\times\mathbb{R}$ via $F_{\beta}$ induces a metric, $g_{\beta}$, which we compute to be
\begin{equation*}\label{gbeta}
\begin{split}
g_{\beta}:=&F_{\beta}^{*}(dx_1^{2}+dx_{2}^{2}+\cdots+dx_{n+2}^{2}+dx_{n+3}^{2})\\
=&dr^{2}+\beta(r)^{2}ds_{n+1}^{2},
\end{split}
\end{equation*}
where $ds_{n+1}^{2}$ is the standard round metric of radius $1$ on $S^{n+1}$. The following proposition is proved in Chapter 1, Section 3.4 of Petersen \cite{petersen}.
\begin{proposition}\cite{petersen}
Provided the smooth function $\beta:[0,b]\rightarrow [0,\infty)$ satisfies the conditions laid out in \ref{beta0}, the metric $g_{\beta}$ extends uniquely to a rotationally symmetric metric on $S^{n+2}$. Furthermore, if we drop condition (iii) of \ref{beta0} and simply insist that $\beta(b)>0$, this metric is now a smooth rotationally symmetric metric on the disk $D^{n+2}$.
\end{proposition}
\noindent In particular, by setting $\beta(t)=\delta\sin{r\over\delta}$ for $r\in[0,\delta\pi]$, we obtain for $g_{\beta}$ the standard round metric of radius $\delta$ on $S^{n+2}$.

\begingroup\singlespacing
Let $\mu:\R\to[0,1]$ be the function
$$\mu(r)=\begin{cases}0&r\leq 0\cr {e^{-1/r}\over e^{-1/r}+e^{-1/(1-r)}}&r\in(0,1)\cr 1&r\geq 1\end{cases}$$
\endgroup
that smoothly transitions from $0$ to $1$. For any $\delta>0$ and $\lambda\geq 0$, let $\eta_{\delta,\lambda}:[0,\frac{\pi\delta}{2}+\lambda]\rightarrow[0,1]$ be the smooth function with derivative
$$\eta_{\delta,\lambda}'(r) = \cos\left({r\over\delta}\right)\mu\left(2-{4r\over\delta\pi}\right).$$
This function satisfies conditions (i.) and (ii.) of (\ref{beta0}) as well as the following:
\begin{enumerate}\label{torpcond}
\item[(i)] $\eta_{\delta,\lambda}(r)=\delta\sin{r\over \delta}$ when $r\leq {\pi\delta\over 4}$,
\item[(ii)] $\eta_{\delta, \lambda}(r)=C\delta$ when $r\geq\frac{\pi\delta}{2}$ for $C\approx 0.916$,
\item[(iii)] $\eta_{\delta,\lambda}(r)\leq\delta\sin{r\over\delta}$ and $\eta'_{\delta,\lambda}(r)\leq\cos{r\over\delta}$ for ${\pi\delta\over 4}<r<{\pi\delta\over 2}$,
\item[(iv)] $\eta''_{\delta,\lambda}(r)\leq 0$, and
\item[(v)] the $k^{\mathrm{th}}$ derivative at $\frac{\pi\delta}{2}$, $\eta_{\delta,\lambda}^{(k)}(\frac{\pi\delta}{2})=0$ for all $k\geq 1$.
\end{enumerate} 

The function $\eta_{\delta,\lambda}$ is known as a {\em torpedo function}. As it satisfies conditions (i) and (ii) of \ref{beta0} and has $\eta_{\delta,\lambda}( \frac{\pi}{2}+\lambda)>0$, it gives rise to a smooth metric on $D^{n+2}$. The resulting metric is called a {\em torpedo metric of radius $\delta$} and {\em neck length} $\lambda$ (or $(\delta,\lambda)$-torpedo metric). It is denoted $g_{\tor}^{n+2}(\delta)_{\lambda}$ and given by the formula:
$$g_{\tor}^{n+2}(\delta)_{\lambda}=dr^{2}+\eta_{\delta, \lambda}(r)^{2}ds_{n+1}^{2},$$ 
where $r\in [0,\frac{\pi\delta}{2}+\lambda]$.
Such a metric is rotationally symmetric metric on the disk $D^{n+2}$ and roughly, a round hemisphere of radius $\delta$ near the center of the disk and takes a radius $\delta$ cylindrical form on the annular region where $r\in[\frac{\pi\delta}{2}, \frac{\pi\delta}{2}+\lambda]$; see Fig. \ref{torpfunc}. 

\begin{figure}[!htbp]
\vspace{-1cm}
\hspace{3cm}
\begin{picture}(0,0)
\includegraphics{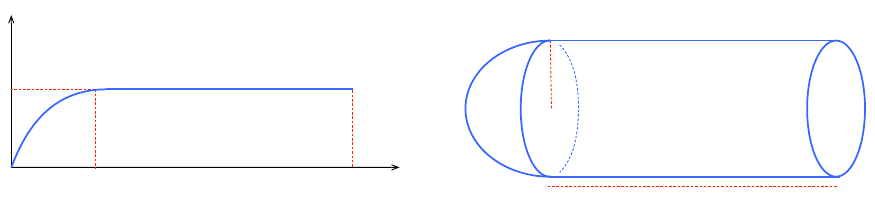}%
\end{picture}
\setlength{\unitlength}{3947sp}%
\begin{picture}(5079,1559)(1902,-7227)
\put(1850,-7250){\makebox(0,0)[lb]{\smash{{\SetFigFont{10}{8}{\rmdefault}{\mddefault}{\updefault}{\color[rgb]{0,0,0}$0$}%
}}}}
\put(1650,-6750){\makebox(0,0)[lb]{\smash{{\SetFigFont{10}{8}{\rmdefault}{\mddefault}{\updefault}{\color[rgb]{0,0,0}$C\delta$}%
}}}}
\put(4500,-6680){\makebox(0,0)[lb]{\smash{{\SetFigFont{10}{8}{\rmdefault}{\mddefault}{\updefault}{\color[rgb]{1,0,0}$C\delta$}%
}}}}
\put(2260,-7250){\makebox(0,0)[lb]{\smash{{\SetFigFont{10}{8}{\rmdefault}{\mddefault}{\updefault}{\color[rgb]{0,0,0}$\delta\frac{\pi}{2}$}%
}}}}
\put(3400,-7250){\makebox(0,0)[lb]{\smash{{\SetFigFont{10}{8}{\rmdefault}{\mddefault}{\updefault}{\color[rgb]{0,0,0}$\delta\frac{\pi}{2}+\lambda$}%
}}}}
\put(5100,-7320){\makebox(0,0)[lb]{\smash{{\SetFigFont{10}{8}{\rmdefault}{\mddefault}{\updefault}{\color[rgb]{1,0,0}$\lambda$}%
}}}}
\end{picture}%
\caption{A torpedo function $\eta_{\delta, \lambda}$ and the resulting torpedo metric $g_{\tor}^{n+2}(\delta)_{\lambda}$ on the disk.}
\label{torpfunc}
\end{figure}
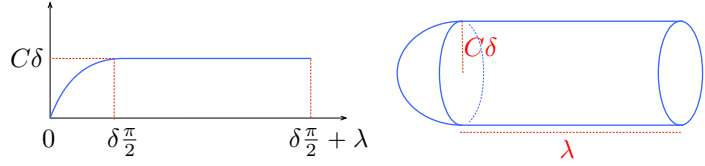

\begin{proposition}\label{positivity of torpedo curvature}
Suppose $n\geq 1$ and $p$ satisfies $0\leq p\leq n-1$. For any $\delta>0, \lambda\geq 0$ the metric $g_{\tor}^{n+2}(\delta)_{\lambda}$ has positive $(p,n+2)$-intermediate scalar curvature. Moreover, this curvature can be bounded below with an arbitrarily large positive constant by choosing $\delta$ sufficiently small.
\end{proposition}
\begin{proof}
Excluding the point at $r=0$, the metric $g_{\tor}^{n+2}(\delta)_\lambda$ is the warped product metric $dr^2+\beta^2ds_{n+1}^2$ where $\beta=\eta_{\delta,\lambda}$. Recall the Riemann curvatures of a warped product from Equation \ref{Riemanncurvatureofwarped}. Here the base is the one-dimensional $B=(0,{\pi\delta\over 2})$ and the fibers are spheres $F=S^{n+1}$ with constant sectional curvature equal to $1$. Therefore Equation \ref{Riemanncurvatureofwarped} reduces to
\begin{equation*}
R(v,w,w,v)  =  \sum_{i<j} (v_iw_j-v_jw_i)^2\beta^{2}\Bigl(1-(\beta')^{2}\Bigr)-\sum_{i} (v_r w_i - v_i w_r)^2\beta\beta''.
\end{equation*}

If $P$ is a $p$-plane in $T_xM$, then $P^\perp$ has dimension $n-p+2$. Since the dimension of $P^\perp+T_xS^{n+1}$ can be at most the dimension of $T_xM$,
\begin{equation*}
\begin{split}
\dim(P^\perp\cap T_xS^{n+1})&=\dim(P^\perp)+\dim(T_xS^{n+1})-\dim(P^\perp+T_xS^{n+1})\\
&\geq (n-p+2)+(n+1)-(n+2)\\
&=n-p+1.
\end{split}
\end{equation*}

Therefore, we have two cases.

\noindent{\bf Case 1.} The projection of $P^\perp$ into $T_xM$ has dimension $n-p+1$. This means there is a direction in $P^\perp$ orthogonal to $T_xM$, and so we can take $\{\partial_1,\ldots,\partial_{n-p+1},\partial_r\}$ as an orthonormal basis for $P^\perp$ where $\partial_1,\ldots,\partial_{n-p+1}$ are tangent to the sphere. Therefore the $(p,n+2)$-intermediate curvature is given by

 \begin{equation}\label{torpedo ISC case 1}
\begin{split}
s_{p,n+2}(P) &= \sum_{i,j=1}^{n-p+1} K(\p_i,\p_j) + 2\sum_{i=1}^{n-p+1} K(\p_i,\p_r)\\
&=(n-p+1)(n-p){1-(\beta')^2\over\beta^2}-2(n-p+1){\beta''\over\beta}.\\
\end{split}
\end{equation}

\noindent{\bf Case 2.} The projection of $P^\perp$ into $T_xM$  has dimension $n-p+2$. Then we can take an orthonormal basis $\{\partial_1,\ldots,\partial_{n-p+1},v\}$ for $P^\perp$ where $\partial_1,\ldots,\partial_{n-p+1}$ are tangent to the sphere. We can write $v=v_{n-p+2}\partial_{n-p+2}+v_r\partial_r$ where $\partial_{n-p+2}$ is a unit vector in $T_xS^{n+1}$ orthogonal to the other $\partial_i$. Therefore the $(p,n+2)$-intermediate sectional curvature is given by

\begin{equation}\label{torpedo ISC case 2}
\begin{split}
s_{p,n+2}(P) &= \sum_{i,j=1}^{n-p+1} K(\p_i,\p_j) + 2\sum_{i=1}^{n-p+1} K(\p_i,v)\\
&=(n-p+1)(n-p){1-(\beta')^2\over\beta^2}\\
&\quad\null+{2\over||v||^2}(n-p+1)\left[v_{n-p+2}^2\bigl(1-(\beta')^{2}\bigr)-v_r^2{\beta''\over\beta}\right].\\
\end{split}
\end{equation}

Since $\beta>0$, $(\beta')^2\leq 1$, and $\beta''\leq 0$, we have that $-\frac{\beta''}{\beta}\geq 0$ and $1-(\beta')^2\geq 0$, so the second terms of Equations \ref{torpedo ISC case 1} and \ref{torpedo ISC case 2} are nonnegative. Consider $K_{ij}={1-(\beta')\over\beta^2}$. For $0<r\leq{\pi\delta\over 4}$, we have $\beta(r)=\eta_{\delta,\lambda}(r)=\delta\sin{r\over\delta}$ and $\beta'(r)=\cos{r\over\delta}$. Therefore, $$K_{ij} = {1-\cos^2({r\over\delta})\over\delta^2\sin^2({r\over\delta})}={\sin^2({r\over\delta})\over\delta^2\sin^2({r\over\delta})}={1\over\delta^2}.$$
On the other hand, when $r\geq{\pi\delta\over 2}$, we have $\beta(r)=\eta_{\delta,\lambda}(r)=C\delta$ and $\beta'(r)=0$. Hence in this case, as $C<1$, we immediately have $K_{ij}={1\over C^2\delta^2}>{1\over\delta^2}$. In the transition region ${\pi\delta\over 4}<r<{\pi\delta\over 2}$, since  $\eta_{\delta,\lambda}(r)\leq\delta\sin{r\over\delta}$ and $\eta'_{\delta,\lambda}(r)\leq\cos{r\over\delta}$, then $$K_{ij} \geq {1-\cos^2({r\over\delta})\over\delta^2\sin^2({r\over\delta})}={1\over\delta^2}.$$

 This means that as long as $r>0$, the second terms of Equations \ref{torpedo ISC case 1} and \ref{torpedo ISC case 2} are non-negative, while the first terms consist of $K_{ij}\geq{1\over\delta^2}$. Since $p\leq n-1$, then $n-p>0$, and so the first terms are both strictly positive and can be bounded below by an arbitrarily large positive constant by choosing $\delta$ sufficiently small.

This leaves just the point at $r=0$. Since $\beta(r)=\eta_{\delta,\lambda}(r)=\delta\sin{r\over\delta}$ when $r$ is near $0$, and the sectional curvature is continuous, we can compute the curvature at this point using limits. Note that $\beta(0)=0$, $\beta'(0)=1$, and $\beta''(0)=0$. Since $\beta'''(r)=-{1\over\delta^2}\cos{r\over\delta}$ for $r$ near $0$, we also have $\beta'''(0)=-{1\over\delta^2}$. Then we apply L'H\^opital, and we get the sectional curvatures
\begin{equation*}
\lim\limits_{r\to0}\frac{1-(\beta')^2}{\beta^2}=\lim\limits_{r\to0}-\frac{2\beta''\beta'}{2\beta'\beta}=\lim\limits_{r\to0}-\frac{\beta''}{\beta}=\lim\limits_{r\to0}-\frac{\beta'''}{\beta'}={1\over\delta^2}.
\end{equation*}
Hence at $r=0$, all sectional curvatures are in fact equal to the positive value $1\over\delta^2$. Therefore $s_{p,n+2}(P)=2(n-p+2)(n-p+1){1\over\delta^2}$ is positive, so that $g_{\tor}^{n+2}(\delta)_\lambda$ has positive $(p,n+2)$-intermediate scalar curvature that can be bounded below by an arbitrarily large positive constant by choosing $\delta$ sufficiently small. 
\end{proof}

We also need to consider the product of a torpedo metric with an interval in the construction of the boot metric.

\begin{proposition}\label{positivity of torpedo curvature crossed with an interval}
Suppose $n\geq 1$ and $p$ satisfies $0\leq p\leq n-2$. For any $\delta>0, \lambda\geq 0$ the metric $g_{\tor}^{n+1}(\delta)_{\lambda}+dt^2$ on the product $D^{n+1}\times[0,1]$ has positive $(p,n+2)$-intermediate scalar curvature.
\end{proposition}
\begin{proof}
Set $g=g_{\tor}^{n+1}(\delta)_{\lambda}+dt^2$. Excluding the point at $r=0$, the metric $g_{\tor}^{n+1}(\delta)_\lambda$ is the warped product metric $dr^2+\beta^2ds_n^2$ where $\beta=\eta_{\delta,\lambda}$. Therefore we can consider $g$ for $r>0$ to be the warped product $g=g_B+\beta(r)^2ds_n^2$ where $g_B$ is the product metric $dt^2+dr^2$ on $B=(0,{\pi\delta\over2}+\lambda]\times[0,1]$. Recall the Riemann curvatures of a warped product from Equation \ref{Riemanncurvatureofwarped}. Here the base is the flat two-dimensional $B$ and the fibers are spheres $F=S^{n+1}$ with constant sectional curvature equal to $1$. Therefore Equation \ref{Riemanncurvatureofwarped} reduces to

\begin{equation*}
R(v,w,w,v)  =  \sum_{i<j} (v_iw_j-v_jw_i)^2\beta^{2}(1-\beta_r^2)-\sum_{i} (v_rw_i-v_iw_r)^2\beta\beta_{rr}
\end{equation*}

If $P$ is a $p$-plane in $T_xM$, then $P^\perp$ has dimension $n-p+2$. Since the dimension of $P^\perp+T_xS^n$ can be at most the dimension of $T_xM$,
\begin{equation*}
\begin{split}
\dim(P^\perp\cap T_xS^n)&=\dim(P^\perp)+\dim(T_xS^n)-\dim(P^\perp+T_xS^n)\\
&\geq (n-p+2)+n-(n+2)\\
&=n-p.
\end{split}
\end{equation*}

Therefore, we have three cases.

\noindent{\bf Case 1.} The projection of $P^\perp$ into $T_xM$ has dimension $n-p$. This means there are two directions in $P^\perp$ orthogonal to $T_xM$, and so we can take $\{\partial_1,\ldots,\partial_{n-p},\partial_r,\partial_t\}$ as an orthonormal basis for $P^\perp$ where $\partial_1,\ldots,\partial_{n-p}$ are tangent to the sphere. Therefore the $(p,n+2)$-intermediate curvature is given by

 \begin{equation}\label{torpedo ISC crossed interval case 1}
\begin{split}
s_{p,n+2}(P) &= \sum_{i,j=1}^{n-p} K(\p_i,\p_j) + 2\sum_{i=1}^{n-p} K(\p_i,\p_r)+2\sum_{i=1}^{n-p} K(\p_i,\p_t)+K(\p_t,\p_r)\\
&=(n-p)(n-p-1){1-\beta_r^2\over\beta^2}-2(n-p){\beta_{rr}\over\beta}.\\
\end{split}
\end{equation}

\noindent{\bf Case 2.} The projection of $P^\perp$ into $T_xM$  has dimension $n-p+1$. Then we can take an orthonormal basis $\{\partial_1,\ldots,\partial_{n-p},v,w\}$ for $P^\perp$ where $\partial_1,\ldots,\partial_{n-p+1}$ are tangent to the sphere and we can write
\begin{align*}
v&=v_{n-p+1}\partial_{n-p+1}+v_r\partial_r+v_t\partial_t,&
w&=w_{n-p+1}\partial_{n-p+1}+w_r\partial_r+w_t\partial_t,
\end{align*}
where $\partial_{n-p+1}$ is a unit vector in $T_xS^{n+1}$ orthogonal to the other $\partial_i$. Therefore the $(p,n+2)$-intermediate sectional curvature is given by

\begin{equation}\label{torpedo ISC crossed interval case 2}
\begin{split}
s_{p,n+2}(P) &= \sum_{i,j=1}^{n-p+1} K(\p_i,\p_j) + 2\sum_{i=1}^{n-p+1} K(\p_i,v)+2\sum_{i=1}^{n-p+1} K(\p_i,w) + 2K(v,w)\\
&=(n-p)(n-p-1){1-\beta_r^2\over\beta^2}\\
&\quad\null+{2(n-p)\over||v||^2}\left[v_{n-p+2}^2\bigl(1-\beta_r^{2}\bigr)-v_r^2{\beta_{rr}\over\beta}\right]\\
&\quad\null+{2(n-p)\over||w||^2}\left[w_{n-p+2}^2\bigl(1-\beta_r^{2}\bigr)-w_r^2{\beta_{rr}\over\beta}\right]\\
&\quad\null-{2\over||v||^2||w||^2}\left[ (v_rw_{n-p+1}-v_{n-p+1}w_r)^2\beta\beta_{rr}+ (v_rw_i-v_iw_r)^2\beta\beta_{rr}\right].
\end{split}
\end{equation}

\noindent{\bf Case 3.} The projection of $P^\perp$ into $T_xM$  has dimension $n-p+2$. Then there is an orthonormal basis $\{\partial_1,\ldots,\partial_{n-p},v,w\}$ for $P^\perp$ with $\partial_1,\ldots,\partial_{n-p+1}$ tangent to the sphere and 
\begin{align*}
v&=v_{n-p+2}\partial_{n-p+1}+v_{n+p+2}\partial_{n-p+2}+v_r\partial_r+v_t\partial_t,\\
w&=w_{n-p+2}\partial_{n-p+1}+w_{n+p+2}\partial_{n-p+2}+w_r\partial_r+v_t\partial_t,
\end{align*}
where $\partial_{n-p+1}$ and $\partial_{n-p+2}$ are orthogonal unit vectors in $T_xS^{n+1}$ orthogonal to the other $\partial_i$. Therefore the $(p,n+2)$-intermediate sectional curvature is given by

\begin{equation}\label{torpedo ISC crossed interval case 3}
\begin{split}
s_{p,n+2}(P) &= \sum_{i,j=1}^{n-p+1} K(\p_i,\p_j) + 2\sum_{i=1}^{n-p+1} K(\p_i,v)+2\sum_{i=1}^{n-p+1} K(\p_i,w) + 2K(v,w)\\
&=(n-p)(n-p-1){1-(\beta')^2\over\beta^2}\\
&\quad\null+{2(n-p)\over||v||^2}\left[(v_{n-p+1}^2+ v_{n-p+2}^2)(1-\beta_r^2)-v_r^2{\beta_{rr}\over\beta}\right]\\
&\quad\null+{2(n-p)\over||w||^2}\left[(w_{n-p+1}^2+ w_{n-p+2}^2)(1-\beta_r^2)-w_r^2{\beta_{rr}\over\beta}\right]\\
&\quad\null-{2\over||v||^2||w||^2}\biggl[ (v_{n-p+1}w_{n-p+2}-v_{n-p+2}w_{n-p+1})^2\beta^{2}(1-\beta_r^2)\\
&\quad\quad\quad\null-\big[(v_rw_{n-p+1}-v_{n-p+1}w_r)^2-(v_rw_{n-p+2}-v_{n-p+2}w_r)^2\big]\beta\beta_{rr}\biggr].
\end{split}
\end{equation}

In all three of Equations \ref{torpedo ISC crossed interval case 1}, \ref{torpedo ISC crossed interval case 2}, and \ref{torpedo ISC crossed interval case 3},  the  leading term is $$(n-p)(n-p-1){1-\beta_r^2\over\beta}.$$ Since $\beta>0$, $\beta_r^2\leq 1$, and $\beta_{rr}\leq 0$, all other terms are nonnegative. Using the same calculations from the proof of Proposition \ref{positivity of torpedo curvature}, we know that $K_{ij}>0$, and since $p\leq n-2$, this means that $s_{p,n+2}(P)>0$ in all three cases.

This leaves just the points where $r=0$. Again from the calculations in the proof of Proposition \ref{positivity of torpedo curvature}, we know that planes in the torpedo factor have a constant sectional curvature of $1\over\delta^2$. Since this is a standard Riemann product, this means the Riemann curvature is given by the standard formula \ref{standard riemann product formula}
\begin{equation*}
R(v,w,w,v)  =  R_B(\check{v}_{B}, \check{w}_{B},\check{w}_{B},\check{v}_{B})+ R_{F}(\hat{v}_{F}, \hat{w}_{F},\hat{w}_{F},\hat{v}_{F}).
\end{equation*}

Therefore, the $(p,n+2)$-intermediate scalar curvature of a $p$-plane $P$, with orthonormal basis $\{\partial_1,\ldots,\partial_{n-p+2}\}$ for $P^\perp$, is given by

\begin{equation}\label{torpedo ISC cross interval vertex}
\begin{split}
s_{p,n+2}(P) &= \sum_{i,j=1}^{n-p+2} |\partial_i^F\wedge\partial_j^F|_F^2 K(\partial_i^F,\partial_j^F)\\
&=\sum_{i,j=1}^{n-p+2} |\partial_i^F\wedge\partial_j^F|_F^2 {1\over\delta^2}.
\end{split}
\end{equation}
Every term in Equation \ref{torpedo ISC cross interval vertex}, and at least one term must be strictly positive since $p\leq n-2$ and so $P^\perp$ cannot lie entirely in the $t$-direction. Therefore $s_{p,n+2}(P)>0$, so that $g_{\tor}^{n+2}(\delta)_\lambda+dt^2$ has positive $(p,n+2)$-intermediate scalar curvature for $p\leq n-2$.
\end{proof}

\subsection{The Toe Metric}

In this section we discuss the construction of the toe metric $g_{\toe}^{n+2}(\delta)_{\lambda_1,\lambda_2}$. In Lemma 2.1 of \cite{walsh2} it is shown that the metric smoothing, necessary in constructing $g_{\toe}^{n+2}(\delta)_{\lambda_1, \lambda_2}$, can be done so as to preserve positive scalar curvature. In the following Lemma, we show that the construction described in \cite{walsh2} actually preserves positive $(p,n+2)$-curvature for $p\leq n-2$.

\begin{lemma}\label{toecurvature}
Suppose $n\geq 2$ and $0\leq p\leq n-2$. For any $\delta>0, \lambda_1$ and  $\lambda_2\geq 0$, the metric $g_{\toe}^{n+2}(\delta)_{\lambda_1, \lambda_2}$ has positive $(p,n+2)$-curvature. Moreover, this curvature can be bounded below with an arbitrarily large positive constant by choosing $\delta$ sufficiently small.
\end{lemma}
\begin{proof}
The strategy of proof involves describing the metric as in Lemma 2.1 of \cite{walsh2} and then computing its $(p,n+2)$-curvature. Thus, we regard $g_{\toe}^{n+2}(\delta)_{\lambda_1, \lambda_2}$ as obtained by, firstly, tracing out a cylinder of torpedo metrics, $g_{\tor}^{n+1}(\delta)$, in one direction before, secondly, bending by an angle of $\frac{\pi}{2}$ to finish with another cylinder in an orthogonal direction. It is then easy to extend the ``rectangular" part of this metric which takes the form $dr^{2}+dt^{2}+\delta^{2}ds_{n}^{2}$ to obtain any desired pair of neck-lengths, $\lambda_1, \lambda_2>0$. 

We will not worry about the choices of $\lambda_1$ and $\lambda_2$ for now. Along the region where the bending has taken place, this metric takes the form
$$g_{\toe}^{n+2}(\delta)_{\lambda_1, \lambda_2}=dr^{2}+\omega(r,t)^{2}dt^{2}+\beta(r)^{2}ds_{n}^{2},$$
for certain smooth warping functions $\omega:[0,b]\times[-2,\frac{\pi}{2}+2]\to[1,\infty)$ and $\beta:[0, b]\rightarrow [0,\infty).$ 
The function $\beta$ is a torpedo function and so we can assume, for all $r\in[0,b]$, $\beta(r)=\eta_{\delta,\lambda}(r)$ as in section 5.2, for some appropriate neck-length $\lambda$. Note that, there is a corresponding real-valued function, $\alpha$, on $[0,b]$ which satisfies the condition (\ref{alpha0}) above. The other warping function, $\omega$, is now constructed to satisfy the following conditions:
\begingroup\singlespacing
$$\omega(r,t)=\begin{cases}1&t\in[-2,-1]\cr\mu(-t)+(1-\mu(-t))\alpha(r)&t\in[-1,0]
\cr\alpha(r)&t\in[0,\frac{\pi}{2}]\cr\mu(t-{\pi\over 2})+(1-\mu(t-{\pi\over 2}))\alpha(r)&t\in[{\pi\over 2},{\pi\over 2}+1]
\cr1&t\in[\frac{\pi}{2}+1,\frac{\pi}{2}+2]\end{cases}$$
\endgroup
where $\mu:[0,1]\rightarrow [0,1]$ is a cut-off function satisfying $\mu(t)=0$ when $t$ is near $0$, $\mu(t)=1$ when $t$ is near $1$ and $\mu'(t)\geq 0$ for all $t\in[0,1]$. 

Given the form of the metric, recall from Proposition \ref{pcurvatureofwarpedproduct}, the $(p, n+2)$-curvatures, $s_{p, n+2}(P)$, of this metric for a $p$-plane $P$ in the tangent space have the form.
\begin{equation*}
s_{p, n+2}(P)=(n-p)(n-p-1)\frac{1-\beta_r^2}{\beta^2}+A^2(1-\beta_r)^2-B^2\omega\omega_{rr}-C^2\beta\beta_{rr}-D^2\beta\beta_r\omega\omega_r
\end{equation*}
for some numbers $A$, $B$, $C$, and $D$ depending on the plane $P$.

Since $p\leq n-2$, then the coefficient on the first term is non-zero. Positivity of this term follows from the definition of $\beta$ as a $\delta$-torpedo function, applying L'H\^opital's rule at $r=0$. Indeed, as in Proposition \ref{positivity of torpedo curvature}, this term can be made arbitrarily large by choosing $\delta>0$ sufficiently small. Similarly, the second term is nonnegative. The remaining terms involve either $-\omega_r\beta_r$, $-\omega_{rr}$ or $-\beta_{rr}$. By construction, $\omega_{rr}$ and $\beta_{rr}$ are both non-positive. Hence, the corresponding terms above are also  non-negative. Finally, $\omega_r$ is a non-negative constant multiple of $\alpha_{r}$. Moreover, when $\alpha_r$ and $\beta_r$ are non-zero, they have opposite signs, so the remaining terms above are also non-negative. This completes the proof.
\end{proof}

\subsection{The Boot Metric}

We now consider, $g_{\boot}^{n+2}(\delta)_{\Lambda, \bar{l}}$, the boot metric introduced above. As our only concern here is establishing conditions for positivity of the $(p, n+2)$-curvature of such a metric, the values of the particular components of the vector $\bar{l}$ are unimportant and so we will suppress them from the notation, writing the metric simply as $g_{\boot}^{n+2}(\delta)_{\Lambda}$. This metric can be regarded as consisting of four pieces:

\begingroup\singlespacing\begin{equation}\label{bootcomponents}
g^{n+2}_{\boot}(\delta)_\Lambda :=
  \begin{cases}	g_{\toe}^{n+2}(\delta) & \text{on $R_1=D_{\stret}^{n+2}$,} \\
                               g_{\bend}^{n+2}(\delta)_\Lambda  & \text{on $R_2\cong D^{n+1}\times[0,1]$,} \\
                               g_{\tor}^{n+1}(\delta)+dt^2  & \text{on $R_3\cong D^{n+1}\times[0,1]$, and} \\
                               dr^2+dt^2+\delta^2ds_n^2  & \text{on $R_4\cong D^2\times S^n$.} \\
  \end{cases}
\end{equation}\endgroup
as depicted in Fig. \ref{BootDetail} below.

\begin{figure}[!htbp]
\vspace{4cm}
\hspace{-1cm}
\begin{picture}(0,0)
\includegraphics{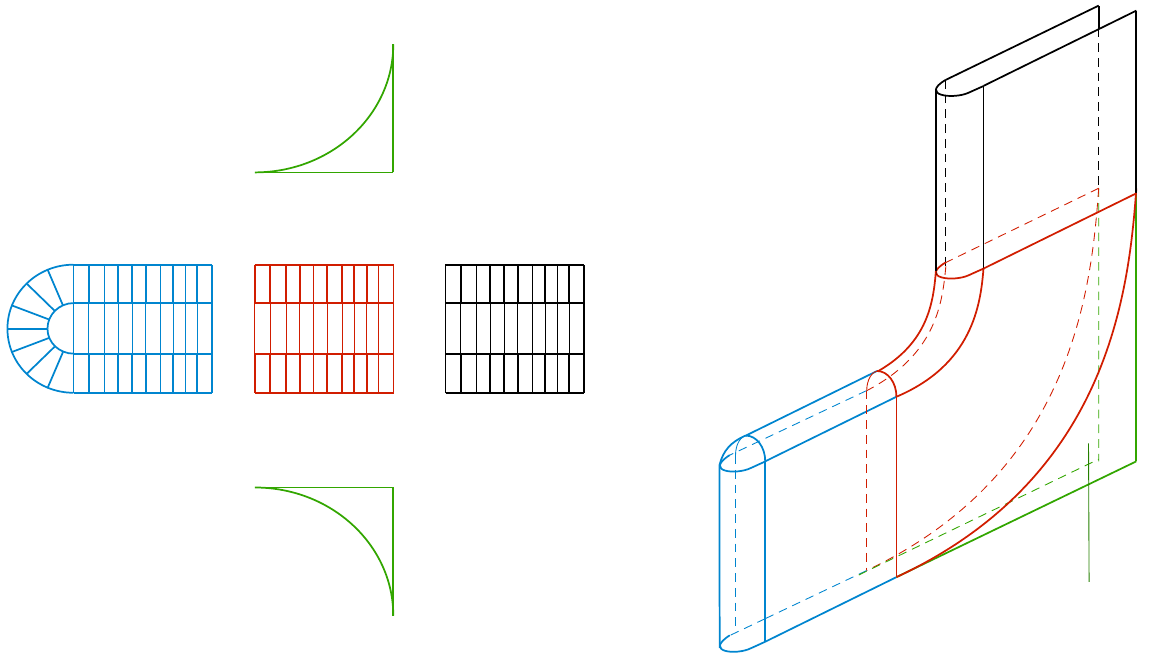}%
\end{picture}
\setlength{\unitlength}{3947sp}%
\begin{picture}(5079,1559)(1902,-7227)
\put(1900,-6150){\makebox(0,0)[lb]{\smash{{\SetFigFont{10}{8}{\rmdefault}{\mddefault}{\updefault}{\color[rgb]{0,0.6,0.8}$R_1=D_{\stret}^{n+2}$}%
}}}}
\put(3000,-5200){\makebox(0,0)[lb]{\smash{{\SetFigFont{10}{8}{\rmdefault}{\mddefault}{\updefault}{\color[rgb]{0.8,0,0}$R_2\cong D^{n+1}\times I$}%
}}}}
\put(4000,-6150){\makebox(0,0)[lb]{\smash{{\SetFigFont{10}{8}{\rmdefault}{\mddefault}{\updefault}{\color[rgb]{0,0,0}$R_3\cong D^{n+1}\times I$}%
}}}}
\put(2400,-6850){\makebox(0,0)[lb]{\smash{{\SetFigFont{10}{8}{\rmdefault}{\mddefault}{\updefault}{\color[rgb]{0,0.6,0}$R_4\cong D^{2}\times S^{n}$}
}}}}
\put(4650,-6850){\makebox(0,0)[lb]{\smash{{\SetFigFont{10}{8}{\rmdefault}{\mddefault}{\updefault}{\color[rgb]{0,0.6,0.8}$\hat{g}_{\tor}^{n+2}(\delta)$}
}}}}
\put(5350,-5500){\makebox(0,0)[lb]{\smash{{\SetFigFont{10}{8}{\rmdefault}{\mddefault}{\updefault}{\color[rgb]{0.8,0,0}$g_{\bend}^{n+2}(\delta)_{\Lambda}$}%
}}}}
\put(5500,-4300){\makebox(0,0)[lb]{\smash{{\SetFigFont{10}{8}{\rmdefault}{\mddefault}{\updefault}{\color[rgb]{0,0,0}$dt^{2}+g_{\tor}^{n+1}(\delta)$}%
}}}}
\put(6000,-7050){\makebox(0,0)[lb]{\smash{{\SetFigFont{10}{8}{\rmdefault}{\mddefault}{\updefault}{\color[rgb]{0,0.6,0}$dr^{2}+dt^{2}+\delta^{2}ds_{n}^{2}$}
}}}}
\end{picture}%
\caption{The various components of the boot metric, to the left are the ``assembly instructions" for the boot and to the right is the assembled boot.}
\label{BootDetail}
\end{figure}
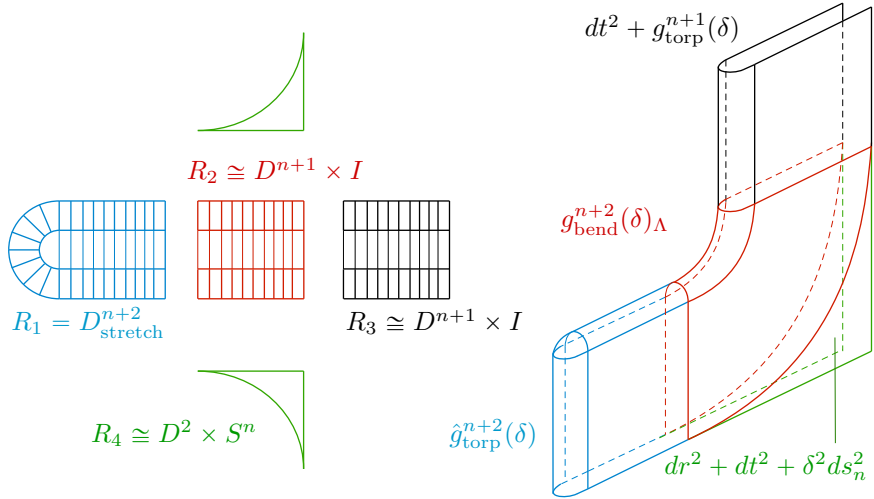

While the first, second and fourth components of this metric above are clearly defined, the third component $g_{\bend}^{n+2}(\delta)_{\Lambda}$ requires some description. A detailed account of this construction is given in section 5 of \cite{walshNY} and so we will be brief. As mentioned above, the metric component, $g_{\bend}^{n+2}(\delta)_{\Lambda}$, is obtained by bending a cylinder of torpedo metrics $g_{\tor}^{n+1}(\delta)+dt^{2}$ around a quarter circle. Importantly, the bend is in the opposite direction to that employed in Lemma \ref{toecurvature}, and, unlike in that case creates negative curvature. Provided we perform the bending slowly enough, that is provided the quarter circle has sufficiently large radius, we can minimize such negative curvature. The parameter $\Lambda$  is the radius of this quarter circle. In section 5, page 892, of \cite{walshNY}, the metric $g_{\bend}^{n+2}(\delta)_{\Lambda}$ is defined as follows:  
\begin{equation*}
{g_{\bend}^{n+2}}(\delta)_{\Lambda}:=dr^{2}+\omega_{\Lambda}(r,t)^{2}dt^{2}+\beta(r)^{2}ds_{n}^{2},
\end{equation*}
where $r\in [0,b]$, $t\in [-3, \frac{\pi}{2}+3]$. Here  $\beta:[0,b]\rightarrow[0,\infty)$ is defined as:
$$ \beta(r)=\eta_{\delta}(b-r),$$
with respect to a torpedo function $\eta_{\delta}$ as defined in Equation \ref{torpcond}, where we suppress the neck length $\lambda$ as it is unimportant. The corresponding smooth function $\alpha$, is constructed from $\beta$ to satisfy $\alpha_r^{2}+\beta_r^{2}=1$, as defined in Equation \ref{alpha0}. The function $$\omega_{\Lambda}:[0,b]\times\left[-3, \frac{\pi}{2}+3\right]\longrightarrow [1, \infty)$$ is now defined to satisfy the following properties:
\begin{enumerate}
\item[(i)] \begingroup\singlespacing$$
\omega_\Lambda(r,t) =
  \begin{cases}		1 & \text{if $-3\leq t\leq -2$,} \\
                                \Lambda  & \text{if $-\frac{3}{2}\leq t\leq -\frac{1}{2}$,} \\
                                 \Lambda+\alpha(r) & \text{if $t\in[0,\frac{\pi}{2}]$,} \\
   \Lambda & \text{if $\frac{\pi}{2}+\frac{1}{2} \leq t\leq \frac{\pi}{2}+\frac{3}{2}$,}\\
   1 & \text{if $\frac{\pi}{2}+2\leq t\leq \frac{\pi}{2}+3$,} 
  \end{cases}
$$\endgroup
\item[(ii)] $\omega_{\Lambda}(r,t) \geq \Lambda-\max\{|\alpha(r)|:r\in[0,b]\}$ when $t\in[-\frac{3}{2}, \frac{\pi}{2}+\frac{3}{2}]$,
\item[(iii)] $\displaystyle\frac{\partial{w_{\Lambda}}}{\partial r}(r,t)=0$  when $t\in[-3,-\frac{1}{2}]\cup [\frac{\pi}{2}+\frac{1}{2}, \frac{\pi}{2}+3]$,

\item[(iv)] $\displaystyle\left|\frac{\partial^{(k)}{w_{\Lambda}}}{\partial r^{(k)}}(r,t)\right|\leq \left| \frac{\partial^{(k)}{\alpha}}{\partial r^{(k)}}(r)\right|$ for all $k\in\{0,1,2,\cdots\}$, and

\item[(v)] $\displaystyle\frac{\partial^{2}{w_{\Lambda}}}{\partial r^{2}}\leq 0$.
\end{enumerate}

\noindent To aid the reader, we provide in Figure \ref{hardbendadjust} a schematic description of this function on its rectangular domain $[0,b]\times[-3, \frac{\pi}{2}+3]$. The white regions in this picture indicate the smooth transition by way of a cut-off function in exactly the spirit of the cut-off function $\mu$ used when defining the warping function $\omega$ in the proof of Lemma \ref{toecurvature}.
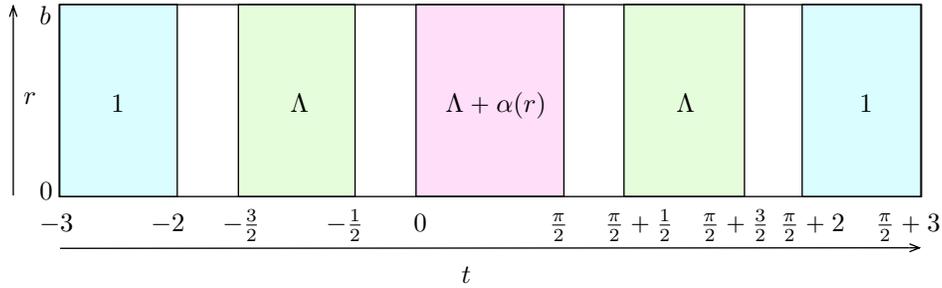
\begin{figure}[!htbp]
\vspace{-1.5cm}
\hspace{-8cm}
\begin{picture}(0,0)%
\includegraphics{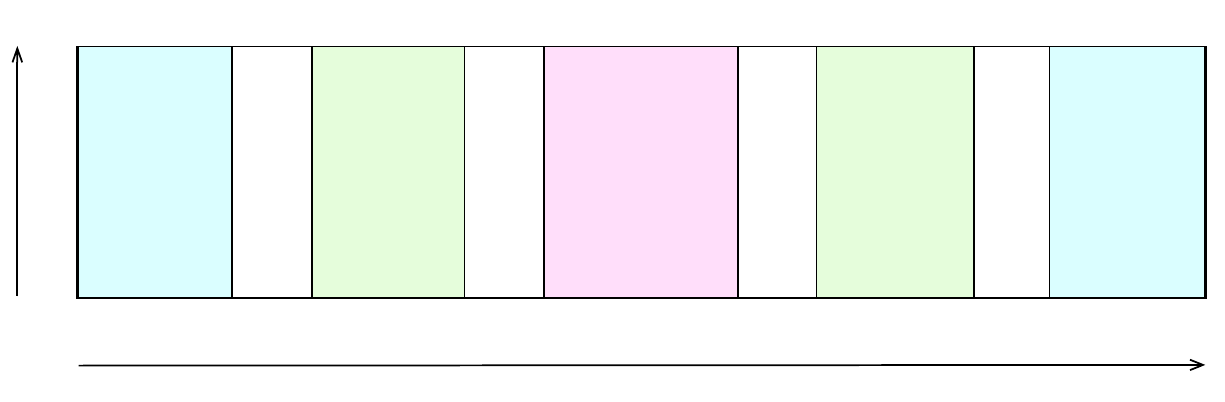}%
\end{picture}%
\setlength{\unitlength}{3947sp}%
\begin{picture}(2079,2559)(1902,-5227)
\put(2600,-4200){\makebox(0,0)[lb]{\smash{{\SetFigFont{10}{8}{\rmdefault}{\mddefault}{\updefault}{\color[rgb]{0,0,0}$1$}%
}}}}

\put(3725,-4200){\makebox(0,0)[lb]{\smash{{\SetFigFont{10}{8}{\rmdefault}{\mddefault}{\updefault}{\color[rgb]{0,0,0}$\Lambda$}%
}}}}
\put(4700,-4200){\makebox(0,0)[lb]{\smash{{\SetFigFont{10}{8}{\rmdefault}{\mddefault}{\updefault}{\color[rgb]{0,0,0}$\Lambda+\alpha(r)$}%
}}}}
\put(6150,-4200){\makebox(0,0)[lb]{\smash{{\SetFigFont{10}{8}{\rmdefault}{\mddefault}{\updefault}{\color[rgb]{0,0,0}$\Lambda$}%
}}}}
\put(7300,-4200){\makebox(0,0)[lb]{\smash{{\SetFigFont{10}{8}{\rmdefault}{\mddefault}{\updefault}{\color[rgb]{0,0,0}$1$}%
}}}}
\put(2150,-4950){\makebox(0,0)[lb]{\smash{{\SetFigFont{10}{8}{\rmdefault}{\mddefault}{\updefault}{\color[rgb]{0,0,0}$-3$}%
}}}}
\put(2850,-4950){\makebox(0,0)[lb]{\smash{{\SetFigFont{10}{8}{\rmdefault}{\mddefault}{\updefault}{\color[rgb]{0,0,0}$-2$}%
}}}}
\put(3300,-4950){\makebox(0,0)[lb]{\smash{{\SetFigFont{10}{8}{\rmdefault}{\mddefault}{\updefault}{\color[rgb]{0,0,0}$-\frac{3}{2}$}%
}}}}
\put(3950,-4950){\makebox(0,0)[lb]{\smash{{\SetFigFont{10}{8}{\rmdefault}{\mddefault}{\updefault}{\color[rgb]{0,0,0}$-\frac{1}{2}$}%
}}}}
\put(4500,-4950){\makebox(0,0)[lb]{\smash{{\SetFigFont{10}{8}{\rmdefault}{\mddefault}{\updefault}{\color[rgb]{0,0,0}$0$}%
}}}}
\put(5350,-4950){\makebox(0,0)[lb]{\smash{{\SetFigFont{10}{8}{\rmdefault}{\mddefault}{\updefault}{\color[rgb]{0,0,0}$\frac{\pi}{2}$}%
}}}}
\put(5700,-4950){\makebox(0,0)[lb]{\smash{{\SetFigFont{10}{8}{\rmdefault}{\mddefault}{\updefault}{\color[rgb]{0,0,0}$\frac{\pi}{2}+\frac{1}{2}$}%
}}}}
\put(6300,-4950){\makebox(0,0)[lb]{\smash{{\SetFigFont{10}{8}{\rmdefault}{\mddefault}{\updefault}{\color[rgb]{0,0,0}$\frac{\pi}{2}+\frac{3}{2}$}%
}}}}
\put(6800,-4950){\makebox(0,0)[lb]{\smash{{\SetFigFont{10}{8}{\rmdefault}{\mddefault}{\updefault}{\color[rgb]{0,0,0}$\frac{\pi}{2}+2$}%
}}}}
\put(7400,-4950){\makebox(0,0)[lb]{\smash{{\SetFigFont{10}{8}{\rmdefault}{\mddefault}{\updefault}{\color[rgb]{0,0,0}$\frac{\pi}{2}+3$}%
}}}}
\put(4800,-5270){\makebox(0,0)[lb]{\smash{{\SetFigFont{10}{8}{\rmdefault}{\mddefault}{\updefault}{\color[rgb]{0,0,0}$t$}%
}}}}
\put(2150,-4750){\makebox(0,0)[lb]{\smash{{\SetFigFont{10}{8}{\rmdefault}{\mddefault}{\updefault}{\color[rgb]{0,0,0}$0$}%
}}}}
\put(2150,-3650){\makebox(0,0)[lb]{\smash{{\SetFigFont{10}{8}{\rmdefault}{\mddefault}{\updefault}{\color[rgb]{0,0,0}$b$}%
}}}}
\put(2050,-4150){\makebox(0,0)[lb]{\smash{{\SetFigFont{10}{8}{\rmdefault}{\mddefault}{\updefault}{\color[rgb]{0,0,0}$r$}%
}}}}
\end{picture}%
\caption{The function $\omega_{\Lambda}$} 
\label{hardbendadjust}
\end{figure} 

\begin{lemma}\label{bendy} For $n\geq 2$ and any $\delta>0$ there is a positive constant $\Lambda$ for which the metric ${g_{\bend}^{n+2}}(\delta)_{\Lambda}$ has positive $(p, n+2)$-intermediate scalar curvature for all $p\in\{0, \cdots, n-2\}$.
\end{lemma}
\begin{proof}
We make use of Proposition \ref{pcurvatureofwarpedproduct} for the intermediate scalar curvatures $s_{p, n+2}(P)$ for some $p$-dimensional subspace of the tangent space $P$. Recall that 
\begin{equation*}
s_{p, n+2}(P)=(n-p)(n-p-1)\frac{1-\beta_r^2}{\beta^2}+A^2(1-\beta_r)^2-B^2\omega\omega_{rr}-C^2\beta\beta_{rr}-D^2\beta\beta_r\omega\omega_r
\end{equation*}
for some numbers $A$, $B$, $C$, and $D$ dependent on the plane $P$.

As with the proof of Lemma \ref{toecurvature}, the condition on $p$ and the fact that $\beta$ is a torpedo function, albeit pointing in the opposite direction, means that the first term is positive and the second term is nonnegative. As $\beta_{rr}\leq 0$ and $\omega_{rr}\leq 0$, the third and fourth terms are also nonnegative.

Unlike the case of Lemma \ref{toecurvature} however, we do not get such a nice relationship between the signs of the first derivatives of $\beta$ and $\alpha$. This is because $\beta_{r}\leq 0$ in this case. Thus, there may be some negativity arising in the fifth term. This negativity arises only in the region where $t\in[-\frac{1}{2},\frac{\pi}{2}+\frac{1}{2}]$, since $\frac{\p \omega_{\Lambda}}{\p r}=0$ off this region. Moreover, the denominator of each of these problem terms is bounded below by $$\Lambda-\max\{|\alpha(r)|:r\in[0,b]\}$$ on this region, while the absolute value of the numerator is bounded by the maximum of $|\alpha_{r}\beta_{r}|$. We emphasize that the functions $\alpha$,  $\beta$, and their derivatives, all of which are bounded, are fixed and independent of $\Lambda$. Since the factor of $D$ in the fifth term depends continuously on the choice of plane $P$ which varies over the compact Grassmannian, $\Grass_p(TR_2)$, there is some choice of $\Lambda$ sufficiently large to minimize the negative impact of these terms and ensure overall positivity of $s_{p,n+2}(P)$.
\end{proof}

\begin{corollary}\label{bootlambda}
For $n\geq 2$ and any $\delta>0$ there is a positive constant, $\Lambda$, for which the boot metric ${g_{\boot}^{n+2}}(\delta)_{\Lambda}$ has positive $(p, n+2)$-intermediate scalar curvature for all $p\in\{0, \cdots, n-2\}$.
\end{corollary}
\begin{proof}
Here we simply utilize the the description in Equation \ref{bootcomponents}, of ${g_{\boot}^{n+2}}(\delta)_{\Lambda}$ and its four component metrics. On $R_1$, the metric $g_{\toe}^{n+2}(\delta)$ has positive $(p, n+2)$-intermediate scalar curvature by Lemma \ref{toecurvature}. On $R_2$, the metric $g_{\bend}^{n+2}(\delta)_\Lambda$ has positive $(p,n+2)$-intermediate scalar curvature by Lemma \ref{bendy}. On $R_3$, the metric $g_{\tor}^{n+1}(\delta)+dt^2$ has positive $(p,n+2)$-intermediate scalar curvature by Proposition \ref{positivity of torpedo curvature crossed with an interval}.

Finally, on $R_4$, we have the metric $dr^2+dt^2+\delta^2ds_n^2$. By  Proposition \ref{pcurvatureofwarpedproduct}, using $\omega=1$ and $\beta=\delta$, we have for any $p$-plane,
$$s_{p,n+2}(P)=(n-p)(n-p-1)+A^2$$
for some function $A$ of the plane $P$. Since $n-p\geq 2$, this is positive for all planes $P$ so we have positive $(p,n+2)$-intermediate scalar curvature on $R_3$.
\end{proof}

\subsection{The Product of the Boot with a Sphere}

The following proposition generalizes the result of Corollary \ref{bootlambda} to the product of a boot metric and a round sphere:

\begin{proposition}
For $n\geq 2$, $m\geq 0$ and any $\delta>0$, there is a positive constant $\Lambda$ for which the product metric $g_{\boot}^{n+2}(\delta)_{\Lambda}+ds_m^2$ has positive $(p,n+m+2)$-intermediate scalar curvature for all $p\in\{0,\ldots,n-2\}$.
\end{proposition}
\begin{proof}
Let $M$ be the product $B\times S^m$ where $B$ is a disk with a boot metric $g_{\boot}^{n+2}(\delta)_{\Lambda}$ and $S^k$ is the round sphere for some $k\geq 0$ with the standard product metric
$$g_{\boot}^{n+2}(\delta)_{\Lambda}+ds_m^2$$

Let $p\in\{0,\ldots,n-2\}$ so that the boot $(B,g_{\boot}^{n+2}(\delta)_{\Lambda})$ can have positive $(p,n+2)$-intermediate scalar curvature for large enough $\Lambda$. If $x\in M$ and $P$ is a $p$-plane in $T_xM$, then the orthogonal complement $P^\perp$ has dimension $n+m+2-p$. Since the direct sum $P^\perp+T_xB$ is at most the entire tangent space $T_xM$, we have
\begin{equation*}
\begin{split}
\dim(P^\perp\cap T_xB)&=\dim(P^\perp)+\dim(T_xB)-\dim(P^\perp+T_xB)\\
&\geq (n+m+2-p)+(n+2)-(n+m+2)\\
&= n+2-p.
\end{split}
\end{equation*}

Therefore, we can take an orthonormal basis for $P^\perp$ consisting of vectors $e_1,\ldots,e_{n+2-p}$ from $T_xB$ and $v_1,\ldots,v_m$ other vectors. Therefore the $(p,n+m+2)$-intermediate scalar curvature of $P$ is
\begin{equation}\label{pcurvatureofbootsphereproduct}
\begin{split}
s_{p,n+m+2}(P) &= \sum_{i,j} K(e_i,e_j) + 2\sum_{i,j} K(e_i,v_j) + \sum_{i,j} K(v_i,v_j).
\end{split}
\end{equation}

If we denote the sectional curvatures of the boot and sphere by $K_B$ and $K_S$ respectively, and we have orthonormal vectors $v=v^B+v^S$ and $w=w^B+w^S$ in $T_xM$ where $v^B,w^B$ are tangent to the boot and $v^S,w^S$ are tangent to the sphere, then recall the formula,
$$K(v,w)=\left|v^B\wedge w^B\right|^2K_B(v^B,w^B)+\left|v^S\wedge w^S\right|^2K_S(v^S,w^S).$$

For the first summand of Equation \ref{pcurvatureofbootsphereproduct}, since all of the $\p_i$ are entirely in $T_xB$ and are still orthonormal there,
$$\sum_{i,j} K(e_i,e_j)=\sum_{i,j} K_B(e_i,e_j) = s_{n+2,p}(Q),$$
where $Q$ is the $p$-plane in $T_xB$ that is the orthogonal complement of the span of the $n+2-p$ vectors $e_i$. As a $(p,n+2)$-curvature of the boot, so long as we choose a large enough $\Lambda$, this summand is strictly positive.

Moving on to the second summand of Equation \ref{pcurvatureofbootsphereproduct}, since $v_j$ is orthogonal to $e_i$, which is already contained in $T_xB$, the projection $v_j^B$ remains orthogonal to $e_i$. Therefore,
$$K(e_i,v_j)=\left|e_i\wedge v_j^B\right|^2K_B(e_i,v_j^B)=\|v_j^B\|^2K_B(e_i,v_j^B).$$

For the third summand, the only immediate simplification we can make is to use the fact that the sphere has constant curvature of $1$,
\begin{equation*}
\begin{split}
K(v_i,v_j) &= \left|v_i^B\wedge v_j^B\right|^2K_B(v_i^B,v_j^B)+\left|v_i^S\wedge v_j^S\right|^2K_S(v_i^S,v_j^S)\\
&= \left|v_i^B\wedge v_j^B\right|^2K_B(v_i^B,v_j^B)+\left|v_i^S\wedge v_j^S\right|^2.
\end{split}
\end{equation*}
The second term of this only adds positivity, so we focus on the first. That is, we need to deal with sectional curvatures of the boot $K_B(v,w)$. There are four cases depending on where the projection of the point $x$ is in Figure \ref{BootDetail}. However by construction, $R_1$, $R_3$, and $R_4$ have nonnegative sectional curvature, so when the projection of $x$ is in any of these regions, we have that the first summand of Equation \ref{pcurvatureofbootsphereproduct} is strictly positive, while the other two are nonnegative. Therefore $s_{p,n+m+2}(P)>0$ in these cases. All that remains is to consider the case when $x$ projects into the $R_2$ region.

Since the metric on $R_2$ is given by
$${g_{\bend}^{n+2}}(\delta)_{\Lambda}=dr^{2}+\omega_{\Lambda}(r,t)^{2}dt^{2}+\beta(r)^{2}ds_{n}^{2},$$
then we have from Equation \ref{fullriemformainmetric}, with $\omega=\omega_\Lambda$,
\begin{equation*}
\begin{split}
\sum_{i,j} K_B(e_i,v^B_j) &= \sum_{i,j} {1\over\|v_j^B\|^2}\biggl[-\omega\omega_{rr}(e_{it}v_{jr}-e_{ir}v_{jt})^2+\sum_{k<\ell}(e_{ik}v_{j\ell}-e_{i\ell}v_{jk})^2\beta^2(1-\beta_r^{2})\\
&-\sum_{k} (e_{ik}v_{jr}-e_{ir}v_{jk})^2\beta\beta_{rr}-\sum_{k}(e_{ik}v_{jt}-e_{it}v_{jk})^2\beta\beta_r\omega\omega_r\biggr],\\
\end{split}
\end{equation*}
and
\begin{equation*}
\begin{split}
\sum_{i,j} K_B(v^B_i,v^B_j) &= \sum_{i,j} {1\over|v_i^B\wedge v_j^B|^2}\biggl[-\omega\omega_{rr}(v_{it}v_{jr}-v_{ir}v_{jt})^2+\sum_{k<\ell}(v_{ik}v_{j\ell}-v_{i\ell}v_{jk})^2\beta^2(1-\beta_r^{2})\\
&-\sum_{k} (v_{ik}v_{jr}-v_{ir}v_{jk})^2\beta\beta_{rr}-\sum_{k}(v_{ik}v_{jt}-v_{it}v_{jk})^2\beta\beta_r\omega\omega_r\biggr].\\
\end{split}
\end{equation*}
Incorporating these two into Equation \ref{pcurvatureofbootsphereproduct}, we have
\begin{equation*}
\begin{split}
s_{p,n+m+2}(P) &= \sum_{i,j} K_B(e_i,e_j) + 2\sum_{i,j} \|v^B_i\|^2 K_B(e_i,v^B_j)\\
&\quad\null+ \sum_{i,j} |v^B_i\wedge v^B_j|^2 K_B(v^B_i,v^B_j)+\sum_{i,j} |v^S_i\wedge v^S_j|^2\\
&=s_{p,n+2}(Q) - A^2\omega\omega_{rr}+B^2(1-\beta_r^2)-C^2\beta\beta_{rr}-D^2\beta\beta_r\omega\omega_r+E^2,\\
\end{split}
\end{equation*}
for real-valued functions $A,B,C,D,E$ ranging over the choice of plane $P$. As in the proof of Lemma \ref{bendy}, the second, third, and fourth terms are all nonnegative and of course the sixth term is also nonnegative. The first term is the $(p,n+2)$-intermediate scalar curvature of the boot metric, and by Corollary \ref{bootlambda} this is positive provided we choose $\Lambda$ large enough.

However, the fifth term may have some negativity. This negativity arises only in the region $t\in[-\frac{1}{2},\frac{\pi}{2}+\frac{1}{2}]$, since $\frac{\p \omega_{\Lambda}}{\p r}=0$ off this region. Moreover, the denominator of each of these problem terms is bounded below by $$\Lambda-\max\{|\alpha(r)|:r\in[0,b]\}$$ on this region, while the absolute value of the numerator is bounded by the maximum of $|\alpha_{r}\beta_{r}|$. The functions $\alpha$,  $\beta$, and their derivatives, all of which are bounded, are fixed and independent of $\Lambda$. Since the factor of $D$ in the fifth term depends continuously on the choice of plane $P$ which varies over the compact Grassmannian, $\Grass_p(TR_2)$, there is a choice of $\Lambda$ sufficiently large to minimize the negative impact of these terms and ensure overall positivity of $s_{p,n+m+2}(P)$.
\end{proof}

\section{Proof of Theorem \ref{A}}\label{s6}
\setcounter{figure}{0}

\subsection{The Surgery Theorem of \cite{gromov-lawson} and \cite{schoen-yau} for Positive $(p, n)$-Intermediate Scalar Curvature} 

We begin with a smooth manifold $M$ of dimension $n$. Suppose $\phi:S^{k}\times D^{\ell+1}\to M$ is an embedding, where $n=k+\ell+1$. Recall that a {\em surgery} on $M$, with respect to the embedding $\phi$, is the construction of a manifold, $M_{\phi}$ obtained by removing the image of $\phi$ from $M$ and using the restricted map $\phi|_{S^{k}\times S^{\ell}}$ to attach $D^{k+1}\times S^{\ell}$ along the common boundary.

The {\em trace of the surgery on $\phi$} is the cobordism between $M$ and $M_\phi$, obtained by gluing the cylinder $M\times [0,1]$ to the disk product $D^{k+1}\times D^{\ell+1}$ via the embedding $\phi$. This is done by attaching $M\times\{1\}$ to the boundary component $S^{k}\times D^{\ell+1}$ through the composition $i\circ\phi:S^{k}\times D^{\ell+1}\to M\times\{1\}$ where $i:M\to M\times\{1\}$ is the inclusion $i(x)=(x,1)$. After appropriate smoothing, we obtain the elementary cobordism $\bar{M}_{\phi}$.

Returning to the embedding, $\phi$, we consider the following family of rescaling maps: 
\begin{equation*}
\begin{split}
\sigma_\rho:S^{k}\times D^{\ell+1}&\longrightarrow S^{k}\times D^{\ell+1}\\
(x,y)&\longmapsto (x, \rho y),
\end{split}
\end{equation*}
where $\rho\in(0,1]$.
We then set $\phi_{\rho}:=\phi\circ\sigma_\rho$ and $N_{\rho}:=\phi_{\rho}(S^{k}\times D^{\ell+1})$ and setting $N:=N_{1}$. Thus, for any metric $g$ on $M$ and any $\rho\in(0,1]$, $\phi_{\rho}^{*}g$ is the metric obtained by taking the restriction metric $g|_{N_{\rho}}$, pulling it back via $\phi$ to obtain the metric $\phi^{*}g|_{N_{\rho}}$ on $S^{k}\times D^{\ell+1}(\rho)$ and finally, via the obvious rescaling map $\sigma_\rho$, pulling it back to obtain a metric on $S^{k}\times D^{\ell+1}$. 

The positive scalar curvature Surgery Theorem of  \cite{gromov-lawson} and \cite{schoen-yau} was generalized for positive $(p,n)$-intermediate scalar curvature in \cite{labbi-surgery}. At its heart is the following theorem from \cite{labbi-surgery}.
\begin{theorem}\cite{labbi-surgery}\label{GLLthm}
Let $M$ be a smooth $n$-dimensional manifold and $\phi:S^{k}\times D^{\ell+1}\to M$ an embedding with $k+\ell+1=n$. We further assume that $\ell\geq 2$ and $p\in\{0,1,\cdots,\ell-2\}$. Then for any metric $g$ on $M$ with positive $(p,n)$-intermediate scalar curvature, there is a metric $g_{\std}$ with positive $(p,n)$-intermediate scalar curvature such that:
\begin{enumerate}
\item[(i.)] In the neighborhood $N_{1/2}=\phi_{1/2}(S^{k}\times D^{\ell+1})$,  $g_{\std}$ pulls back to the metric
$$\phi_{1/2}^{*}g_{\std}=ds_{k}^{2}+g_{\tor}^{\ell+1}, \text{ and}$$
\item[(ii.)] Outside $N=\phi(S^{k}\times D^{\ell+1})$, $g_{\std}=g$.
\end{enumerate}
\end{theorem}
\noindent The metric $g_{\std}$ is thus {\em prepared for surgery} or {\em standardized} on $N_{1/2}$. By removing part of the standard piece taking the form $$(S^{k}\times D^{\ell+1}, ds_{k}^{2}+g_{\tor}^{\ell+1}),$$ and replacing it with $$(D^{k+1}\times S^{\ell}, g_{\tor}^{k+1}+ds_{\ell}^{2}),$$ we obtain a metric $g'$ on $M'$ with positive $(p,n)$-intermediate scalar curvature.

In order to prove our main theorem, we require one more fact: that the metrics $g$ and $g_{\std}$ above are $(s_{p,n}>0)$-isotopic. Proofs of this fact for the case of positive scalar curvature, that is when $p=0$, can be found in Theorem 2.3 of \cite{walsh1} and in \cite{ebert-frenck}. More recently, Kordass \cite{kordass}, proved this fact for a variety of general curvature conditions including positive $(p,n)$-intermediate scalar curvature. In particular, Theorem 3.1 from \cite{kordass}, which we state below, is a special case of his results.
\begin{lemma}[\cite{kordass}]\label{isotopylemma}
Let $M$ be a smooth $n$-dimensional manifold and let $\phi:S^{k}\times D^{\ell+1}\to M$ be an embedding with $k+\ell+1=n$, $\ell\geq 2$ and $p\in\{0,1,\cdots,\ell-2\}$. Then for any metric $g$ on $M$ with positive $(p,n)$-intermediate scalar curvature, there is an isotopy through $(s_{p,n}>0)$-metrics, $g_{t}, t\in[0,1]$ which satisfies the following conditions:
\begin{enumerate}
\item[(i)] $g_t|_{M\setminus N}=g|_{M\setminus N}$ for all $t\in[0,1]$, and
\item[(ii)] $g_{0}=g$ and $g_{1}=g_{\std}$, where $g_{\std}$ is the metric obtained from $g$ by Theorem \ref{GLLthm} above and satisfies
$$ \phi_{1/2}^{*}g_{\std}=ds_{k}^{2}+g_{\tor}^{\ell+1}.$$
\end{enumerate}
\end{lemma}
\noindent We will now make use of this lemma to prove Theorem \ref{A}, which we restate for the sake of the reader.
\\

\begin{theorem} 
Let $M$ be a smooth $n$-dimensional manifold, $\phi:S^{k}\times D^{\ell+1}\to M$, a smooth embedding, and $\{\bar{M}_{\phi}; M, M_{\phi}\}$, the trace of the surgery on $\phi$. Suppose that $\ell\geq 2$ and $p\in\{0,1,\cdots \ell-2\}$. Then for any metric $g$ on $M$ with positive $(p,n)$-intermediate scalar curvature, there are metrics $g_{\phi}$ on $M_{\phi}$ and $\bar{g}_{\phi}$ on $\bar{M}_{\phi}$ satisfying:
\begin{enumerate}
\item
 The metrics $g_{\phi}$ and $\bar{g}_{\phi}$ have respectively positive $(p,n)$ and $(p,n+1)$-intermediate scalar curvature on $M_{\phi}$ and $\bar{M}_{\phi}$, and
\item Near the boundary components $M$ and $M_{\phi}$ of $\bar{M}_{\phi}$, $\bar{g}_{\phi}$ takes the form of respective products $\bar{g}_{\phi}=g+dt^{2}$ and $\bar{g}_{\phi}=g_{\phi}+dt^{2}$.
\end{enumerate}
\end{theorem}
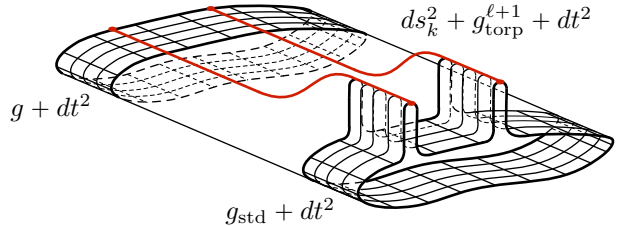
\begin{figure}[b]
\vspace{0cm}
\hspace{5cm}
\begin{picture}(0,0)
\includegraphics{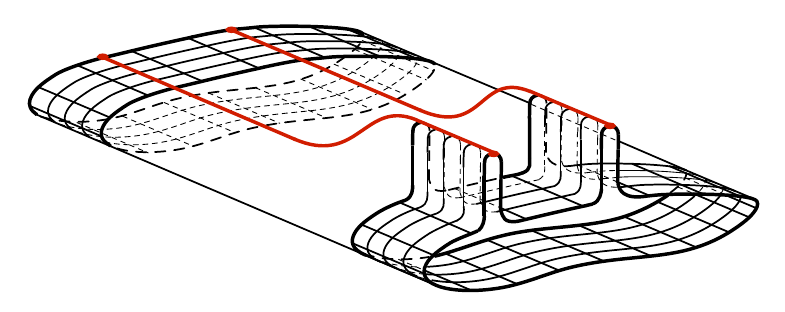}%
\end{picture}
\setlength{\unitlength}{3947sp}%
\begin{picture}(5000,1559)(1902,-7227)
\put(1700,-6550){\makebox(0,0)[lb]{\smash{{\SetFigFont{10}{8}{\rmdefault}{\mddefault}{\updefault}{\color[rgb]{0,0,0}$g+dt^{2}$}%
}}}}
\put(3050,-7200){\makebox(0,0)[lb]{\smash{{\SetFigFont{10}{8}{\rmdefault}{\mddefault}{\updefault}{\color[rgb]{0,0,0}$g_{\std}+dt^{2}$}
}}}}
\put(4150,-6000){\makebox(0,0)[lb]{\smash{{\SetFigFont{10}{8}{\rmdefault}{\mddefault}{\updefault}{\color[rgb]{0,0,0}$ds_{k}^{2}+g_{\tor}^{\ell+1}+dt^{2}$}%
}}}}
\end{picture}%
\caption{The concordance $\bar{g}$}
\label{cobootism5}
\end{figure}
\begin{proof}
We begin by employing Lemma \ref{isotopylemma} to obtain an isotopy, $g_{t}$ for $t\in[0,1]$, between $g_0=g$ and $g_1=g_{\std}$, as defined above. Corollary \ref{isotoconc} gives us a concordance, $\bar{g}$, on a cylinder 
$M\times[0,L+2]$ for some $L>0$ which satisfies the following conditions:
$$ \bar{g}|_{M\times [0,1]}=g+dt^{2}\quad and \quad \bar{g}|_{M\times [L+1,L+2]}=g_{\std}+dt^{2}.$$
This concordance is schematically depicted in Fig. \ref{cobootism5}. 

Consider a boot metric, $g_{\boot}^{\ell+2}(1)_{\Lambda, \bar{l}}$, as depicted in Figure \ref{torpcorners3}, now with $\delta=1$. By Corollary \ref{bootlambda}, we may choose $\Lambda$ to ensure positive $(p,n+1)$-curvature for $p\in\{0, \cdots, \ell-2\}$. Recall here that $\bar{l}$ is a quadruple $(l_1, l_2, l_3, l_4)$ determining the lengths of various sides of the boot metric (see Figure \ref{torpcorners3}). We set $l_1=l_4=1$. Recall that $l_2$ and $l_3$ depend on $\Lambda$ and may be large, so without loss of generality, we may assume that both are greater than $1$. To simplify notation, we refer to this metric henceforth as $g_{\boot}^{\ell+2}$. 

We next extend the collar $M\times [L+1, L+2]$ to $M\times [L+1, L+1+l_3]$ and extend the metric $\bar{g}$ in the obvious way as $g_{\rm std}+dt^{2}$ on this larger cylinder. Consider the restriction of the metric, $\bar{g}$, to the region, $N_{1/2}\times [L+1, L+l_{3}]$. Here, $\bar{g}$ takes the form
$$\bar{g}|_{N_{1/2}\times [L+1, L+1+l_{3}]}=ds_{k}^{2}+g_{\tor}^{\ell+1}+dt^{2}.$$
We replace $\bar{g}|_{N_{1/2}\times [L+1, L+1+l_{3}]}$ on $N_{1/2}\times [L+1, L+1+l_3]$ with the metric
$$ ds_{k}^{2}+g_{\boot}^{\ell+2},$$
as depicted in Figure \ref{cobootism6}, to obtain the metric

\begingroup\singlespacing$$
\bar{g}_{\boot} :=
  \begin{cases}		\bar{g}|_{M\times[0,L+1]} & \text{on $M\times[0,L+1]$,} \\
                                \bar{g}|_{(M\setminus N_{1/2})\times[L+1,L+1+l_3]}  & \text{on $(M\setminus N_{1/2})\times[L+1,L+1+l_3]$, and} \\
                                ds_k^2+g^{\ell+2}_{\boot} & \text{on $N_{1/2}\times[L+1,L+1+l_3]$.} \\
  \end{cases}
$$\endgroup

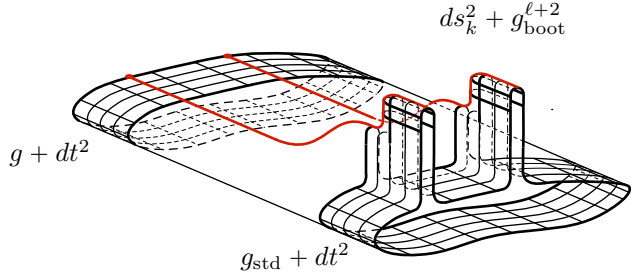
\begin{figure}[h]
\vspace{1cm}
\hspace{5cm}
\begin{picture}(0,0)
\includegraphics{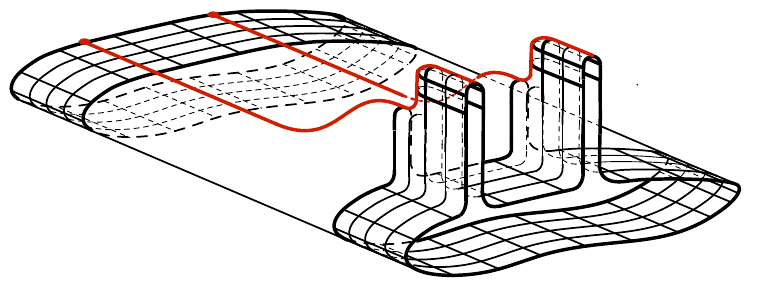}%
\end{picture}
\setlength{\unitlength}{3947sp}%
\begin{picture}(5079,1559)(1902,-7227)
\put(1500,-6550){\makebox(0,0)[lb]{\smash{{\SetFigFont{10}{8}{\rmdefault}{\mddefault}{\updefault}{\color[rgb]{0,0,0}$g+dt^{2}$}%
}}}}
\put(2950,-7200){\makebox(0,0)[lb]{\smash{{\SetFigFont{10}{8}{\rmdefault}{\mddefault}{\updefault}{\color[rgb]{0,0,0}$g_{\std}+dt^{2}$}
}}}}
\put(4200,-5700){\makebox(0,0)[lb]{\smash{{\SetFigFont{10}{8}{\rmdefault}{\mddefault}{\updefault}{\color[rgb]{0,0,0}$ds_{k}^{2}+g_{\boot}^{\ell+2}$}%
}}}}
\end{picture}%
\caption{The metric resulting from ``attaching boots", $\bar{g}_{\boot}$}
\label{cobootism6}
\end{figure}
\newpage

This metric has positive $(p, n+1)$-curvature for all $p\in\{0,1,\cdots \ell-2\}$. In particular, on $N_{1/2}\times [L+l_{3}, L+1+l_3]$, this metric takes the form of a product $g_{\std, l_2}+dt^{2}$, where 
$$g_{\std}=ds_{k}^{2}+g_{\tor}^{\ell+1}(1)_{l_{2}}.$$

In other words, this metric is a cylinder of standard metrics with torpedo necks stretched from length $1$ to length $l_2$. The metric $\bar{g}_{\boot}$ is readymade for surgery.

Returning to the region, ${N_{1/2}\times [L+1, L+1+l_{3}]}$, where the metric  $\bar{g}_{\boot}$ takes the form $ds_{k}^{2}+g_{\boot}^{\ell+1}$, we recall that the boot factor $g_{\boot}^{\ell+1}$ takes the form of a toe, $\hat{g}_{\toe}^{\ell+1}(1)$ (see Figure \ref{BootDetail}), on a subregion. As depicted in Figure \ref{cobootism7}, we replace $ds_{k}^{2}+\hat{g}_{\toe}^{\ell+2}(1)$ with $g_{\tor}^{k+1}+g_{\tor}^{\ell+1}$ by basically ``cutting off the toes'' and introducing a ``handle''.
\begin{figure}[h]
\vspace{1cm}
\hspace{5cm}
\begin{picture}(0,0)
\includegraphics{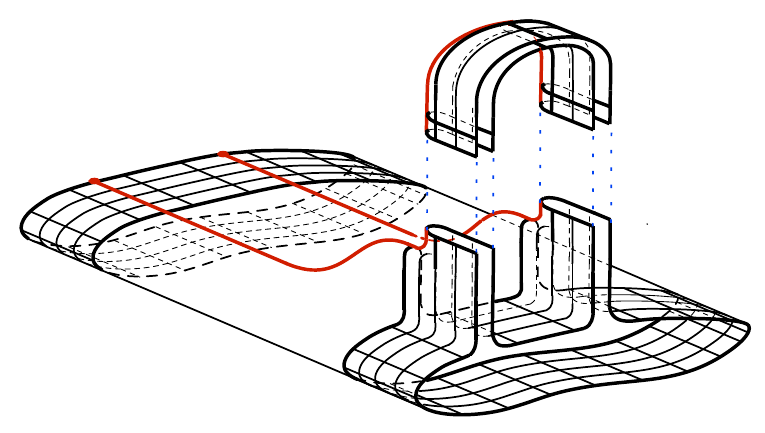}%
\end{picture}
\setlength{\unitlength}{3947sp}%
\begin{picture}(5079,1559)(1902,-7227)
\put(1500,-6550){\makebox(0,0)[lb]{\smash{{\SetFigFont{10}{8}{\rmdefault}{\mddefault}{\updefault}{\color[rgb]{0,0,0}$g+dt^{2}$}%
}}}}
\put(4900,-5500){\makebox(0,0)[lb]{\smash{{\SetFigFont{10}{8}{\rmdefault}{\mddefault}{\updefault}{\color[rgb]{0,0,0}$g_{\tor}^{k+1}+g_{\tor}^{\ell+1}$}%
}}}}
\end{picture}%
\caption{Replacing the toes with a handle.}
\label{cobootism7}
\end{figure}
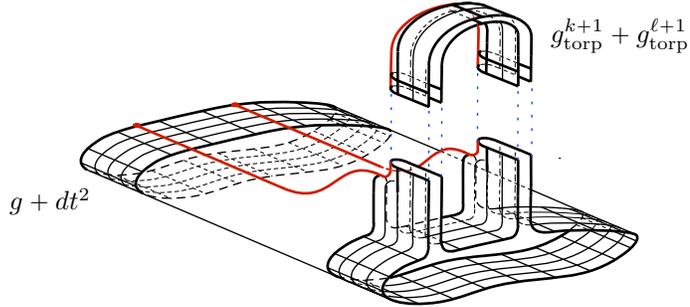

After gluing on this handle, we obtain the desired metric  $\bar{g}_{\phi}$ on the trace $\bar{M}_{\phi}$ of the surgery as shown in Figure \ref{cobootism8}. 

\begin{figure}[h]
\vspace{1cm}
\hspace{5cm}
\begin{picture}(0,0)
\includegraphics{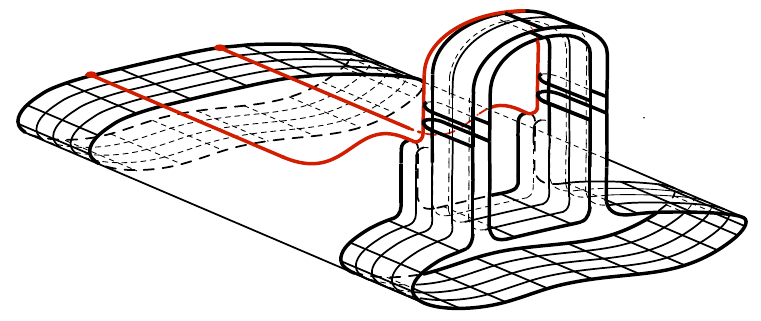}%
\end{picture}
\setlength{\unitlength}{3947sp}%
\begin{picture}(5079,1559)(1902,-7227)
\put(1500,-6550){\makebox(0,0)[lb]{\smash{{\SetFigFont{10}{8}{\rmdefault}{\mddefault}{\updefault}{\color[rgb]{0,0,0}$g+dt^{2}$}%
}}}}
\put(2950,-7200){\makebox(0,0)[lb]{\smash{{\SetFigFont{10}{8}{\rmdefault}{\mddefault}{\updefault}{\color[rgb]{0,0,0}$g_{\phi}+dt^{2}$}
}}}}
\end{picture}%
\caption{The space $(\bar{M}_\phi,\bar{g}_{\phi})$ obtained as the trace of the surgery.}
\label{cobootism8}
\end{figure}
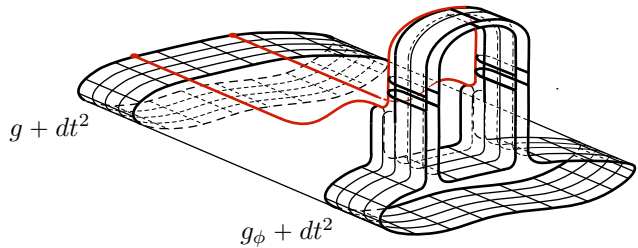

This completes the proof of Theorem \ref{A}.
\end{proof}

\section{The proof of Theorem \ref{B}}\label{s7}
The proof of Theorem \ref{B} follows closely that of Carr in \cite{carr} for the case of positive scalar curvature. There is also a very readable summary of the method of \cite{carr} in section 4.2 of \cite{TW}. We will therefore provide only the relevant information and refer the reader to these sources for further detail. As before, we restate the theorem to aid the reader.
\\

\begin{theorem}
 Let $M$ be a smooth, closed, spin manifold of dimension $4n-1, n\geq 2$, which admits an $s_{p,4n-1}>0$ curvature metric for some $p\in\{0,1,\cdots, 2n-3\}$. Then  $\Riem^{s_{p,4n-1}>0}(M)$ has infinitely many path components.
 \end{theorem}

\begin{proof}
 We begin by considering the case where $M$ is the sphere, $S^{4n-1}$ for some $n\geq 2$. In section 4 of \cite{carr}, the author constructs an infinite collection, $X_q$, $q\in\{1,2,\cdots\}$,  with the following properties:
 \begin{enumerate}
\item $X_q$ is a smooth $4n$-dimensional manifold with boundary $\p X_{q}$ diffeomorphic to $S^{4n-1}$, the standard sphere.
\item $X_q$ is homotopy equivalent to a finite wedge of $2n$-spheres. 
\item For $q_0\neq q_1$, the closed manifold $\bar{X}_{q_0, q_1}:=X_{q_0}\cup(S^{4n-1}\times I)\cup X_{q_1}$, where $I=[0,1]$ and obtained by gluing along common boundaries, has $\alpha(\bar{X}_{q_0, q_1})\neq 0$. Thus $\bar{X}_{q_0, q_1}$ admits no metrics of positive scalar curvature.
\end{enumerate}
These manifolds are constructed using the technique of {\em plumbing} disk bundles of the tangent bundle $TS^{4n}$ with respect to certain graphs, see, for example, section 4.2.1 of \cite{TW}. It is well known that performing this construction with respect to the graph which is the Dynkin diagram of the exceptional Lie group $E_8$ results in a smooth manifold, $X_{E_8}$, whose boundary is homeomorphic to the sphere $S^{4n-1}$; see, for example, Ch. VI, Sec. 12 in \cite{Kozinski}. In the case of $X_{q}$, the plumbing construction is in accordance with graphs based on $q\theta_{4n-1}$ copies of the Dynkin diagram for $E_8$, where $\theta_{4n-1}$ is the order of the boundary, $\p X_{E_8}$, in the group, $\Theta_{4n-1}$, of homeomorphic smooth $(4n-1)$-spheres. 

Let $W_q:=X_q\,\setminus \,D^{4n}$ denote the result of removing a disk, $D^{4n}$, from the interior of $X_q$. Thus, we obtain a cobordism, $\{W_q; \p D^{4n}, \p X_q\}$, where each boundary component is diffeomorphic to $S^{4n-1}$. Making use of an appropriate Morse function, we may decompose this into a finite union of elementary cobordisms:
$$ \{W_q; S_0^{4n-1}=\p D^{4n}, S_1^{4n-1}=\p X_q\}=\{Y_1;S_0^{4n-1}, S_1^{4n-1}\}\cup \{Y_2;S_1^{4n-1}, S_2^{4n-1}\}\cup\cdots \cup\{Y_m;S_{m-1}^{4n-1}, S_{m}^{4n-1}\},$$
where $S_{0}^{4n-1}$ is the boundary of the $4n$-dimensional disk we removed from the interior of $X_q$ and $S_m^{4n-1}=\p X_q$. Property $2$ above means that each elementary cobordism is the trace of a surgery on a $(2n-1)$-dimensional embedded sphere. Thus, each surgery takes place in codimension $(4n-1)-(2n-1)=2n$.

Equipping $S_0^{4n-1}$ with the 
round metric, $ds_{4n-1}^{2}$, and fixing collars near the boundaries of the individual elementary cobordisms, we repeatedly apply the geometric trace construction of Theorem \ref{A} to obtain a Riemannian metric, $\bar{g}_q$ on $W_q$. Given that all surgeries are in codimension $2n$, we can ensure that the resulting metric, $\bar{g}_q$, has $s_{p, 4n}>0$ for $p=2n-3$ and hence all $p\in\{0,1,\cdots 2n-3\}$. Moreover, $\bar{g}_q$ takes a product structure near the boundary and the restriction $g_q=\bar{g}|_{\p X_q}$ is a metric on $S^{4n-1}$, with $s_{p, 4n-1}>0$. 

Suppose now that we apply this procedure for a pair $q_0\neq q_1$. For each $i=0,1$, we can now form $s_{p, 4n}>0$ metrics, $\bar{h}_{q_i}$ on $X_{q_i}$, from $\bar{g_i}$ on $W_{q_i}$, by simply replacing the previously removed interior $4n$-disks and equipping the replacement disks with torpedo metrics $g_{\tor}^{4n}$. Thus, $$(X_{q_i}, \bar{h}_{q_i})=(W_{q_i}, \bar{g}_{q_i})\cup (D^{4n}, g_{\tor}^{4n}).$$ The product structure near the 
round boundary $(4n-1)$-spheres mean these attachments are smooth.

If $g_{q_0}$ and $g_{q_1}$ lie in the same path component of $\Riem^{s_{p,4n-1}>0}(S^{4n-1})$, there is an $(s_{p,4n-1}>0)$-isotopy and so, by Lemma \ref{isotopyconcordance}, an $(s_{p,4n-1}>0)$-concordance between them. We denote this concordance, $\bar{h}_{q_0, q_1}$. This allows us to construct the closed Riemannian manifold $$(\bar{X}_{q_0, q_1}, \bar{g}_{q_0, q_1}):=(X_{q_0}, \bar{h}_{q_0})\cup(S^{4n-1}\times I, \bar{h}_{q_0, q_1})\cup (X_{q_1}, \bar{h}_{q_1}),$$
depicted below in Fig. \ref{CarrClosedManif}.

\begin{figure}[h]
\vspace{2cm}
\hspace{-3cm}
\begin{picture}(0,0)
\includegraphics{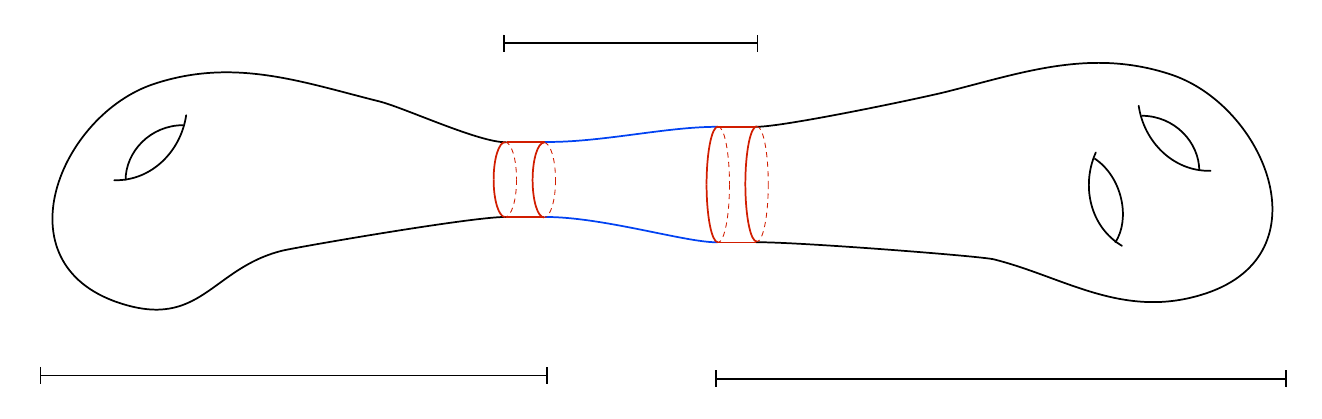}%
\end{picture}
\setlength{\unitlength}{3947sp}%
\begin{picture}(5079,1559)(1902,-7227)
\put(2950,-7200){\makebox(0,0)[lb]{\smash{{\SetFigFont{10}{8}{\rmdefault}{\mddefault}{\updefault}{\color[rgb]{0,0,0}$(X_{q_0}, \bar{h}_{q_0})$}
}}}}
\put(6350,-7200){\makebox(0,0)[lb]{\smash{{\SetFigFont{10}{8}{\rmdefault}{\mddefault}{\updefault}{\color[rgb]{0,0,0}$(X_{q_1}, \bar{h}_{q_1})$}
}}}}
\put(4250,-5200){\makebox(0,0)[lb]{\smash{{\SetFigFont{10}{8}{\rmdefault}{\mddefault}{\updefault}{\color[rgb]{0,0,0}$(S^{4n-1}\times I, \bar{h}_{q_0,q_1})$}
}}}}
\put(4050,-6500){\makebox(0,0)[lb]{\smash{{\SetFigFont{10}{8}{\rmdefault}{\mddefault}{\updefault}{\color[rgb]{1,0,0}$g_{q_0}+dt^{2}$}
}}}}
\put(5100,-6600){\makebox(0,0)[lb]{\smash{{\SetFigFont{10}{8}{\rmdefault}{\mddefault}{\updefault}{\color[rgb]{1,0,0}$g_{q_1}+dt^{2}$}
}}}}
\end{picture}%
\caption{The implied Riemannian manifold $\bar{X}_{q_0, g_1}, \bar{g}_{q_0, g_1}$}
\label{CarrClosedManif}
\end{figure}
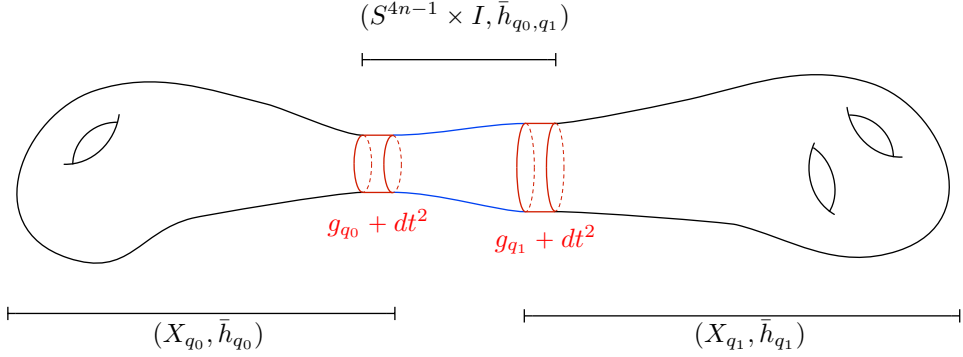

The metric $\bar{g}_{q_0, g_1}$ is a union of $s_{p,4n}>0$ metrics and so is itself an $s_{p,4n}>0$ metric for $p\in\{0,1,\cdots, 2n-3\}$. However this is impossible since, by Property $3$ above, $\bar{X}_{q_0, q_1}$ does not  admit positive scalar curvature metrics when $q_0\neq q_1$, and hence admits no $s_{p, n}>0$ metric for any $p\geq 0$. Thus, there is no isotopy between $g_{q_0}$ and $g_{q_1}$ and, by implication, the space $\Riem^{s_{p,4n-1}}(S^{4n-1})$ has infinitely many path components. This proves Theorem \ref{B} for the case $M=S^{4n-1}$.

We now consider the more general case of a closed smooth spin manifold, $M$, of dimension $4n-1$, which admits an $s_{p,4n-1}>0$ metric, $g_{M}$, for $p\in\{0,1,\cdots, 2n-3\}.$ For any $s_{p,4n-1}>0$ metric, $g$, on the sphere $S^{4n-1}$, the connected sum metric, $g_{M}\# g$, obtained by applying the surgery construction in Theorem \ref{A}, is an $s_{p,4n-1}>0$ metric also. We will show that for any pair $q_0\neq q_1$, the metrics $g_{M}\# g_{q_0}$ and $g_{M}\# g_{q_1}$, where $g_{q_0}$ and $g_{q_1}$ are the metrics constructed above, lie in different path components of $\Riem^{s_{p,4n-1}>0}(M)$. 

For each $i=0,1$, let $W_{q_i}$ denote the manifold obtained above by removing a $4n$-dimensional disk, $D^{4n}$, from the interior of $X_{q_i}$. Thus, $\p W_{q_i}=S_{i0}^{4n-1}\sqcup S_{i1}^{4n-1}$, where $S_{i0}^{4n-1}=\p D^{4n}$ and $S_{i1}^{4n-1}=\p X_{q_i}$. For $i,j=0,1$, we let the maps: 
$\tau_{ij}:S^{4n-1}\times [0,\epsilon)\longrightarrow W_{q_i},$
where $\tau_{ij}(S^{4n-1}\times\{0\})=S_{ij}^{4n-1}$, denote the disjoint collar neighborhood embeddings, employed in the metric construction of Theorem \ref{A}. Choose path embeddings, $\gamma_i:[0,1]\rightarrow W_{q_i}$, $i=0,1$, satisfying the following compatibility conditions:
\begin{itemize}
\item the endpoints of $\gamma_i$ satisfy $\gamma_{i}(0)\in S_{i0}^{4n-1}$ and $\gamma_{i}(1)\in S_{i1}^{4n-1}$,
\item when $t$ is near $0$, $\tau_{i0}^{-1}\circ\gamma_{i}(t)=t$ while, when $t$ is near $1$, $\tau_{i1}^{-1}\circ\gamma_{i}(t)=1-t$.
\end{itemize}
Finally, we specify a path $\gamma:I\rightarrow M\times I$, defined $\gamma(t)=(x_0,t)$ for some fixed point $x_0\in M$. By removing small tubular neighborhoods around $\gamma_i$ and $\gamma$, we perform a slicewise-connected sum along these embedded paths to obtain $Z_{q_i}:=W_{q_i}\# M\times I$, as depicted in Fig. \ref{SWconnsum}. We will assume that in $Z_{q_i}$, our slicewise connected sum construction associates $S_{ij}^{4n-1}$ with $M\times\{j\}$. Thus, $Z_{q_i}$ is a manifold whose boundary is a disjoint union of two diffeomorphic copies of $M$.

\begin{figure}[h]
\vspace{3cm}
\hspace{-3cm}
\begin{picture}(0,0)
\includegraphics{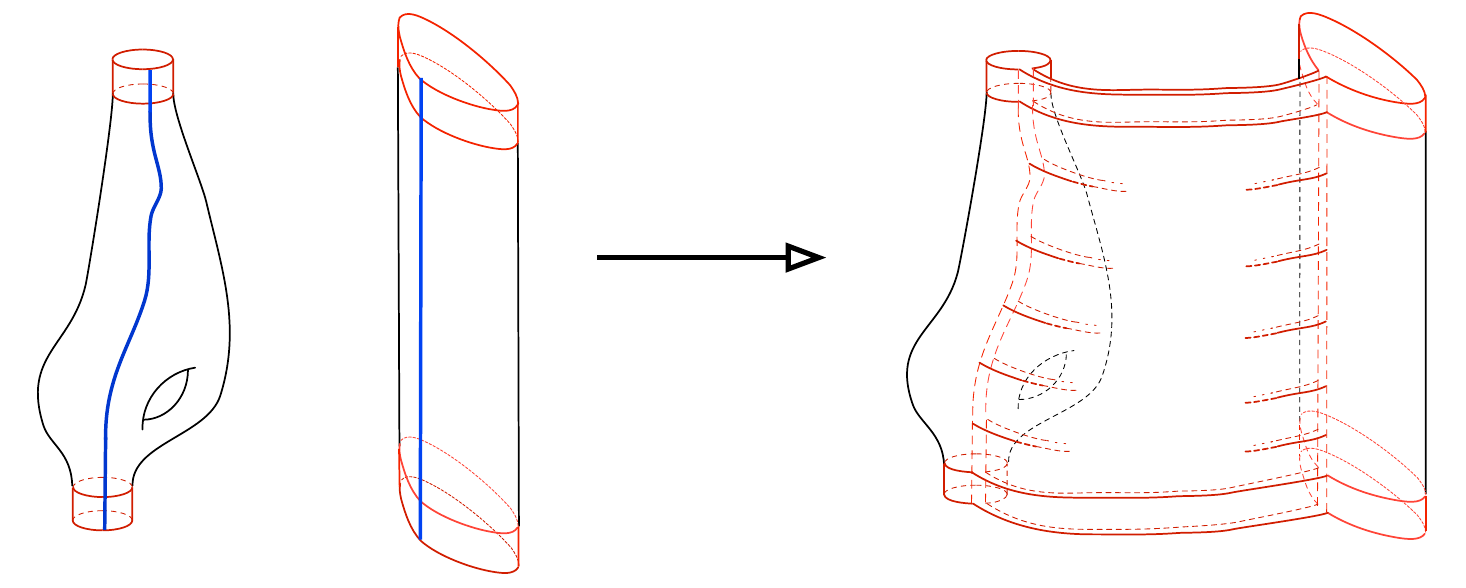}%
\end{picture}
\setlength{\unitlength}{3947sp}%
\begin{picture}(5079,1559)(1902,-7227)
\put(1900,-4700){\makebox(0,0)[lb]{\smash{{\SetFigFont{10}{8}{\rmdefault}{\mddefault}{\updefault}{\color[rgb]{1,0,0}$S_{i1}^{4n-1}$}
}}}}
\put(1700,-7000){\makebox(0,0)[lb]{\smash{{\SetFigFont{10}{8}{\rmdefault}{\mddefault}{\updefault}{\color[rgb]{1,0,0}$S_{i0}^{4n-1}$}
}}}}
\put(2200,-7300){\makebox(0,0)[lb]{\smash{{\SetFigFont{10}{8}{\rmdefault}{\mddefault}{\updefault}{\color[rgb]{0,0,0}$W_{q_i}$}
}}}}
\put(2550,-6000){\makebox(0,0)[lb]{\smash{{\SetFigFont{10}{8}{\rmdefault}{\mddefault}{\updefault}{\color[rgb]{0,0,1}$\gamma_{i}$}
}}}}
\put(3900,-6000){\makebox(0,0)[lb]{\smash{{\SetFigFont{10}{8}{\rmdefault}{\mddefault}{\updefault}{\color[rgb]{0,0,1}$\gamma$}
}}}}
\put(3900,-7300){\makebox(0,0)[lb]{\smash{{\SetFigFont{10}{8}{\rmdefault}{\mddefault}{\updefault}{\color[rgb]{0,0,0}$M\times I$}
}}}}
\put(6350,-7300){\makebox(0,0)[lb]{\smash{{\SetFigFont{10}{8}{\rmdefault}{\mddefault}{\updefault}{\color[rgb]{0,0,0}$Z_{q_i}:=W_{q_i}\# (M\times I)$}
}}}}

\end{picture}%
\caption{Constructing the manifold $Z_{q_i}$ as a slicewise connected sum (with respect paths $\gamma_i$ and $\gamma$) of $X_{q_i}$ and $M\times I$  }
\label{SWconnsum}
\end{figure}
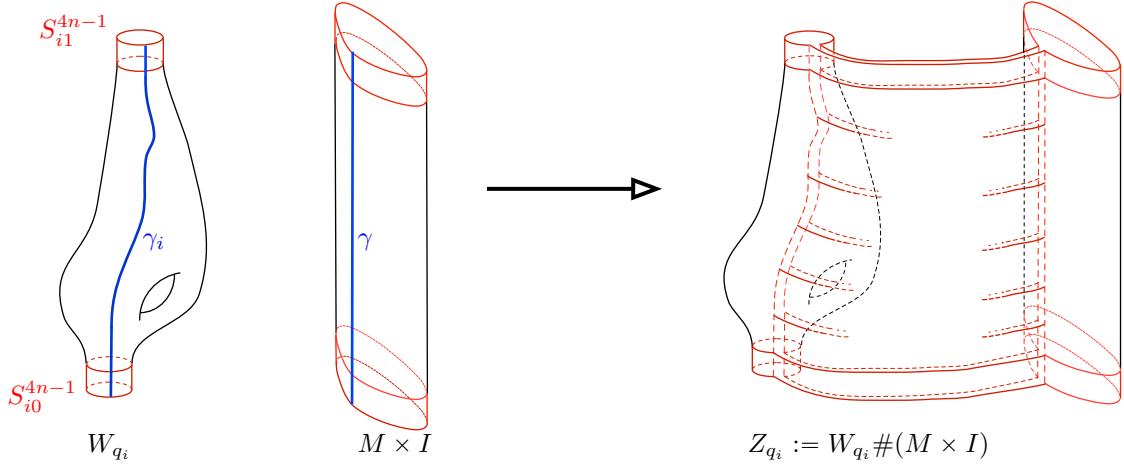
 
By Theorem \ref{A} (with respect to the collars $\tau_{ij}$), we have $s_{p,n}>0$ metrics, $\bar{g}_{q_i}$ on $W_{q_i}$ for each $i=0,1$. Each metric, $\bar{g}_{q_i}$, has a product structure near the boundary and, we assume, restricts as $ds_{4n-1}^{2}$ near $S_{i0}^{4n-1}$ and $g_{q_i}+dt^{2}$ near $S_{i1}^{4n-1}$. Employing the technique of Theorem \ref{A} over the slicewise connected sum leads to a metric, $\bar{g}_{Z_{q_i}}$, on $Z_{q_i}$ which has the following properties: 
\begin{itemize}
\item The metric $\bar{g}_{Z_{q_i}}$ is $s_{p,4n}>0$ and has a product structure near the boundary, 
\item The metric $\bar{g}_{Z_{q_i}}$ takes the form, $ds_{4n-1}^{2}\# g_{M}$, on the $S_{i0}^{4n-1}\# M\times\{0\}\cong M$ boundary component. 
\end{itemize}
The first of the properties is ensured by the earlier compatibility conditions on the paths, $\gamma_i$, which ensure the connected sum construction is constant near the boundary.

Thus, we can form the Riemannian manifold with boundary, $$(Z_{q_0, q_1}, \bar{g}_{Z_{q_0, q_1}}):=(Z_{q_0}, \bar{g}_{Z_{q_0}})\cup (Z_{q_1}, \bar{g}_{Z_{q_1}}),$$
by gluing together the $S_{00}^{4n-1}$ and $S_{10}^{4n-1}$ boundary components where both metrics agree. By construction this metric is $s_{p,4n}>0$, has a product structure near its boundary and restricts respectively on its two spherical boundary components, $S_{01}^{4n-1}$ and $S_{11}^{4n-1}$, as $g_{q_0}$ and $g_{q_1}$. As before, if $g_{q_0}$ and $g_{q_1}$ are $(s_{p, 4n-1}>0)$-isotopic, there exists an $(s_{p, 4n-1}>0)$-concordance on the cylinder, $(S^{4n-1}\times I, \bar{h}_{q_0, q_1})$. Attaching this cylinder to $Z_{q_0, q_1}$, so that each boundary of the cylinder is attached to one of the boundary components of $Z_{q_0, q_1}$,  gives rise to a closed, $s_{p,4n}>0$ Riemannian manifold: 
$$(Y_{q_0, q_1}, \bar{g}_{Y_{q_0, q_1}}):=(Z_{q_0, q_1}, \bar{g}_{Z_{q_0, q_1}})\cup (S^{4n-1}\times I, \bar{h}_{q_0, q_1}).$$
Now $Y_{q_0, q_1}$ is easily seen to be diffeomorphic to the connected sum $X_{q_0, q_1}\#  (M\times S^{1})$. The additive property of the $\alpha$-invariant over connected sums implies that:
$$ \alpha(Y_{q_0, q_1})=\alpha(X_{q_0, q_1}\#  (M\times S^{1}))=\alpha(X_{q_0, q_1})+\alpha(M\times S^1)=\alpha(X_{q_0, q_1})+0\neq 0.$$ The summand, $\alpha(M \times S^{1})$, vanishes since $M$ admits a psc-metric and hence, so does $M\times S^{1}$.
Thus, $Y_{q_0, q_1}$ admits no metric of positive scalar curvature and so we have a contradiction. Therefore, the metrics $g_{M}\# g_{q_0}$ and $g_{M}\# g_{q_1}$ lie in different path components of $\Riem^{s_{p,4n-1}>0}(M)$.
\end{proof}

\end{document}